\documentclass[letterpaper, reqno,11pt]{article}
\usepackage[margin=1.0in]{geometry}
\usepackage{amsmath,amssymb,amsthm}
\usepackage{graphicx}
\graphicspath{{images_cropped/} }

\usepackage[percent]{overpic}

\usepackage{color}

\newtheorem{thm}{Theorem}[section]
\newtheorem{lem}{Lemma}[section]
\newtheorem{defn}{Definition}[section]

\newtheorem{prop}{Proposition}[section]
\newtheorem{cor}{Corollary}[section]

\newtheorem*{killingHeisenbergPropEnv}{Proposition \ref{killingHeisenbergProp}}
\newtheorem*{killingSl2PropEnv}{Proposition \ref{killingSl2Prop}}

\theoremstyle{remark}
\newtheorem{rem}{Remark}[section]
\numberwithin{equation}{section}

\newcommand{\eps}{\varepsilon}

\newcommand{\RR}{\mathbb{R}}

\newcommand{\CC}{\mathbb{C}}

\newcommand{\ZZ}{\mathbb{Z}}
\newcommand{\tubes}{\mathbb{T}}
\newcommand{\FP}{\mathbb{F}_p}
\newcommand{\dist}{\operatorname{dist}}
\newcommand{\pts}{\mathcal{P}}
\newcommand{\lines}{\mathcal{L}}

\renewcommand{\skew}{\operatorname{skew}}
 
 \newcommand{\itemizeEqnVSpacing}{\rule{0pt}{1pt}\vspace*{-12pt}}

\newcommand{\tube}{\mathcal{T}}

\begin{document}
 \title{An improved bound on the Hausdorff dimension of Besicovitch sets in $\mathbb{R}^3$}
\author{Nets Hawk Katz\thanks{California Institute of Technology, Pasadena CA, supported by NSF grants DMS 1266104 and DMS 1565904},\and Joshua Zahl\thanks{University of British Columbia, Vancouver BC, supported by an NSERC Discovery grant}}

\maketitle
\begin{abstract}
We prove that any Besicovitch set in $\mathbb{R}^3$ must have Hausdorff dimension at least $5/2+\epsilon_0$ for some small constant $\epsilon_0>0$. This follows from a more general result about the volume of unions of tubes that satisfy the Wolff axioms. Our proof grapples with a new ``almost counter example'' to the Kakeya conjecture, which we call the $SL_2$ example; this object resembles a Besicovitch set that has Minkowski dimension 3 but Hausdorff dimension $5/2$. We believe this example may be an interesting object for future study.
\end{abstract}
\section{Introduction}\label{introSection}
A Besicovitch set is a compact set $X\subset\RR^n$ that contains a unit line segment pointing in every direction. In this paper, we will prove the following theorem:
\begin{thm}\label{KakeyaHausdorffTheorem}
Every Besicovitch set in $\RR^3$ has Hausdorff dimension at least $5/2+\epsilon_0$, where $\epsilon_0>0$ is a small absolute constant.
\end{thm}
Theorem \ref{KakeyaHausdorffTheorem} is a small improvement over a previous result of Wolff \cite{wolff}, who proved a version of Theorem \ref{KakeyaHausdorffTheorem} with $\epsilon_0=0$. Katz, \L{}aba, and Tao \cite{KLT} also proved a version of Theorem \ref{KakeyaHausdorffTheorem} where ``Hausdorff dimension'' was replaced by upper Minkowski dimension\footnote{We do not know whether the constant $\eps_0$ in Theorem \ref{KakeyaHausdorffTheorem} is larger or smaller than the corresponding constant in \cite{KLT}}.  These results and related work will be discussed further in Section \ref{previousWorkSection}.

Given a number $0<\delta<1$, we say a set $\tube\subset\RR^3$ is a $\delta$--tube (or just a tube) if $\tube$ is the $\delta$--neighborhood of a unit line segment.

\begin{defn}\label{defnTubeWolffAxioms}
We say that a set $\tubes$ of tubes satisfies the Wolff axioms if:
\begin{itemize}
\item Each $\tube\in\tubes$ is contained in the unit ball in $\RR^3$.
\item If $\delta\leq s,t\leq 1$, then at most $st\delta^{-2}$ tubes from $\tubes$ are contained in any rectangular prism of dimensions $2\times  s \times t$ (the prism need not be aligned with the coordinate axes).
\end{itemize}
\end{defn}

Note that the second condition implies that $|\tubes|\lesssim\delta^{-2}$. Theorem \ref{KakeyaHausdorffTheorem} is a corollary of the following result.

\begin{thm}\label{mainThm}
There exist positive constants $C$ (large) and $c>0,\ \eps_0>0$ (small) so that the following holds. Let $\delta>0,\ \delta\leq\lambda\leq 1$, and let $\tubes$ be a set of tubes that satisfy the Wolff axioms. For each $\tube\in\tubes,$ let $Y(\tube)\subset \tube$ and suppose that $\sum_{\tube\in\tubes}|Y(\tube)|\geq\lambda$. Then
\begin{equation}\label{volumeBdEqn}
\Big|\bigcup_{\tube\in\tubes}Y(\tube)\Big|\geq c\lambda^C \delta^{1/2-\eps_0}.
\end{equation}
\end{thm}

\subsection{Previous work}\label{previousWorkSection}
It is conjectured that every Besicovitch set in $\RR^n$ must have Hausdorff dimension $n$. The case $n=2$ was solved (in the affirmative) by Davies \cite{Davies}. For $n\geq 3$, the problem remains open. We will only discuss progress on this conjecture in $\RR^3$. A broader and more detailed survey can be found in the survey articles by Wolff \cite{wolff3} and by Katz-Tao \cite{KT2}.

There are now a number of (non-equivalent) conjectures that all fall under the umbrella of the ``Kakeya conjecture.'' For a given number $0<d\leq 3$, each of the estimates below implies the estimates preceding it.
\begin{itemize}
\item \textbf{Upper  Minkowski dimension estimate:} Any Besicovitch set $X\subset\RR^3$ must have upper Minkowski dimension at least $d$.
    \item \textbf{Lower Minkowski dimension estimate:} Any Besicovitch set $X\subset\RR^3$ must have lower Minkowski dimension at least $d$.
\item \textbf{Hausdorff dimension estimate:} Any Besicovitch set $X\subset\RR^3$ must have Hausdorff dimension at least $d$.
\item \textbf{Maximal function estimate:} Let $\tubes$ be a set of direction-separated $\delta$-tubes in $\RR^3$. Then for every $\epsilon>0$, there is a constant $C_\epsilon$ so that
 \begin{equation}\label{maximalFnBound}
 \Big\Vert\sum_{\tube\in\tubes}\chi_\tube\Big\Vert_{d^\prime}\leq C_\epsilon \delta^{1-3/d-\epsilon}.
 \end{equation}
\item \textbf{X-ray estimate:} There is a $\beta>0$ with the following property. Let $\tubes$ be a set of $\delta$--tubes in $\RR^3$. Suppose that no tube is contained in the two-fold dilate of any other tube. Suppose furthermore that for each direction $e\in S^2$, at most $m$ tubes point in a direction that is $\delta$--close to $e$. Then for every $\epsilon>0$, there is a constant $C_\epsilon$ so that
    \begin{equation}
    \Big\Vert\sum_{\tube\in\tubes}\chi_\tube\Big\Vert_{d^\prime}\leq C_\epsilon \delta^{1-3/d-\epsilon}m^{1-\beta}.
    \end{equation}
\end{itemize}

The Kakeya conjecture in $\RR^3$ asserts that all of the above estimates hold with $d=3$. We will say that a Minkowski dimension, Hausdorff dimension, etc.~estimate holds in dimension $d$ if the corresponding statement above has been established. Davies's estimate \cite{Davies} immediately implies that any Besicovitch set in $\RR^3$ must have Hausdorff dimension at least 2. A more elaborate argument by Cordoba established a maximal function estimate in dimension $d=2$.

In \cite{Bourgain3}, Bourgain showed that every Besicovitch set in $\RR^3$ has Hausdorff dimension at least $d=7/3$.
 In \cite{wolff}, Wolff established a maximal function estimate in dimension $d=5/2$, and then in \cite{wolff2} Wolff established a X-ray estimate in dimension $d=5/2$.

As we will discuss further in Section \ref{enemiesSection}, establishing Kakeya estimates above dimension $d=5/2$ is difficult because there are sets that closely resemble Besicovitch sets, but which have dimension 5/2. Thus any proof of a Kakeya estimate above dimension $d=5/2$ must grapple with these examples.

In \cite{KLT}, Katz, \L{}aba, and Tao proved that every Besicovitch set in $\RR^3$ must have upper Minkowski dimension at least $5/2+\epsilon_0$ for a small constant $\epsilon_0$. To do this, they extensively studied the structure of a (hypothetical) Besicovitch set with upper Minkowski dimension near $5/2$. They showed that such a Besicovitch set must have several properties, which they called ``planiness,'' ``graininess,'' and ``stickiness.'' In brief, planiness asserts that the tubes in a Besicovitch set passing through a typical point lie in a small neighborhood of a plane. Graininess asserts that as one moves the point in question, the corresponding planes change in a controlled way. Finally, stickiness asserts that if $\tubes$ is a set of direction-separated $\delta$--tubes whose union is a Besicovitch set with small volume, then the map $E\colon S^2\to\tubes$ which sends a direction $e\in S^2$ to the corresponding tube in $\tubes$ satisfies a weak Lipschitz-continuity type property.

After Katz, \L{}aba, and Tao's result, a natural question arises: are the properties planiness, graininess, and stickiness fundamental to Besicovitch sets, or are they artifacts of the proof techniques used in \cite{KLT}? In \cite{BCT}, Bennett, Carbery, and Tao established the multilinear Kakeya theorem. One of the implications of this theorem is that any Besicovitch set in $\RR^n$ with dimension less than $n$ must be plany. In \cite{Guth2}, Guth showed that any Besicovitch set in $\RR^3$ with dimension less than 3 must be grainy and similar ideas can be used to establish graininess-like properties in higher dimensions. Thus the properties of planiness and graininess are not merely artifacts of the Katz-\L{}aba-Tao proof; they are fundamental features of (hypothetical) Besicovitch sets that violate the Kakeya conjecture.

What about stickiness? It is not known whether every Besicovitch set in $\RR^n$ (or $\RR^3$) with dimension less than $n$ must be sticky. In a blog post \cite{T}, Tao gave a heuristic argument why a sticky Besicovitch set in $\RR^3$ with minimal dimension strictly less than $3$ is impossible. Thus if one could show that every Besicovitch set in $\RR^3$ with dimension less than three must be sticky, then this would be a promising step towards resolving the Kakeya conjecture in $\RR^3$. However, in proving Theorem \ref{mainThm} we are forced to deal with an object that resembles a non-sticky Besicovitch set. Informally, this means that unlike being plany and grainy, Besicovitch sets in $\RR^3$ with dimension less than three do not need to be sticky\footnote{Of course if the Kakeya conjecture is true then Besicovitch sets in $\RR^3$ with dimension less than three do not exist. Another interpretation of our result is that in order to prove the Kakeya conjecture in $\RR^3$, it will likely not be possible to first prove that every Besicovitch set of dimension less than three is sticky.}.

\subsection{Enemies old and new}\label{enemiesSection}
\subsubsection{The Heisenberg group}\label{HGroupEnemySec}One of the reasons that it is difficult to strengthen Wolff's result from \cite{wolff} is that the Heisenberg group is an ``almost counter-example'' to the Kakeya conjecture. the Heisenberg group is the set
\begin{equation}\label{defnOfH}
\mathbb{H}=\{(x,y,z)\in\CC^3\colon \operatorname{Im}(z)=\operatorname{Im}(x\bar y)\}.
\end{equation}
$\mathbb{H}$ shares many properties in common with a Besicovitch set. It is a subset of $\CC^3$ rather than $\RR^3$, but it contains a two (complex) dimensional family of lines. Indeed, for every $a,b\in\RR$ and $w\in\CC$, the complex line
\begin{equation*}
L_{a,b,w}=\{(s,w+as,s\bar w+b)\colon s\in\CC\}
\end{equation*}
is contained in $\mathbb{H}$. Furthermore, if we restrict our attention to $\mathbb{H}\cap B(0,1)$ and cover each point of $\mathbb{H}\cap B(0,1)$ by a ball of radius $\delta$, we obtain a set of complex $\delta$--tubes (the intersection of a unit ball with the $\delta$--neighborhood of a complex line in $\CC^3$), and these tubes satisfy the natural analogue of the Wolff axioms.

However, the Heisenberg group differs from a genuine Besicovitch set in two key respects. First, while the Heisenberg group contains a two (complex) dimensional family of lines that satisfy a natural analogue of the Wolff axioms, these lines do not all point in different directions. Instead, there is a $3/2$--dimensional family of directions, and there is a half dimensional family of lines pointing in each of these directions. In \cite{KLT}, Katz, \L{}aba, and Tao exploit the fact that the tubes in a Besicovitch set all point in different directions to obtain their improved estimate; this is how their proof distinguishes between the Heisenberg group and a genuine Besicovitch set.

The second difference between the Heisenberg group and a genuine Besicovitch set is that the Heisenberg group is a subset of $\CC^3$, while a Besicovitch set is a subset of $\RR^3$. This distinction is crucial. Observe from \eqref{defnOfH} that the definition of $\mathbb{H}$ makes use of the complex conjugation map $z\mapsto\bar z$, and the existence of this map is closely related to the fact that $\CC$ contains a half-dimensional subfield.  $\RR$, however does not contain a half-dimensional subfield. This observation (or rather a quantitative version of it) is a key component of the proof of Theorem \ref{mainThm}; this is how our proof distinguishes between the Heisenberg group and a genuine Besicovitch set.
\subsubsection{The $SL_2$ example}
In proving Theorem \ref{mainThm} we encountered a new difficulty that has not appeared before. We call this problem the $SL_2$ example. The $SL_2$ example is a (hypothetical) set of $\delta^{-2}$ tubes that satisfy the Wolff axioms. The union of these tubes has volume $\delta^{1/2}$, but the union of the $\delta^{1/2}$ neighborhoods of these tubes has volume $\sim 1$. Thus the $SL_2$ example is an almost counter-example to the Hausdorff version of the Kakeya conjecture in $\RR^3$, but it is not an almost counter-example to the upper Minkowski dimension version of the conjecture. Furthermore, the $SL_2$ example is not sticky, in the sense described in Section \ref{previousWorkSection} above.

Of course, Theorem \ref{mainThm} asserts that the $SL_2$ example cannot actually exist in $\RR^3$. However, it is possible to construct the $SL_2$ example in a slightly different setting. Let $R$ be the ring $\FP[t]/(t^2)$. Each number $x\in[0,1]\subset\RR$ can be written as $x = \delta^{1/2}x_1 + \delta x_2+O(\delta)$, where $x_1$ and $x_2$ are integers between $0$ and $\lfloor\delta^{-1/2}\rfloor$. The ring $R$ is meant to model this decomposition of the interval $[0,1]$. Elements of $\FP\subset \FP[t]/(t^2)$ represent the coarse (i.e.~$\delta^{1/2}$) scale, while elements of $t\cdot\FP$ represent the fine (i.e.~$\delta$) scale.

Define
\begin{equation}\label{defnOfSL2}
X = \{(x_1+x_2t,\ y_1+y_2t,\ z_1+z_2t)\in R^3\colon z_2 = x_1y_2 - x_2y_1\}.
\end{equation}
We have that $|X|=p^5 = |R|^{5/2}$. Observe that the definition of $X$ looks nearly identical to the definition of $\mathbb{H}$ from \eqref{defnOfH}. However, since the rings $\CC$ and $R$ have dramatically different properties, the resulting sets $\mathbb{H}$ and $X$ will differ markedly as well.

We will consider subsets of $R^3$ of the form $\{ (a,b,0)+s(c,d,1)\colon s\in R \}$, where $a,b,c,d\in R$. These are the analogue of lines in $\RR^3$ that are not parallel to the $xy$ plane. Each of these ``lines'' can be identified with the point $(a,b,c,d)\in R^4$.

With this identification, define
$$
\mathcal{L}=\{(a+\alpha at,\ b+\alpha bt,\ c+\alpha ct,\ d+\alpha dt)\in R^4\colon a,b,c,d,\alpha\in\FP,\ ad-bc=1\}.
$$

$\mathcal{L}$ is a set of $p^4=|M|^2$ lines, each of which is contained in $X$. The lines in $\mathcal{L}$ satisfy an analogue of the Wolff axioms. However, the lines in $\mathcal{L}$ are not ``sticky,'' and the analogue of the upper (and lower) Minkowski dimension of $X$ is large: if $\pi\colon R^3\to\FP^3$ is the projection $(x_1+x_2t, y_1+y_2t, z_1+z_2t)\mapsto(x_1,y_1,z_1)$, then
\begin{equation}\label{largeAtCoarseScales}
\pi(X)=\FP^3.
\end{equation}
If the $SL_2$ example were defined over $\RR^3$ rather than $R^3$, then \eqref{largeAtCoarseScales} says that the $\delta^{1/2}$--neighborhood of the union of lines in the $SL_2$ example has volume approximately one. This is dramatically different from the Heisenberg group example, in which the union would have volume approximately $\delta^{1/4}$. The $SL_2$ example will be discussed further in Appendix \ref{SL2Appendix}.

The ring $R$ has an ideal $t\FP\subset R$ of nilpotent elements, and this ideal plays a crucial role in the construction of the $SL_2$ example. The reals do not contain such an ideal, and this fact is (implicitly) exploited when showing no analogue of the $SL_2$ example can exist in $\RR^3.$ 

\subsubsection{The Regulus map}
To prove Theorem \ref{mainThm}, we will show that any hypothetical counter-example to the theorem must resemble either the Heisenberg or $SL_2$ examples
described above. To do this, we will consider an object called the ``regulus map,'' (defined in Section \ref{regMapSection}), which describes how much algebraic structure the counter-example possesses at coarse scales. At one extreme is the $SL_2$ example: at coarse scales, the tubes in the $SL_2$ example lie close to an algebraic variety in the parameter space of lines (namely $\{(a,b,c,d)\in\RR^4\colon ad-bc=1\}$). At the opposite extreme is the Heisenberg example: at coarse scales (indeed at all scales), few tubes in the Heisenberg example lie close to any low degree algebraic variety in the parameter space of lines.

Unfortunately, a hypothetical counter-example to Theorem \ref{mainThm} may lie anywhere between these two extremes. Quantifying this intermediate behavior is one of the major technical difficulties we encounter in this paper.

\subsection{Notation and and epsilon management}
Throughout this paper, $\delta$ will denote a small positive number. Unless noted otherwise, all constants will be independent of $\delta$.

The goal of this paper is to prove Theorem \ref{mainThm} for some absolute constant $\eps_0>0$. To do this, we will prove a series of statements that all depend on a small parameter $\eps>0$. Eventually we will show that if $\eps$ is sufficiently small then we arrive at a contradiction. Thus the variable $\eps$ will denote a small positive number (whose meaning may differ in different statements) that is always larger than some absolute constant $\eps_0>0$ that will be determined at the end of the paper. In particular, it will always be the case that $\delta^{\eps}\leq \delta^{\eps_0}\leq c|\log\delta|^{-1}$ for some absolute constant $c>0$.

We write $A\lessapprox_{\eps}B$ to mean that there exists a constant $C$ (independent of $\delta$) so that $A\leq \delta^{-C\eps}B$. The constant $C$ may vary from line to line. If $A\lessapprox_{\eps}B$ and $B\lessapprox_{\eps}A$, then we write $A\approx_{\eps}B$.

To avoid keeping track of many different constants, we will often use the following sort of notation:

\noindent Lemma A: Suppose that $P \leq \delta^{-\eps}Q$. Then $R\lessapprox_{\eps}S$.\\
\noindent Lemma B: Suppose that $R \leq \delta^{-\eps}S.$ Then $T\lessapprox_{\eps}V.$

By combining Lemmas A and B, we conclude that if $P\leq\delta^{-\eps}Q$ then $T\lessapprox_{\eps}V$. This is because Lemma A asserts that $P\leq\delta^{-\eps}Q$  implies $R\leq \delta^{-C_1\eps}S$ for some constant $C_1$, while Lemma B (with $\eps$ replaced by $C_1\eps$) asserts that $R\leq\delta^{-C_1\eps}S$ implies $T\leq \delta^{-C_2C_1\eps}V$ for some constant $C_2$. The latter statement is equivalent to $T\lessapprox_{\eps}V$. In future, we will chain together multiple lemmas that use $\lessapprox_{\eps}$ notation without further comment.

If $A$ is a set and $t>0$, we will use $N_{t}(A)$ to denote the $t$-neighborhood of $A$. For example, the statement ``$B\subset N_{\lessapprox_{\eps}t}A$'' means: there exists an absolute constant $C$ so that the set $B$ is contained in the $\delta^{-C\eps}t$-neighborhood of the set $A$.

The table below lists the notation for various objects used in the proof of Theorem \ref{mainThm}, and shows where in the paper a definition can be found.
\begin{table}[h]
\centering 
\begin{tabular}{r|l|l}
Notation & Meaning & Definition location\\
\hline
$\tube$ & $\delta$ tube & Section \ref{introSection}\\
$\tubes$ & Set of $\delta$ tubes & Section \ref{introSection}\\
$Y(\tube)$ & Shading of a tube & Section \ref{introSection}\\
$\skew(L_1,L_2)$ & Skewness of lines $L_1$ and $L_2$ & Definition \ref{defnOfSkew}\\
$R_{\tube_1,\tube_2,\tube_3}$ & Regulus generated by lines coaxial with $\tube_1,\tube_2,\tube_3$ & Definition \ref{algRegulusDefn}\\
$S$ & Regulus strip & Definition \ref{regulusStripDefn}\\
$H_Y(\tube_1,\ldots,\tube_k)$& The (joint) hairbrush of $\tube_1,\ldots,\tube_k$ & Definition \ref{hairbrushDef}\\
$R(\tube_1,\tube_2)$ & The regulus containing the joint hairbrush of $\tube_1$ and $\tube_2$ & Corollary \ref{hairbrushTwoTubesInFatRegulusOne}\\
$H^\prime(\tube_1,\tube_2)$ & The refined hairbrush of $\tube_1$ and $\tube_2$ & Definition \ref{definedHairbrushDef}.\\
$\tube_{\delta^{1/2}}(S)$ & The $\delta^{1/2}$ tube containing the regulus strip $S$ & Lemma \ref{regStripsAreTangent}\\
$\tubes(\tube_{\delta^{1/2}})$ & The $\delta$ tubes from $\tubes$ contained in $\tube_{\delta^{1/2}}$ & Section \ref{incidenceGeomRegStripsSec}\\
$C_s$ & The cone in $\RR^4$ containing the regulus strip $S$ & Section \ref{tubesStripsSection}\\
$\Sigma_S$ & The hyperplane in $\RR^4$ containing the cone $C_S$ & Section \ref{tubesStripsSection}\\
$G$ & A grain & Section \ref{discetizationScaleDelta12Sec}\\
$Q_G$ & The cube containing the grain $G$ & Section \ref{discetizationScaleDelta12Sec} 
\end{tabular}
\end{table}

%
%
%
%
\section{Reguli}\label{reguliiSection}
It is a well-known result (see e.g.~\cite{HC}) that the union of the set of lines intersecting three skew lines in $\RR^3$ forms a regulus (a doubly-ruled quadric surface). Reguli play a central role when studying the structure of Besicovitch sets near dimension $5/2$. In the next section, we will develop various quantitative statements about set of tubes.
\subsection{Quantitative skewness}

\begin{defn}[Linear cone]
Let $L$ be a line in $\RR^3$ and let $\Pi$ be a plane containing $L$. A \emph{linear cone} of angle $\alpha$ with vertex $L$ is a set of the form
$$
\{p\in\RR^3\colon \dist(p,\Pi)\leq \alpha \dist(\pi(p), L)\},
$$
where $\pi(p)$ is the orthogonal projection of $p$ to $\Pi$. For example, if $L=(0,0,0)+\RR(1,0,0)$ and $\Pi$ is the plane $\{y=0\}$, then the above set is given by $\{(x,y,z)\in\RR^3\colon |z|\leq \alpha |y|\}$.
\end{defn}

\begin{figure}[h!]
 \centering
\begin{overpic}[width=0.4\textwidth]{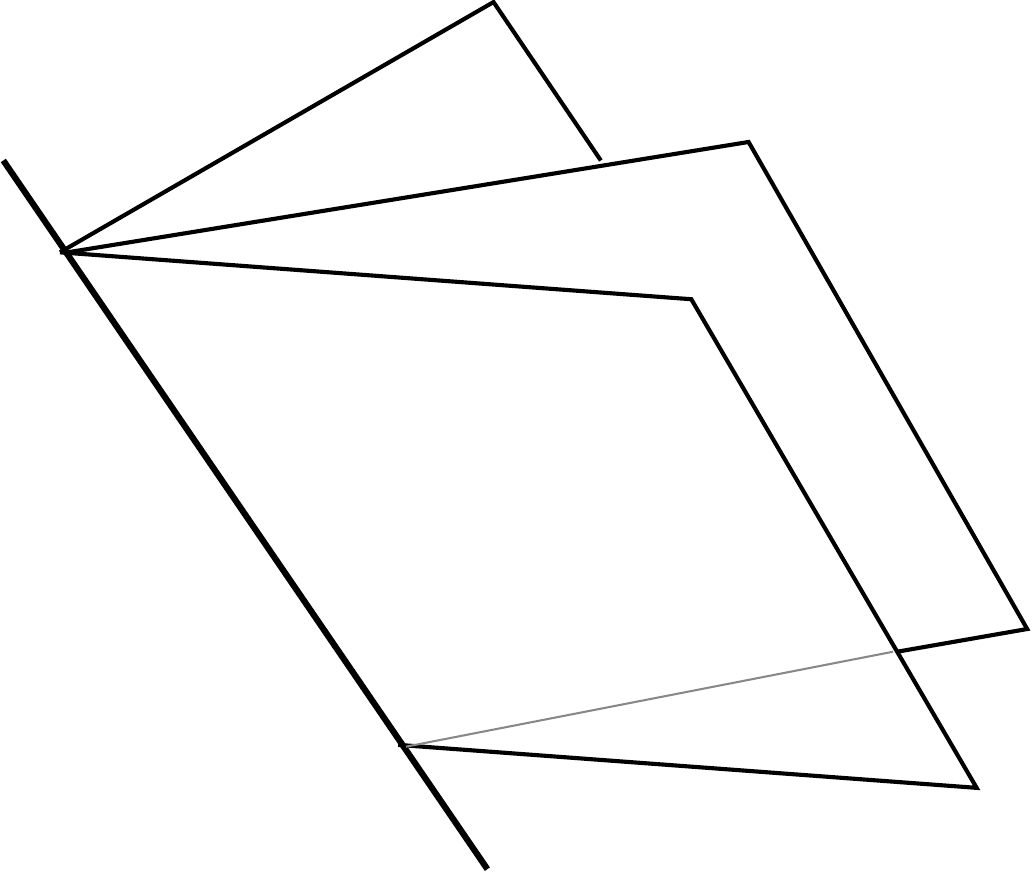}
 \put (39,2) {$L$}
  \put (68,60) {$\Pi$}
\end{overpic}
 \caption{A cone of angle $\alpha$. The angle between $\Pi$ and adjacent planes is $\arctan\alpha$. }\label{linearConePic}
\end{figure}

\begin{defn}[Quantitative skewness]\label{defnOfSkew}
Let $L_1$ and $L_2$ be two lines in $\RR^3$ that intersect the unit ball. We define $\operatorname{skew}(L_1,L_2)$ to be the minimum value of $\alpha$ so that $L_2\cap B(0,2)$ is contained in a linear cone of angle $\alpha$ with vertex $L_1$. We say that $L_1$ and $L_2$ are $\geq\delta^{\eps}$ skew if $\skew(L_1,L_2)\geq \delta^{\eps}$, and we say that $L_1$ and $L_2$ are $\approx_{\eps}1$ skew if $\skew(L_1,L_2)\approx_{\eps}1$.
\end{defn}

\begin{defn}[Quantitative separation]\label{defnOfSeparated}
Let $L_1,L_2$ be two distinct lines that intersect the unit ball. We say that $L_1,L_2$ are $t$ separated (with error $\delta^{\eps})$ if
\begin{equation}
\delta^\eps t\leq \dist(p,L_2)\leq \delta^{-\eps}t \quad\textrm{for all}\ p\in L_1\cap B(0,1).
\end{equation}
We say that $L_1$ and $L_2$ are $t$ separated (with error $\approx_{\eps}1$) if
\begin{equation}
\dist(p,L_2)\approx_{\eps} t\quad\textrm{for all}\ p\in L_1\cap B(0,1).
\end{equation}
We say that $L_1$ and $L_2$ are uniformly separated with error $\delta^{\eps}$ (resp. $\approx_{\eps}1$) if they are $t$ separated with error $\delta^{\eps}$ (resp. $\approx_{\eps}1$) for some value of $t$.
\end{defn}

\begin{figure}[h!]
 \centering
\begin{overpic}[width=0.3\textwidth]{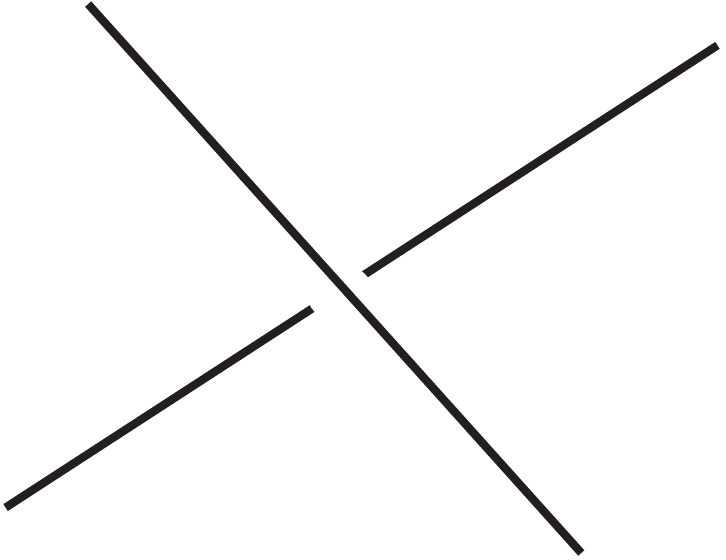}
\end{overpic}
 \caption{Two lines do not need to be quantitatively separated to be quantitatively skew.}\label{quantSepSkewPic}
\end{figure}

For brevity, we will say that two lines are $\geq\delta^{\eps}$ separated and skew if they are $\geq\delta^{\eps}$ skew and 1 separated with error $\delta^{\eps}$. We will say that two lines are $\approx_{\eps}1$ separated and skew if they are $\approx_{\eps}1$ skew and 1 separated with error $\approx_{\eps}1$.

If $\tube_1$ and $\tube_2$ are tubes, we define the skew (resp.~separation) of $\tube_1$ and $\tube_2$ to be the skew (resp.~separation) of their coaxial lines.

%
%
%
%
\subsection{Reguli and curvature}
In this section we will prove quantitative bounds on the Gauss curvature of reguli defined by certain triples of lines. Given two lines $L_1$ and $L_2$, we define $\angle(L_1,L_2)$ as
the angle between their orientations.

\begin{lem}\label{linesInDifferentPlanes}
Let $L,L_1$ and $L_2$ be lines. Suppose that $L$ makes an angle $\geq\delta^{\eps}$ with $L_1$ and $L_2$, and that $L$ intersects $L_1$ and $L_2$  inside the unit ball. Let $\Pi_i$ be the plane spanned by $L$ and $L_i$. Suppose that
$$
\delta^{\eps}\angle(\Pi_1,\Pi_2) \leq \dist(L\cap L_1,\ L\cap L_2) \leq \delta^{-\eps}\angle(\Pi_1,\Pi_2),
$$
and
\begin{equation}\label{ControlOnAngleL1L2}
\delta^{\eps} \angle(L_1,L_2) \leq \dist(L\cap L_1,\ L\cap L_2) \leq \delta^{-\eps} \angle(L_1,L_2).
\end{equation}

Then $L_1$ and $L_2$ are $\approx_{\eps} 1$ skew and are uniformly separated with error $\approx_{\eps}1$. See Figure \ref{lemma21_fig}.
\end{lem}

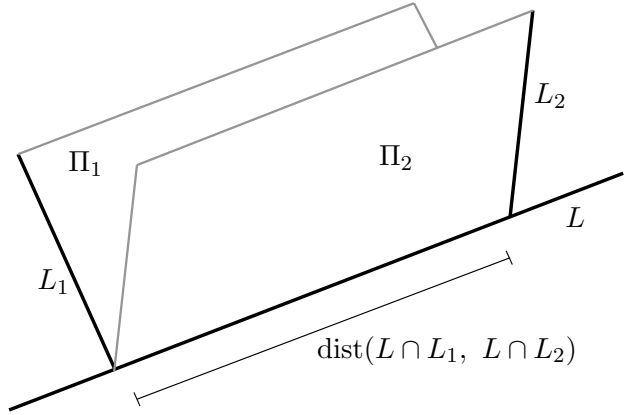
\begin{figure}[h!]
 \centering
\begin{overpic}[width=0.5\textwidth ]{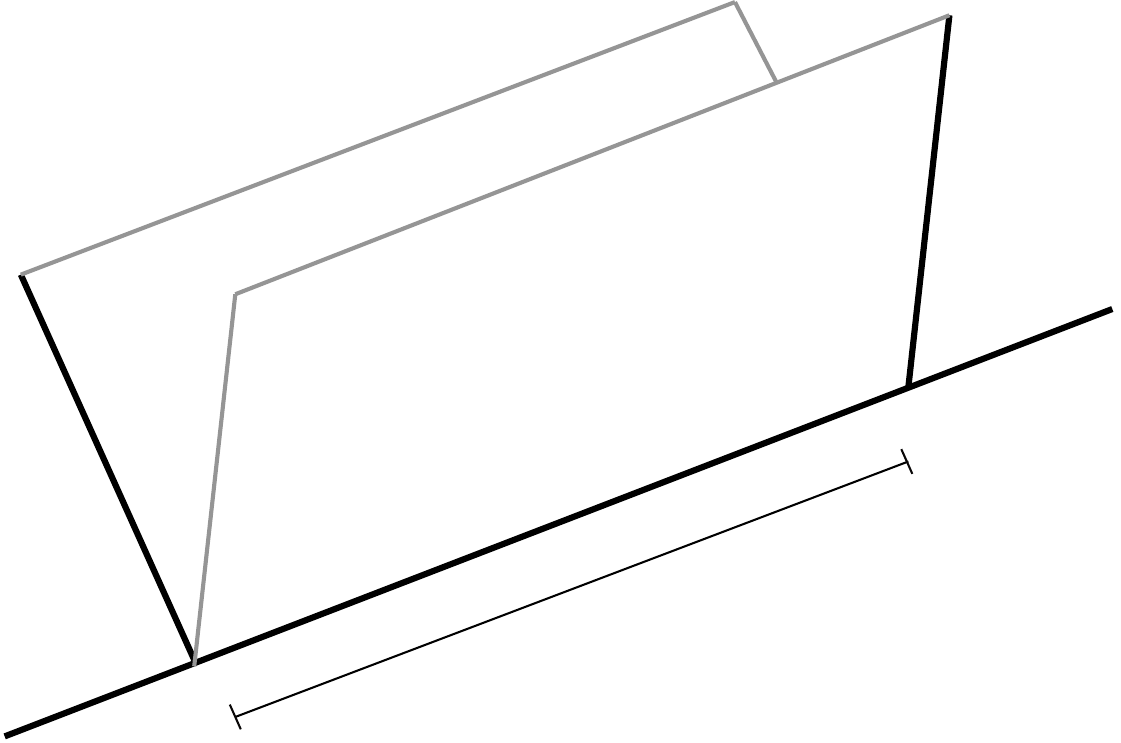}
 \put (90,30) {$L$} 
 \put (50,9) {$\dist(L\cap L_1,\ L\cap L_2)$}
  \put (5,20) {$L_1$} 
  \put (85,50) {$L_2$} 
  \put (60,40) {$\Pi_2$} 
  \put (10,39) {$\Pi_1$} 
\end{overpic}
 \caption{Illustration of Lemma \ref{linesInDifferentPlanes}.}\label{lemma21_fig} 
\end{figure}

\begin{proof} Without loss of generality, we can assume that $L\cap L_1$ is the origin. Applying a linear transformation that distorts angles by a factor of $\lessapprox_{\eps}1$, we can assume that $L_1$ is the $z$ axis and $L$ is the $x$ axis. Let $t = \dist(L\cap L_1,\ L\cap L_2)$. Then $\Pi_2$ is a plane containing the $x$ axis that makes an angle $\approx_\eps t$ with the $z$ axis, and $L_2$ is a line in this plane that makes an angle $\approx_\eps t$ with the $z$ axis; $L_2$ contains the point $(t,0,0)$. Thus we can write $L_2 = (t,0,0) + \RR(a,b,1)$, where $|a|\lessapprox_{\eps}t$ and $|b|\approx_{\eps}t$. This immediately implies that $L_1$ and $L_2$ are $t$-separated (and thus uniformly separated) with error $\approx_{\eps}1$. 

The point $(t+a,b,1)\in L_2$ has $y$-coordinate $\gtrapprox_\eps t$. Since this point has distance $\approx_\eps t$ from $L_1$, and since any linear cone with vertex $L_1$ that contains $L_2\cap B(0,2)$ must also contain the $xz$ plane, any linear cone with vertex $L_1$ that contains $L_2\cap B(0,2)$ must have angle $\gtrapprox_{\eps}1$. On the other hand, there clearly exists a linear cone of angle $\lessapprox_\eps 1$ with vertex $L_1$ that contains $L_2$; we conclude that $L_1$ and $L_2$ are $\approx_{\eps}1$ skew. 
 \end{proof}

\begin{lem}\label{speedOfSpannedPlane}
Let $L_1$ and $L_2$ be lines that intersect the unit ball. Suppose that $L_1$ and $L_2$ are uniformly separated with error $\delta^{\eps}$ and are $\geq\delta^{\eps}$ skew. For each $p\in L_1\cap B(0,1)$, let $\Pi_p$ be the plane spanned by $p$ and $L_2$. Let $p_0\in L_1\cap B(0,1)$, and define the function
$$
f\colon L_1\cap B(0,1)\to\RR,\quad p\mapsto \angle(\Pi_{p_0},\Pi_p),
$$
where $\angle(\Pi_{p_0},\Pi_p)$ denotes the signed angle between $\Pi_{p_0}$ and $\Pi_p$.

Then $f$ is continuously differentiable and has derivative $\approx_{\eps}1$ for all $p\in L_1\cap B(0,1)$. In particular,
\begin{equation}\label{angleSimDist}
\angle(\Pi_{p_1},\Pi_{p_2})\approx_{\eps} \dist(p_1,p_2)
\end{equation}
for all $p_1,p_2\in L_1\cap B(0,1)$. See Figure \ref{lemma22_fig}
\end{lem}

\begin{figure}[h!]
 \centering
\begin{overpic}[width=0.4\textwidth ]{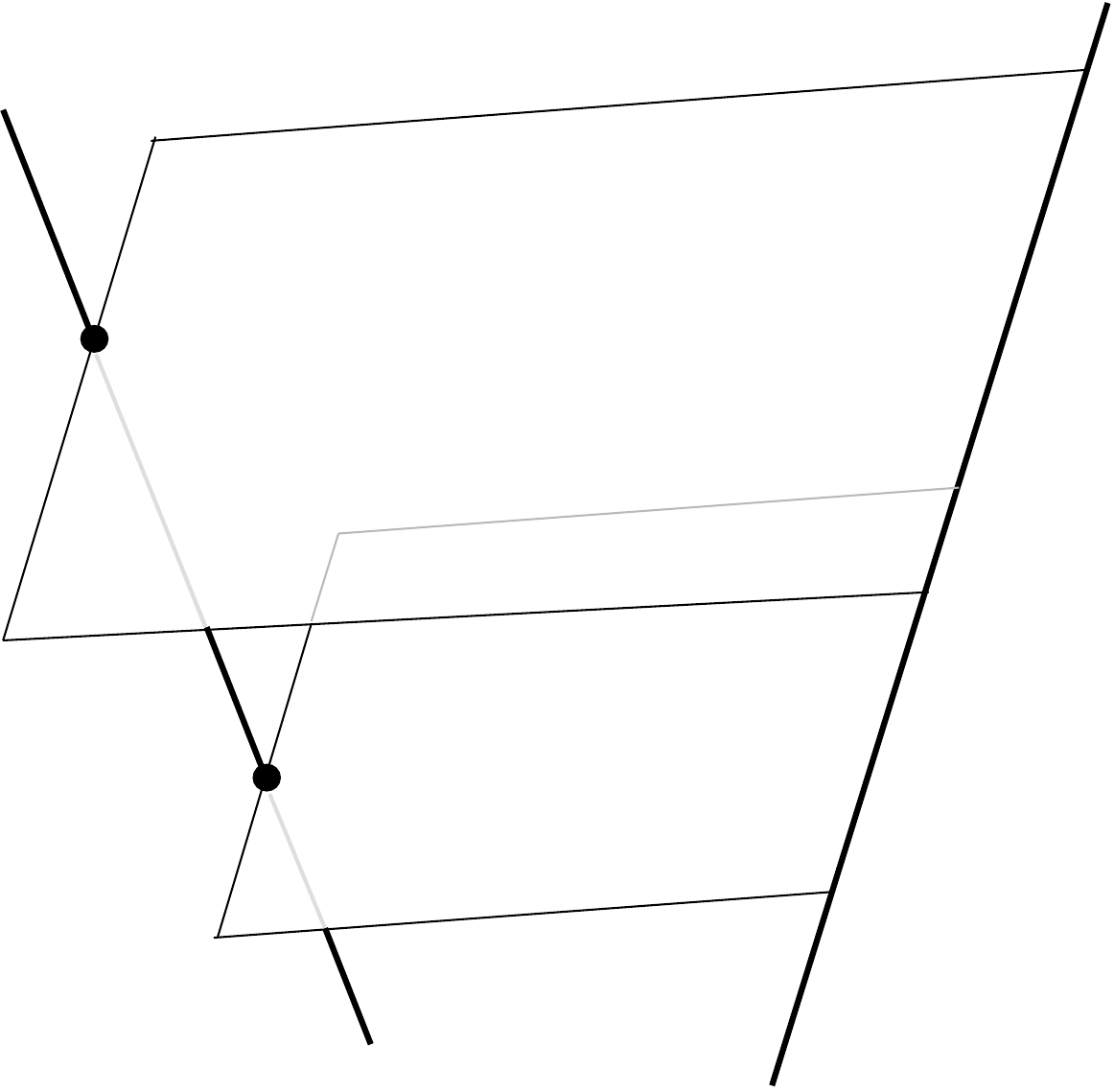}
  \put (25,5) {$L_1$} 
  \put (74,8) {$L_2$} 
  \put (16,27) {$p_0$} 
  \put (3,66) {$p$} 
  \put (50,66) {$\Pi_p$} 
  \put (50,28) {$\Pi_{p_0}$} 
\end{overpic}
 \caption{Illustration of Lemma \ref{speedOfSpannedPlane}.}\label{lemma22_fig} 
\end{figure}
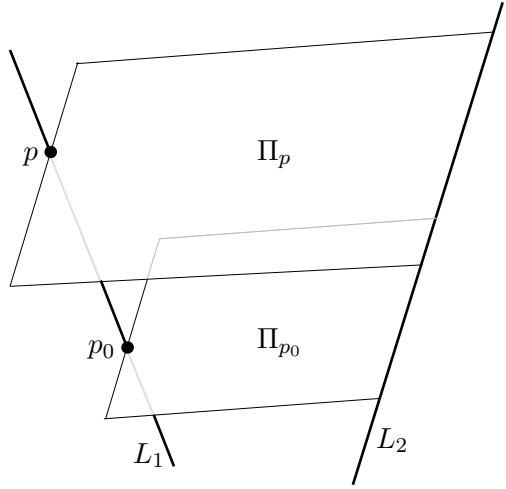

\begin{proof}
Applying a rigid transformation, we can assume that $L_2$ is the $z$ axis, $p_0 = (t,0,0)$ for some $t\in (0,1)$, and $L_1$ is the line $(t, 0, 0) + \RR(a, b, 1)$, where $|b|\approx_{\eps}t$ and $|a|\lessapprox_{\eps}t$. Then $\Pi_{p_0}$ has unit normal vector $(0,1,0)$. If $p=p(s) = (t + as, bs, s)\in L_1$, then $\Pi_p$ has normal vector $(0,0,1)\times (t+as,bs,s)=(-b s, a s + t, 0)$. Thus if $\theta(s)$ is the angle between $\Pi_{p_0}$ and $\Pi_{p(s)}$, then $\sin(\theta(s)) = \frac{|bs|}{\Vert(-b s, a s + t, 0)\Vert}$. It is now a straightforward computation to verify that the function $f$ described above is continuously differentiable and has derivative $\approx_{\eps}1$ for all $p\in L_1\cap B(0,1)$. 
\end{proof}

\begin{lem}\label{BoxOfFourLines}
Let $L_1$ and $L_2$ be lines that intersect the unit ball. Suppose they are $\geq\delta^{\eps}$ separated and skew. Let $L_1^*$ and $L_2^*$  intersect both $L_1$ and $L_2$ inside the unit ball. Suppose
$$
\delta^{\eps}\dist(L_2\cap L_1^*,\ L_2\cap L_2^*)\leq \dist(L_1\cap L_1^*,\ L_1\cap L_2^*)\leq\delta^{-\eps}\dist(L_2\cap L_1^*,\ L_2\cap L_2^*).
$$
Then $L_1^*$ and $L_2^*$ are uniformly separated with error $\approx_{\eps}1$ and $\approx_{\eps} 1$ skew. Furthermore, $\angle(L_i,L_j^*)\approx_{\eps} 1$ for each pair $i,j  \in \{1,2\}$.
\end{lem}
\begin{proof}
Note that since $L_1$ and $L_2$ are 1 separated with error $\delta^{\eps}$, both $L_1^*$ and $L_2^*$ must make an angle $\gtrapprox_{\epsilon} 1$ with each of $L_1$ and $L_2$. Let $t=\dist(L_1\cap L_1^*,\ L_1\cap L_2^*)$. Let $\Pi_1$ be the plane spanned by $L_1^*$ and $L_2$, and let $\Pi_2$ be the plane spanned by $L_2^*$ and $L_2$. Since $\Pi_1$ and $\Pi_2$ intersect $L_1$ at $\approx_\eps t$ separated points, By Lemma \ref{speedOfSpannedPlane}, we have that $\angle(\Pi_1,\Pi_2)\approx_{\eps} t$. Thus $\angle(L_1^*,\ L_2^*)\gtrapprox_{\eps}t$. On the other hand, since $L_1$ and $L_2$ are $1$ separated with error $\delta^{\eps}$, we have $\dist(L_1\cap L_1^*,\ L_2\cap L_1^*)\gtrapprox_{\eps}1$. We also have $\dist(L_1\cap L_1^*,\ L_2^*)\leq \dist(L_1\cap L_1^*,\ L_1\cap L_2^*)\approx_{\eps}t,$ and $\dist(L_2\cap L_1^*,L_2^*)\leq \dist(L_2\cap L_1^*,\ L_2\cap L_2^*)\approx_{\eps}t,$ i.e. there are two $\approx_{\eps}1$-separated points $p\in L_1$ where $\dist(p,L_2)\lessapprox_{\eps}t$. This implies that $\angle(L_1^*,\ L_2^*)\lessapprox_{\eps}t$. Combined with our previous inequality, we conclude that $\angle(L_1^*,\ L_2^*)\approx_{\eps}t$.

Thus
$$
\angle(L_1^*, L_2^*)\approx_{\eps} \angle(\Pi_1,\Pi_2) \approx_{\eps} \dist(L_1^*\cap L_2,\ L_2^*\cap L_2),
$$
so by Lemma \ref{linesInDifferentPlanes}, $L_1^*$ and $L_2^*$ are $\approx_{\eps} 1$ skew and are uniformly separated with error $\approx_{\eps}1$.
\end{proof}

\begin{lem}\label{ThreeSkewLinesTwoTransversals}
Let $L_1,L_2,$ and $L_3$ be three lines that intersect the unit ball. Suppose that all three lines are pairwise $\geq\delta^{\eps}$ skew; $L_1,L_2$ are uniformly separated with error $\delta^{\eps}$; and $L_1,L_3$ and $L_2,L_3$ are $1$ separated with error $\delta^{\eps}$.

Let $L_1^*$ and $L_2^*$ be distinct lines intersecting each of $L_1,L_2,$ and $L_3$ inside the unit ball. Then $L_1^*$ and $L_2^*$ are uniformly separated with error $\approx_{\eps}1$ and $\approx_{\eps} 1$ skew. Furthermore, $\angle(L_1,L_1^*)\approx_{\eps} 1$ and $\angle(L_1,L_2^*)\approx_{\eps} 1$ (see Figure \ref{lemma24_fig}).
\end{lem}

\begin{figure}[h!]
 \centering
\begin{overpic}[width=0.6\textwidth]{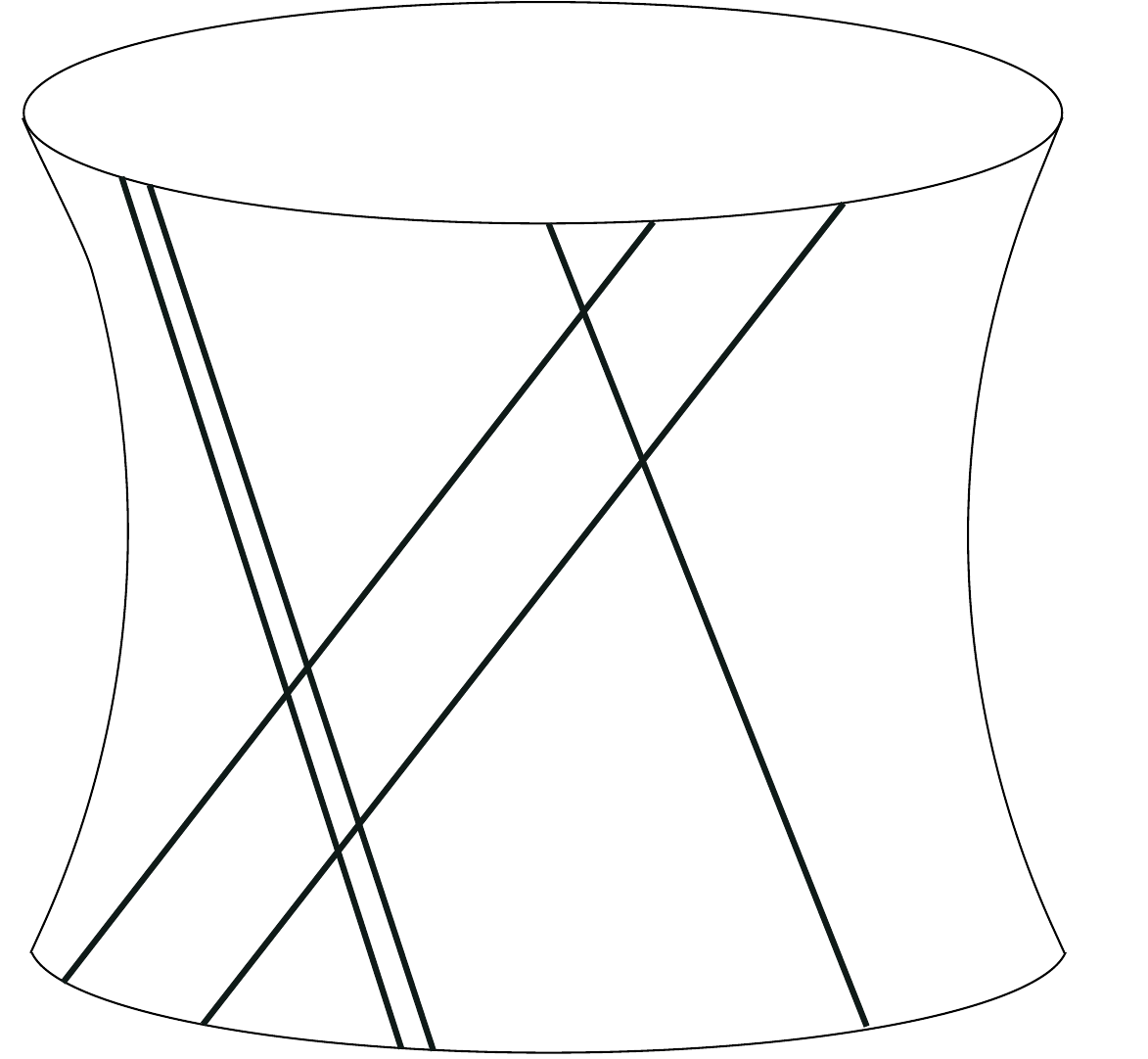}
  \put (9,8) {$L_1^*$} 
  \put (16,8) {$L_2^*$} 
  \put (28,5) {$L_1$} 
  \put (38,4) {$L_2$}
  \put (69,5) {$L_3$}
\end{overpic}
 \caption{Illustration of Lemma \ref{ThreeSkewLinesTwoTransversals}.}\label{lemma24_fig} 
\end{figure}
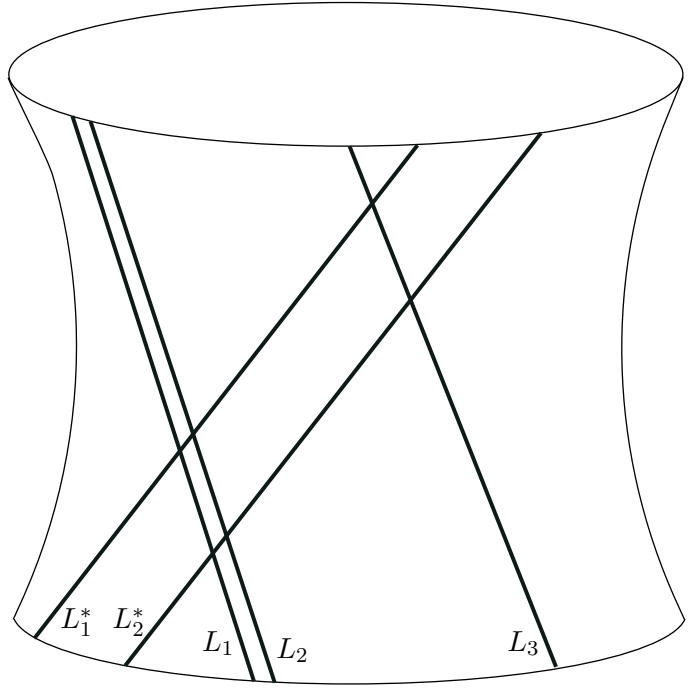

\begin{proof}
For $i=1,2$, let $p_i=L_i^*\cap L_1,$ and let $\Pi_{i}$ be the plane spanned by $p_i$ and $L_2$. Let $t=\dist(p_1,p_2)$. By Lemma \ref{speedOfSpannedPlane} applied to $L_1,L_2$, $p_1$, and $p_2$, we have $\angle(\Pi_1,\ \Pi_2)\approx_{\eps} t$. Next, observe that $\Pi_i\cap L_3 = L_i^*\cap L_3$; in particular, these intersections occur within the unit ball. Since $L_3$ is $\approx_{\eps}1$ skew to $L_1$, we have that $\angle(\Pi_1,L_3)\approx_{\eps}1$ and $\angle(\Pi_2,L_3)\approx_{\eps}1$. This means that $\operatorname{dist}(\Pi_1\cap L_3,\Pi_2\cap L_3)\approx_\eps\angle(\Pi_1,\Pi_2)\approx_{\eps}t$, i.e. $\operatorname{dist}(L_1^*\cap L_3,\ L_2^*\cap L_3)\approx \operatorname{dist}(L_1^*\cap L_1,\ L_2^*\cap L_1)$. Thus by Lemma \ref{BoxOfFourLines} (applied to $L_1,L_3,L_1^*,L_2^*$), we conclude that $L_1^*$ and $L_2^*$ are uniformly separated with error $\approx_{\eps}1$ and $\approx_{\eps} 1$ skew, and that $\angle(L_1,L_1^*)\approx_{\eps} 1$ and $\angle(L_1,L_2^*)\approx_{\eps} 1$.
\end{proof}

\begin{lem}\label{threeSkewLinesGetYouAnywhere}
Let $L_1,L_2,$ and $L_3$ be three lines that intersect the unit ball. Suppose that all three lines are pairwise $\geq\delta^{\eps}$ skew; $L_1,L_2$ are uniformly separated with error $\delta^{\eps}$; and $L_1,L_3$ and $L_2,L_3$ are $1$ separated with error $\delta^{\eps}$.

Let $R$ be the regulus containing $L_1,L_2,$ and $L_3$. Then for any $p\in R\cap B(0,1)$, there are lines $L_1^*,L_2^*,L_3^*$ and $L$ with the following properties
\begin{itemize}
\item $L_1^*,L_2^*,L_3^*,$ and $L$ are contained in $R$.
\item $L$ intersects each of $L_1^*,L_2^*,$ and $L_3^*$, and the points of intersection occur in $B(0, \delta^{-C\eps})$, where $C$ is an absolute constant.
\item $L\cap L_1^*=p$.
\item $\angle(L,L_1^*)\approx_{\eps} 1$.
\item $L_1^*, L_2^*,$ and $L_3^*$ are pairwise $\approx_{\eps}1$ separated and skew.
\end{itemize}
See figure \ref{lemma25_fig}.
\end{lem}
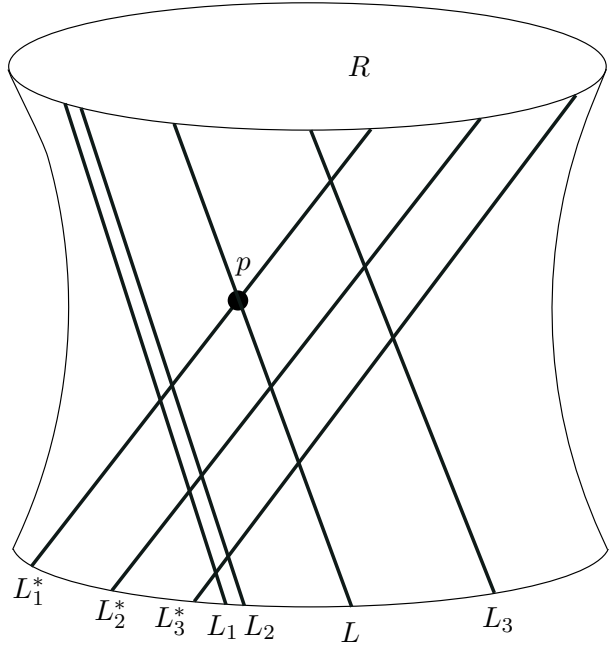
\begin{figure}[h!]
 \centering
\begin{overpic}[width=0.6\textwidth]{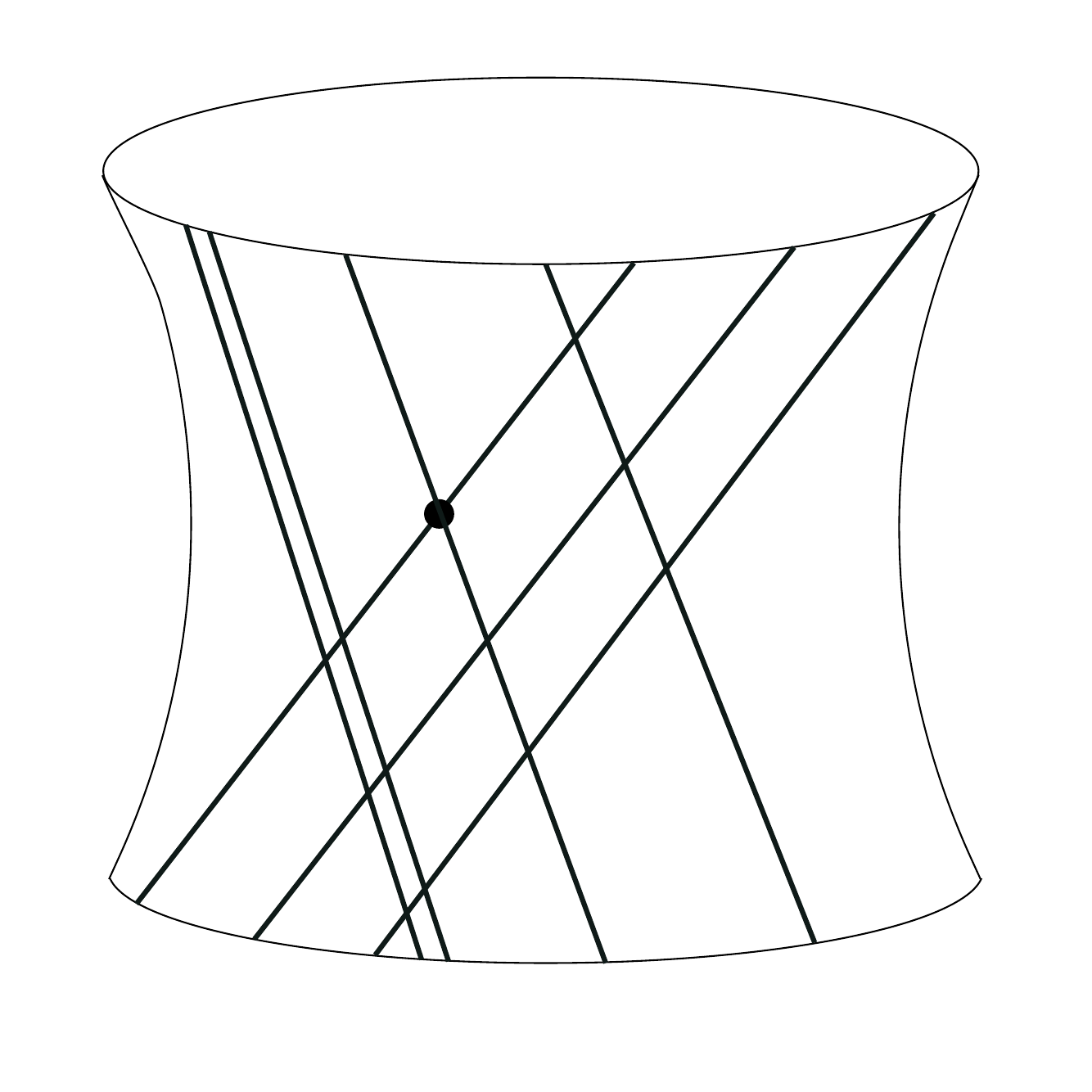}
  \put (10,13) {$L_1^*$} 
  \put (21,10) {$L_2^*$} 
  \put (29,9) {$L_3^*$} 
  \put (36,8) {$L_1$} 
  \put (41,8) {$L_2$}
  \put (54,7) {$L$}
  \put (73,9) {$L_3$}
  \put (40,57) {$p$}
  \put (55,83) {$R$}
\end{overpic}
 \caption{Illustration of Lemma \ref{threeSkewLinesGetYouAnywhere}.}\label{lemma25_fig} 
\end{figure}

\begin{proof}
Let $L_1^*$ be the unique line passing through $p$ that is incident to $L_1,L_2,$ and $L_3$. Select $L_2^*$ and $L_3^*$ to be any two lines incident to $L_1,L_2,$ and $L_3$ with $\dist(L_i^*\cap L_1,L_j^*\cap L_1)\approx_{\eps} 1$ for each pair $i,j$. Apply Lemma \ref{ThreeSkewLinesTwoTransversals} to each of the pairs $L_1^*,L_2^*$; $L_1^*,L_3^*$; and $L_2^*,L_3^*$. Select $L$ to be the unique line passing through $p$ that is incident to $L_1^*,L_2^*,$ and $L_3^*$. The proof that $\angle(L_1^*,L)\approx_{\eps} 1$ is the same as that in Lemma \ref{BoxOfFourLines} (applied to $L_1^*$ and $L_3^*$).
\end{proof}
\begin{lem}\label{CurvatureAtAPoint}
Let $L_1,L_2,$ and $L_3$ be three lines that intersect the unit ball. Suppose that all three lines are pairwise $\geq\delta^{\eps}$ separated and skew. Suppose as well that $\angle(L_1,L_2)\leq 1/10,\ \angle(L_1,L_3)\leq 1/10$.

Let $L$ be a line intersecting each of $L_1,L_2,$ and $L_3$, and suppose $\angle(L,L_i)\geq\delta^{\eps}$, for all $i=1,2,3$. Let $R$ be the regulus containing $L_1,L_2,$ and $L_3$ and let $p=L_1\cap L$. Then the Gauss curvature of $R$ at $p$ satisfies
\begin{equation}\label{boundOnK}
|K_p|\approx_{\eps} 1.
\end{equation}
\end{lem}
\begin{proof}
After applying a rigid transformation, we can assume that $L_1$ is the $z$ axis, $p$ is the origin, and $T_pR$ is the $yz$ plane. After re-scaling by a factor of $O(1)$, we can also assume that $L_1,L_2,$ and $L_3$ still intersect the unit ball. Let $\theta=\angle(L_1^*,L)$, so $L=(0, u \sin \theta,u \cos \theta)$. Then the lines $L_1^*,L_2^*,L_3^*$ have the form
\begin{equation}\label{threeLines}
\begin{split}
&L_1:\ (0,0,0) + \RR(0, 0, 1),\\
&L_2:\ (0, u_1 \sin \theta, u_1 \cos \theta)  + \RR (v_1,v_2,v_3),\\
&L_3:\ (0, u_2 \sin \theta, u_2 \cos \theta)  + \RR (w_1,w_2,w_3).
\end{split}
\end{equation}
We will call these three lines the ``generators'' of the regulus.

Next, we will calculate the function $\big(y_1(s),y_2(s),y_3(s)\big)$ so that the line through $(0,0,s)$ in the direction $\big(y_1(s),y_2(s),y_3(s)\big)$ intersects the second and third generators, i.e.
\begin{align*}
&\big(y_1(s),y_2(s),y_3(s)\big)\in \operatorname{span}\Big((0, u_1 \sin \theta,u_1 \cos \theta  - s),\ (v_1,v_2,v_3)\Big),\\
&\big(y_1(s),y_2(s),y_3(s)\big)\in \operatorname{span}\Big( (0, u_2 \sin \theta, u_2 \cos \theta - s),\ (w_1,w_2,w_3)\Big).
\end{align*}
We obtain
$$
(y_1,y_2,y_3) = \Big( (0,\ u_1 \sin \theta,\ u_1 \cos \theta  - s) \times (v_1,v_2,v_3)\Big)  \times
                 \Big(0,\ u_2 \sin \theta,\ u_2 \cos \theta - s) \times (w_1,w_2,w_3)\Big).
$$
The Gauss curvature at the origin is given by
$$
K=  -\frac{ (y_1^{\prime}(0))^2}{(y_2(0))^2}.
$$
This is because we can parameterize the surface as
$$
r(s,t)=(0,0,s) + t \big(y_1(s),\ y_2(s),\ y_3(s)\big).
$$
The Gauss curvature at the origin is then given by
$$
K=\frac{LN-M^2}{EG - F^2},
$$
where $(E,F,G)$ is the first fundamental form, and $(L,M,N)$ is the second fundamental form, i.e.
\begin{align*}
L&= r_{ss} \cdot n,\\
M&= r_{st} \cdot n,\\
N&= r_{tt} \cdot n,
\end{align*}
where $n$ is the normal vector of $R$ at the origin.

At the origin, $s=t=0$, so we have $r_{ss}=r_{tt} = 0$, and $r_{st} =\big(y_1^{\prime}(0),y_2^{\prime}(0),y_3^{\prime}(0)\big)$. Since $n= (1,0,0)$, we have $M= y_1^{\prime}(0)$. Similarly $E=1$, $G=(y_1(0))^2 + (y_2(0))^2 + (y_3(0))^2$ and $F = y_3(0)$, but since $y(0)$ is in the direction $(0, \sin \theta,\cos \theta)$ we have that $y_1(0)=0$.

Each of the lines $L_1,L_2,$ and $L_3$ from \eqref{threeLines} can be written in the form $L_i = (a_i, b_i, 0) + \RR(c_i, d_i,1).$ Define
$$
X_{ij}=\left|\begin{array}{ll}
a_i - a_j &    b_i - b_j\\
c_i - c_j &    d_i - d_j
\end{array}\right|.
$$

Observe that $|X_{ij}|$ is six times the volume of the tetrahedron spanned by the points $ (a_i, b_i, 0),\ (a_i + c_i, b_i + d_i, 1),\ (a_j, b_j, 0),\ (a_j + c_j, b_j + d_j, 1)$. Since $L_i$ and $L_j$ are $\geq\delta^{\eps}$ separated and skew,
\begin{equation}\label{XijIsOne}
|X_{ij}|\approx_{\eps} 1.
\end{equation}

The values of $(a_i,b_i,c_i,d_i)$ associated to the lines $L_1,L_2,$ and $L_3$ are as follows.

\begin{align*}
&L_1:\ \big(0,\ 0,\ 0,\ 0\big),\\
&L_2:\ \big(-(v_1/v_3)u_1\cos\theta,\ u_1\sin\theta- (v_2/v_3)u_1\cos\theta,\ v_1/v_3,\ v_2/v_3\big),\\
&L_3:\ \big(-(w_1/w_3)u_2\cos\theta,\ u_1\sin\theta- (w_2/w_3)u_2\cos\theta,\ w_1/w_3,\ w_2/w_3\big).
\end{align*}

The matrix determinants $X_{ij}$ are given by

\begin{align*}
X_{12} &= (u_1 v_1 / v_3) \sin \theta,\\
X_{13} &= (u_2 w_1 / w_3) \sin \theta,\\
X_{23} & = (u_1-u_2)(v_3w_1 \sin \theta -v_2w_1 \cos \theta - v_1w_3 \sin \theta + v_1w_2 \cos \theta) / v_3w_3.
\end{align*}

A computation shows that
$$
K^{1/2} X_{12} X_{13} X_{23} = (u_1-u_2)^2 v_1^2 w_1^2 \sin^2 \theta / (v_3^2 w_3^2).
$$
As noted above, $|X_{ij}|\approx_{\eps} 1$ whenever $i\neq j$.  Since $L_2$ and $L_3$ are $1$ separated with error $\delta^{\eps}$, $(u_1-u_2)\approx_{\eps} 1$. Since $\angle(L_1,L_2)\leq 1/10$, we have $v_3\sim 1$ and similarly $w_3\sim 1$.

We will now show that $v_1\approx_{\eps} 1$ and $w_1\approx_{\eps} 1$. Select another line $L^*$ that is incident to $L_1,L_2,$ and $L_3$ and is $\approx_{\eps} 1$ separated and skew to $L$ (Lemma \ref{ThreeSkewLinesTwoTransversals} guarantees that this is possible). Since $L$ lies in the $yz$ plane, we must have that $L^*$ makes an angle $\approx_{\eps} 1$ with the $yz$ plane. Furthermore, $L_2$ and $L_3$ intersect $L^*$ at locations that have distance $\approx_{\eps} 1$ from the origin, and thus the $x$--coordinate of these intersections has magnitude $\approx_{\eps} 1$. Since the $x$ coordinate of $L_2\cap L$ and $L_3\cap L$ is $0$, we conclude that $v_1\approx_{\eps} 1$ and $w_1\approx_{\eps} 1$. Thus we conclude that $|K|\approx_{\eps} 1$, which establishes \eqref{boundOnK}.
\end{proof}

\begin{lem}\label{GaussCurvatureOfRegulus}
Let $L_1,L_2,$ and $L_3$ be three lines that intersect the unit ball. Suppose that all three lines are pairwise $\geq\delta^{\eps}$ skew; $L_1,L_2$ are uniformly separated with error $\delta^{\eps}$; and $L_1,L_3$ and $L_2,L_3$ are $1$ separated with error $\delta^{\eps}$.

Let $R$ be the regulus containing $L_1,L_2,$ and $L_3$. Then for all $p\in R\cap B(0,1)$, we have the bound
$$
|K_p|\approx_{\eps} 1.
$$
\end{lem}
\begin{proof}
Apply Lemma \ref{threeSkewLinesGetYouAnywhere} and then Lemma \ref{CurvatureAtAPoint}.
\end{proof}
\subsection{Quantitatively non-degenerate reguli}\label{quantNonDegenerateRegSection}
A regulus in $\RR^3$ is specified by nine parameters, and a regulus can degenerate in many different ways as these parameters vary. A particularly subtle example is the following. As discussed at the beginning of Section \ref{reguliiSection}, a regulus can be expressed as the union of all lines that intersect three skew lines $L_1,L_2,$ and $L_3$. If $L_1,L_2$, and $L_3$ are all parallel to a common plane, then the regulus is called a hyperbolic paraboloid. Otherwise, the regulus is called a hyperboloid. We will generally work with the latter type of regulus. Since we require quantitative estimates about set of tubes, we will sometimes need to quantify the extent to which a regulus resembles a hyperbolic paraboloid versus a hyperboloid.

In the lemma below, we say a polynomial $Q\in\RR[x,y,z]$ is monic if each coefficient has magnitude at most one, and at least one coefficient has magnitude one. . We define $Z(Q) = \{(x,y,z)\in\RR^3\colon Q(x,y,z)=0\}$ to be the zero-locus of $Q$. 

\begin{lem}\label{regulusThreeLinesFarFromCommonPlane}
Let $L_1,L_2,$ and $L_3$ be three lines that intersect the unit ball. Suppose that $L_1,L_2,$ and $L_3$ are $\geq\delta^{\eps}$ separated and skew. Suppose furthermore that $L_3$ makes an angle $\geq\delta^{\eps}$ with the plane spanned by the vectors $v(L_1)$ and $v(L_2)$ (i.e.~$L_1,L_2$ and $L_3$ are far from being parallel to a common plane, and thus the regulus generated by $L_1,L_2,$ and $L_3$ is far from being a hyperbolic paraboloid). Then there is a monic degree-two polynomial $Q$ that vanishes on $L_1,L_2,$ and $L_3$ that satisfies $|\nabla Q|\approx_{\eps} 1$ on $Z(Q)\cap B(0,1)$. 
\end{lem}
\begin{proof}
We will describe an affine transformation $T\colon\RR^3\to\RR^3$ that sends $L_1,L_2,$ and $L_3$ to the lines $L_1^*=\RR(1,0,0)$, $L_2^*=(0,1,0)+\RR(0,0,1)$, and $L_3^*=(1,0,1)+\RR(0,1,0)$, respectively. This affine transformation is of the form $T(x)= Ax+b$, where the entries of $A$ have magnitude $\lessapprox_{\eps}1$, the determinant of $A$ has magnitude $\approx_{\eps}1$, and $|b|\lessapprox_{\eps}1$.

There is an explicit monic degree-two polynomial $Q_0$ that vanishes on $L_1^*,L_2^*$ and $L_3^*$ and that satisfies $|\nabla Q|\gtrapprox_{\eps}1$ on $Z(Q_0)\cap B(0,\delta^{-C\eps})$, where the implicit constant in the $\lessapprox_{\eps}$ notation depends on the constant $C$. Thus the pull-back of $Q_0$ by $T$ is a degree-two polynomial $Q_1$ whose largest coefficient has magnitude $\approx_{\eps}1$ and whose gradient has magnitude $\approx_{\eps} 1$ on $Z(Q_1)\cap B(0,1)$. Multiplying this polynomial by a constant of magnitude $\approx_{\eps} 1$, we obtain the desired monic polynomial $Q$. In the remainder of this proof we will describe the affine transformation $T$ in greater detail.

Applying a rotation and a translation of the form $(x,y,z)\mapsto (x,y,z)+w$ with $|w|\sim 1$, we can assume that $L_1$ is the $x$ axis $\RR(1,0,0)$; $L_2$ is parallel to the $xz$ plane, and $L_2$ passes through the point $(0, y_0,0)$, with $|y_0|\sim 1$. Furthermore, $\angle(L_1,L_2)\approx_{\eps} 1$. Applying a linear transformation with determinant $\approx_{\eps} 1$ that fixes $L_1$, we can assume that $L_2 = (0,1,0)+\RR(0,0,1)$. After applying this transformation, $L_1,L_2,$ and $L_3$ are still $1$ separated with error $\delta^{C_0\eps}$; $L_1,L_2,$ and $L_3$ are still $\approx_{\eps}1$ skew; $L_3$ makes an angle $\approx_{\eps}1$ with the $xz$ plane; and $L_3$ intersects $B(0, \delta^{-C_1\eps})$, where $C_1$ is an absolute constant.

Thus $L_3$ contains a point $(x_0,0,z_0)$ with $|z_0|\approx_{\eps} 1$ and $|x_0|\lessapprox_{\eps} 1$. Observe that any affine transformation of the form
$$
(x,y,z)\mapsto (ax+b(1-y),\ y,\ cz+dy)
$$
fixes the lines $L_1$ and $L_2$. Let $T_1$ be the transformation of the above type with $a=1,\ b=1-x_0,\ c=1/z_0,\ d=1$. Then $T(x_0,0,z_0)=(1,0,1)$, so $(1,0,1)\in T_1(L_3)$. Abusing notation slightly, we will apply the transformation $T_1$ to the lines $L_1,L_2,$ and $L_3$, and call their images $L_1,L_2$, and $L_3$ as well. Note that $L_3$ still makes an angle $\approx_{\eps} 1$ with the $xz$ plane. Thus we can select a point $(x_1,1/2,z_1)\in L_3$ with $|x_1|,|z_1|\lessapprox_{\eps} 1$.

Furthermore, since $L_3$ is $\approx_{\eps} 1$ skew to $L_2=(0,1,0)+\RR(0,0,1)$, we have $|x_1-1/2|\gtrapprox_{\eps} 1,$ i.e. $|x_1-1/2|\geq \delta^{C_2\eps}$ for some absolute constant $C_2$. If not, then $L_3$ must lie in the $\delta^{C_2\eps}$ neighborhood of the plane $(1,0,1)+(1,1,0)^\perp = (1,0,0)+(1,1,0)^\perp$. This plane contains the line $L_2$. Thus if $C_2$ is chosen sufficiently large, then this contradicts the assumption that $L_2$ and $L_3$ are $\approx_{\eps}1$ skew.

Now consider the map
$$
(x,y,z)\mapsto(ax+(1-a)(1-y),\ y,\ z+dy).
$$
This map fixes the lines $L_1$ and $L_2$, and the point $(1,0,1)$. We wish to select a map of the above form so that the image of $L_3$ under this map points in the $(0,1,0)$ direction, i.e.
\begin{align*}
&(1,0,0)\cdot\Big((ax_1+(1-a)(1-1/2), 1/2, z_1+d(1/2))-(1,0,1)\Big)=0,\\
&(0,0,1)\cdot\Big((ax_1+(1-a)(1-1/2), 1/2, z_1+d(1/2))-(1,0,1)\Big)=0.
\end{align*}

We conclude that $a=1/(2x_1-1),\ d = 2-2z_1$. This linear map has determinant
$$
\left|\begin{array}{lll}a&a&0\\0&1&0\\0&d&1\end{array}\right|=a=1/(2x_1-1)\approx_{\eps} 1,
$$
which completes the proof.
\end{proof}
\subsection{Regulus strips}

\begin{defn}\label{defnOfNonDegenerate}
Let $R$ be a regulus that intersects the unit ball. We say that $R$ is $\delta^{\eps}$ non-degenerate if there are three lines contained in $R$ that intersect the unit ball and that are pairwise $\geq\delta^\eps$ separated and skew. By Lemma \ref{GaussCurvatureOfRegulus}, the Gauss curvature of such a regulus has magnitude $\approx_{\eps} 1$ on $B(0,1)\cap R$. If the value of $\eps$ is apparent from context, we may say that a regulus is ``non-degenerate'' if it is $\delta^{C\eps}$ non-degenerate for some absolute constant $C$.

\end{defn}

\begin{defn}\label{algRegulusDefn}
Let $\tube_1,\tube_2,\tube_3$ be three tubes. The set $R_{\tube_1,\tube_2,\tube_3}$ is the union of all lines in $\RR^3$ that intersect the lines coaxial with $\tube_1,$ $\tube_2$, and $\tube_3$. In practice these lines will always be skew, so $R_{\tube_1,\tube_2,\tube_3}$ will always be a regulus.
\end{defn}

\begin{defn}\label{regulusStripDefn}
A $\delta^{\eps}$ regulus strip is the intersection of the $\delta$-neighborhood of a $\delta^{\eps}$ non-degenerate regulus with the $\delta^{1/2}$ neighborhood of a line in the ruling of the regulus. If $R_{\tube_1,\tube_2,\tube_3}$ is a regulus, a ``regulus strip'' will always refer to the $\delta^{1/2}$ neighborhood of a line in the ruling of the regulus that is dual to the lines coaxial with $\tube_1,\tube_2,$ and $\tube_3$.

Heuristically, one should think of a regulus strip as a rectangular prism of dimensions $1\times\delta^{1/2}\times\delta$ that has been ``twisted'' along its long axis. A regulus strip can also be thought of as a union of $\delta^{-1/2}$ disjoint rectangular prisms, each of dimensions $\delta^{1/2}\times\delta^{1/2}\times\delta$. Each of these prisms has a normal vector (defined up to uncertainty $\delta^{1/2})$, and these normal vectors point in $\gtrapprox_{\eps}\delta^{1/2}$--separated directions.
\end{defn}

\section{The structure of Besicovitch sets near dimension $5/2$}

\begin{defn}\label{defnEpsExtremalKakeyaSet}
We say a set $(\tubes,Y)$ of $\delta$--tubes satisfying the Wolff axioms is $\eps$--extremal if
\begin{equation}\label{mostlyFull}
\sum_{\tube\in\tubes}|Y(\tube)|\geq\delta^{\eps},
\end{equation}
and
\begin{equation}\label{smallVolume}
\Big|\bigcup_{\tubes}Y(\tube)\Big|\leq\delta^{1/2-\eps}.
\end{equation}
Observe that \eqref{mostlyFull} implies that $|\tubes|\gtrsim\delta^{-2+\eps}$. On the other hand, since $\tubes$ satisfies the Wolff axioms we must have $|\tubes|\lesssim \delta^{-2}$.
\end{defn}

\subsection{Wolff's Hairbrush estimate and its consequences}
In this section we will show that every $\eps$--extremal set of tubes must have a certain structure. The main tool we will use is Wolff's maximal function estimate at dimension $5/2$. The following theorem is a direct consequence of Theorem 1 from \cite{wolff}.


\begin{thm}\label{WolffBd}
 Let $(\tubes,Y)$ be a set of $\delta$--tubes that satisfy the Wolff axioms and suppose that $\sum_{T\in\tubes}|Y(\tube)|\geq \lambda \delta^{2}|\tubes|$ for some $\delta\leq\lambda\leq 1$. Then for each $s>0$, there is a constant $C_s$ so that
 \begin{equation}\label{Wolff52BoundEqn}
  \Big|\bigcup_{\tube\in\tubes}Y(\tube)\Big|\geq C_s \lambda^{5/2} \delta^{1/2+s} (\delta^{2}|\tubes|)^{3/4}.
 \end{equation}
\end{thm}

A key lemma used to prove Wolff's result shows that the ``hairbrush'' of a tube has large volume. The following lemma is a slight variant of Lemma 3.4 from \cite{wolff}.

\begin{lem}\label{WolffHairbrushBd}
 Let $(\tubes,Y)$ be a set of $\delta$--tubes that satisfy the Wolff axioms and suppose that $\sum_{T\in\tubes}|Y(\tube)|\geq \lambda \delta^{2}|\tubes|$ for some $\delta\leq\lambda\leq 1$. Let $L\subset\RR^3$ be a line and suppose that each tube  $\tube\in\tubes$ intersects $L$ and satisfies $\angle(v(L), v(\tube))\geq\delta^\eps$. Then for each $s>0$, there is a constant $C_s$ so that
$$
\Big|\bigcup_{\tube\in\tubes}Y(\tube)\Big|\gtrapprox_{\eps}C_s\lambda^{5/2}\delta^{2+s}|\tubes|.
$$
\end{lem}

Theorem \ref{WolffBd} has several consequences, which we will detail below.

\begin{defn}
We say that a pair $(\tubes^\prime,Y^{\prime})$ is a refinement of $(\tubes,Y)$ if $\tubes^\prime\subset\tubes$,
$Y^{\prime}(\tube) \subset  Y(\tube)$
for each $\tube \in \tubes^\prime$, and
$$
\sum_{\tube\in\tubes^\prime}|Y^\prime(\tube)|\gtrapprox_{\eps}\sum_{\tube\in\tubes}|Y(\tube)|.
$$
The meaning of the variable $\eps$ will always be apparent from context.

If $(\tubes^\prime,Y^\prime)$ is a refinement of $(\tubes,Y)$ with $\tubes^\prime=\tubes,$ then we may refer to the refinement as $(\tubes,Y^\prime)$. If $(\tubes^\prime,Y^\prime)$ is a refinement of $(\tubes,Y)$ with $Y^\prime(\tube)=Y(\tube)$ for all $\tube\in\tubes^\prime$, then we may refer to the refinement simply as $\tubes^\prime$.
\end{defn}

In order to study a set $(\tubes,Y)$ of tubes and their associated shadings, it will often be helpful to consider the function $\sum_{\tube\in\tubes}\chi_{Y(T)}$. This is a function that takes integer values between $0$ and $|\tubes|$. Frequently, we will use dyadic pigeonholing to find a refinement $(\tubes,Y^\prime)$ so that $\sum_{\tube\in\tubes}\chi_{Y^\prime(T)}(p)\sim\mu\chi_A(p)$ for all $p\in\RR^3$. Here $\mu$ is an integer and $A\subset\RR^3$ with $\mu|A|\sim\sum_{\tube\in\tubes}|Y^\prime(\tube)|$. If $(\tubes,Y)$ is $\eps$-extremal, then $\mu\approx_{\eps}\delta^{-1/2}$ and $|A|\approx_{\eps}\delta^{1/2}$. In future, we will perform this type of dyadic pigeonholing without remarking on it further.

The next lemma allows us to reduce to the case where most pairs of intersecting tubes make a large angle. Arguments of this type appear frequently in the Kakeya literature (see e.g.~\cite[Section 2.3]{GZ}, and specifically Proposition 2.3) , and the following lemma is a standard application of common techniques.
\begin{lem}[Robust transversality]\label{robustTransLem}
Let $(\tubes,Y)$ be an $\eps$--extremal set of tubes. Then there is a refinement $(\tubes, Y^\prime)$ so that $\sum_{\tube\in\tubes}\chi_{Y^\prime(T)}\sim \mu\chi_A$, and for each $p\in A$, each $\theta\in(\delta,1)$, and each unit vector $v$, we have
\begin{equation}\label{tubesNotTooFocussed}
|\{\tube\in\tubes\colon p\in Y^\prime(\tube),\ \angle(v(\tube),v)\leq \theta \}| \lessapprox_\eps \theta^{1/10}|\{\tube\in\tubes\colon p\in Y^\prime(\tube)\}|.
\end{equation}
\end{lem}
Note that if $(\tubes,Y)$ obeys the conclusions of Lemma \ref{robustTransLem} and if $(\tubes^\prime, Y^\prime)$ is a refinement of $(\tubes,Y)$ with $\sum_{\tube\in\tubes^\prime}\chi_{Y^\prime(\tube)}\sim\mu\chi_A$, then $(\tubes^\prime,Y^\prime)$ must also obey the conclusions of Lemma \ref{robustTransLem} (though the implicit constant in \eqref{tubesNotTooFocussed} may be larger).

The next lemma is a version of Lemma \ref{WolffHairbrushBd} where both the assumptions and conclusion have been weakened. It says that if all of the tubes in a Besicovitch set intersect the (thickened) neighborhood of a line segment, then the union of these tubes must have large volume.
\begin{lem}[Volume of a fat hairbrush]\label{concentratedTubesHaveLargeVolume}
Let $\rho>\delta$.
Let $(\tubes,Y)$ be a set of $\delta$ tubes that satisfy the Wolff axioms. Suppose that $\sum_{\tube\in\tubes}|Y(\tube)|\geq\delta^{2+\eps}|\tubes|$, and that there exists a line $L$ so that $\tube\cap N_{\rho}(L)\neq\emptyset$ and $\angle(v(\tube),\ L)\geq\delta^{\eps}$ for each $\tube\in\tubes$.  Then
\begin{equation}\label{volumeOfCompressedKakeyaSet}
\Big|\bigcup_{\tube\in\tubes}Y(\tube)\Big|\gtrapprox_{\eps} \delta^{1/2}\rho^{-1/4}(\delta^2|\tubes|).
\end{equation}
\end{lem}
\begin{proof}
Since $\tubes$ satisfies the Wolff axioms, we have $|\tubes|\lesssim\delta^{-2}$. Replacing $(\tubes,Y)$ by a refinement if necessary, we can assume that $|Y(\tube)|\geq\frac{1}{C}\delta^{\eps}|\tube|$ for each $\tube\in\tubes$, where $C$ is an absolute constant, and that $\sum_{\tube\in\tubes}|Y(T)|\gtrsim \delta^{\eps}$. Let $\Pi_i,\ i=1,\ldots,\rho^{-1}$ be a set of planes that contain $L$ and have normal vectors that point in $\rho$--separated directions. For each index $i,$ let $H_i$ be the $\rho$--neighborhood of $\Pi_i$. Then the sets $\{H_i \backslash N_{\delta^{2\eps}}(L)\}$ are at most $\delta^{-2\eps}$-overlapping. Note as well that since $\angle(v(\tube),\ L)\geq\delta^{\eps}$ for each $\tube\in\tubes,$
$$
|Y(\tube)\backslash N_{\delta^{2\eps}}(L)|\gtrsim\delta^{\eps}|\tube|.
$$

Observe that if $\tube$ is a $\delta$--tube that intersects $N_{\rho}(L)$, then there exists (at least one) index $i$ so that $\tube\subset 2H_i$. Assign each tube $\tube\in \tubes$ to an index $i$. For each index $i$, let $N_i$ be the number of tubes assigned to $i$.

By Theorem \ref{WolffBd}, for each index $i$ we have
 \begin{equation}
  \Big|\bigcup_{\substack{\tube\in\tubes\\ \tube\ \textrm{assigned to}\ H_i}}\!\!\!\!\!\!\!\! Y(\tube)\backslash N_{\delta^{2\eps}}(L)\Big|\gtrapprox_{\eps} \delta^{1/2} (\delta^{2}N_i)^{3/4}.
 \end{equation}
Since the tubes in $\tubes$ satisfy the Wolff axioms, we have $N_i\leq \rho\delta^{-2}$ for each index $i$, and of course $\sum_i N_i\geq|\tubes|$. Thus by appling Theorem \ref{WolffBd} to the set of tubes assigned to each $H_i$ and summing, we obtain
 \begin{equation*}
  \sum_i \Big|\bigcup_{\substack{\tube\in\tubes\\ \tube\ \textrm{assigned to}\ H_i}}\!\!\!\!\!\!\!\! Y(\tube)\backslash N_{\delta^{2\eps}}(L)\Big|\gtrapprox_{\eps} \delta^{1/2}\rho^{-1/4}(\delta^2|\tubes|).
 \end{equation*}
 Since the sets $\{H_i\}$ are at most $\delta^{2\eps}$--overlapping, we have
 \begin{equation*}
\Big|\bigcup_{\tube\in\tubes}Y(\tube)\Big|\gtrapprox_{\eps} \delta^{1/2}\rho^{-1/4}(\delta^2|\tubes|).\qedhere
 \end{equation*}
\end{proof}

\begin{cor}[Few tubes in a fat hairbrush]\label{fewTubesHitALine}
Let $(\tubes,Y)$ be an $\eps$--extremal set of tubes. Then for every line $L\subset\RR^3$, we have
$$
|\{\tube\in\tubes\colon \tube\cap N_{\rho}(L)\neq\emptyset,\ \angle(v(\tube),\ L)\geq\delta^{\eps}\}|\lessapprox_\eps\delta^{-2}\rho^{1/4}.
$$
\end{cor}

At the beginning of Section \ref{quantNonDegenerateRegSection}, we observed that three skew lines either determine a hyperbolic paraboloid or a hyperboloid. The next lemma shows that if most triples of tubes in a Besicovitch set are parallel to a common plane, then the Besicovitch set must have large volume. In particular, this means that hyperbolic paraboloids will not play an important role in the analysis of Besicovitch sets. The proof of the next lemma is nearly identical to the proof of Lemma \ref{concentratedTubesHaveLargeVolume}.

\begin{lem}[Tubes parallel to a common plane have large union]\label{allTubesCoplanar}
Let $\rho>\delta$.
Let $(\tubes,Y)$ be a set of $\delta$ tubes that satisfy the Wolff axioms. Suppose that $\sum_{\tube\in\tubes}|Y(\tube)|\geq\delta^{\eps}$, and that there exists a unit vector $v_0$ so that $v(T)\cdot v_0\leq\rho$ for each $\tube\in\tubes$.  Then
\begin{equation}\label{volumeOfParallelKakeyaSet}
\Big|\bigcup_{\tube\in\tubes}Y(\tube)\Big|\gtrapprox_{\eps} \delta^{1/2}\rho^{-1/4}.
\end{equation}
\end{lem}

%
%
\subsection{Planiness}
In this section, we will show that for each point $p\in \RR^3$, there is a plane $\Pi_p$ containing $p$ so that all of the tubes containing $p$ make a small angle with $\Pi_p$. This is a consequence of Bennett, Carbery, and Tao's multilinear Kakeya theorem from \cite{BCT}. The version we will use here is a variant due to Bourgain and Guth.

\begin{thm}[Bourgain-Guth, \cite{BG}, Theorem 6]
Let $\tubes$ be a set of $\delta$--tubes in $\RR^3$. Then for each $s>0$, there exists a constant $C_s$ so that
\begin{equation}\label{weightedMultiLineKakeya}
\int \Big(\sum_{\tube_1,\tube_2,\tube_3\in\tubes} \chi_{\tube_1}\chi_{\tube_2}\chi_{\tube_3}\ \big|v(\tube_1)\wedge v(\tube_2)\wedge v(\tube_3)\big|\Big)^{1/2}\leq C_s\delta^{-s}(\delta^2|\tubes|)^{3/2}.
\end{equation}
\end{thm}

\begin{lem}[Planiness]\label{Planininess}\label{planinessProp}
Let $(\tubes,Y)$ be a set of $\eps$--extremal tubes. Then there is a refinement $Y^\prime$ of $Y$ so that for each $p\in\RR^3$ there is a plane $\Pi_p$ so that if $p\in Y^\prime(\tube)$, then
$$
\angle\big(v(\tube),\ \Pi_p\big)\lessapprox_{\eps}\delta^{1/2}.
$$
\end{lem}
\begin{proof}
First apply Lemma \ref{robustTransLem} to $(\tubes,Y)$, and let $(\tubes,Y_1)$ be the resulting refinement. Recall that $\sum_{\tube\in\tubes}\chi_{Y_1(\tube)}\sim\mu\chi_A$ for some number $\mu\approx_\eps\delta^{-1/2}$ and some set $A\subset\RR^3$ with $|A|\approx_{\eps}\delta^{1/2}$.
Next, there exists a subset $A_1\subset A$ with $|A_1|\gtrapprox_\eps|A|$ and a number $\tau$ so that for every point $p\in A_1$, there are $\gtrapprox_\eps\mu^3$ triples $\tube_1,\tube_2,\tube_3$ satisfying:
\begin{itemize}
\item $p\in Y_1(\tube_1)\cap Y_1(\tube_2)\cap Y_1(\tube_3)$,
\item $\big|v(\tube_1)\wedge v(\tube_2)\wedge v(\tube_3)\big|\sim\tau$,
\item $\angle(v(\tube_i),v(\tube_j))\gtrapprox_{\eps}1$ if $i\neq j$.
\end{itemize}
(The last item holds because each point $p\in A_1$ satisfies the conclusion of Lemma  \ref{robustTransLem} ). The bound \eqref{weightedMultiLineKakeya} implies that $\tau\lessapprox_{\eps}\delta^{1/2}$. Thus for each $p\in A_1$, there is a pair $(\tube_1,\tube_2)$ with $p\in Y_1(\tube_1)\cap Y_1(\tube_2)$ and $\angle(v(\tube_1),\ v(\tube_2))\gtrapprox_{\eps}1$ so that if $\Pi_p$ is the plane containing $p$ spanned by the vectors $v(\tube_1)$ and $v(\tube_2)$, then there is an absolute constant $C$ so that
$$
|\{\tube\in\tubes\colon p\in Y_1(\tube),\ \angle(v(\tube),\Pi_p)\leq \delta^{1/2-C\eps} \}|\gtrapprox_\eps\delta^{-1/2}.
$$

For each $\tube\in\tubes$, define
$$
Y^\prime(\tube)=\{p\in Y_1(\tube)\colon \angle(v(\tube), \Pi_p) \leq \delta^{1/2-C\eps}\}.\qedhere
$$
\end{proof}
\begin{defn}
If $(\tube,Y)$ satisfies the conclusions of Lemma \ref{planinessProp}, then we say it is plany.
\end{defn}
\subsection{Hairbrushes}
\begin{defn}\label{hairbrushDef}
Let $(\tubes,Y)$ be a set of tubes and let $\tube_1,\ldots,\tube_k\in\tubes$. We define the hairbrush of $\tube_1,\ldots\tube_k$ to be the set of tubes from $\tubes$ that intersect each of $\tube_1,\ldots\tube_k$. More precisely, we define
\begin{equation}\label{defnOfHYEqn}
H_{Y}(\tube_1,\ldots,\tube_k)=\{\tube\in\tubes\colon Y(\tube)\cap Y(\tube_i)\neq\emptyset,\ i=1,\ldots,k\}.
\end{equation}

If there is no chance of confusion we may omit the subscript $Y$ in \eqref{defnOfHYEqn}.
\end{defn}
One consequence of Wolff's hairbrush argument is that if $(\tubes,Y)$ is an $\eps$--extremal set of tubes, then for most $\tube\in\tubes$, most pairs of tubes in the hairbrush $H_Y(\tube)$ are $\gtrapprox_{\eps}$ separated and skew. To be more precise, there is an absolute constant $C$ so that
\begin{equation}\label{mostHairbrushSkew}
\begin{split}
|\{&(\tube,\tube_1,\tube_2)\in\tubes^{3}\colon \tube_1,\tube_2\in H(\tube);\ \tube_1,\tube_2\ \textrm{are}\ \geq \delta^{C\eps}\ \textrm{separated and skew}\}|\\
&\approx_{\eps} |\{(\tube,\tube_1,\tube_2)\in\tubes^{3}\colon \tube_1,\tube_2\in H(\tube)\}|.
\end{split}
\end{equation}

The proof of \eqref{mostHairbrushSkew} is standard, though somewhat lengthy, so we will only provide a brief sketch. If \eqref{mostHairbrushSkew} failed, then after pigeonholing, there is a refinement $(\tube^\prime,Y^\prime)$ so that for each $\tube\in\tubes^\prime$, the tubes in the hairbrush of $\tube$ are contained in the $\delta^{C\eps}$ neighborhood of a plane. On the one hand, the union of the tubes in this hairbrush have volume $\gtrapprox_{\eps}\delta^{1/2}$. On the other hand, the union of these tubes only covers a $\delta^{C\eps/C_1}$ fraction of the union $\bigcup_{T\in\tubes}Y(T)$, where $C_1$ is an absolute constant. Thus if $C$ is chosen sufficiently large, this contradicts the assumption that $(\tubes,Y)$ is an $\eps$--extremal set of tubes. An analogous similar argument also shows that a similar statement holds for $k$-tuples of tubes in $H_Y(\tube$).

\begin{lem}[Most tubes in a hairbrush are skew]\label{mostOfHairbrushIsSepAndSkew}
Let $(\tubes,Y)$ be an $\eps$--extremal set of tubes and let $k\geq 2$ be an integer. Then there is a constant $C$ (depending only on $k$) so that
\begin{align*}
|\{&(\tube,\tube_1,\ldots,\tube_k)\in\tubes^{k+1}\colon \tube_1,\ldots,\tube_k\in H(\tube);\ \tube_1,\ldots,\tube_k\ \textrm{are pairwise}\ \geq \delta^{C\eps}\ \textrm{separated and skew}\}|\\
&\approx_{\eps} |\{(\tube,\tube_1,\ldots,\tube_k)\in\tubes^{k+1}\colon \tube_1,\ldots,\tube_k\in H(\tube)\}|.
\end{align*}
\end{lem}

\subsection{The regulus map}\label{regMapSection}
Let $(\tubes,Y)$ be an $\eps$--extremal set of tubes and let $\tube\in\tubes$. For each $p\in Y(T)$, there is an associated plane $\Pi_p$, as described in Lemma \ref{planinessProp}. In this section, we will ask: as the point $p$ moves along the tube $\tube$, how does $\Pi_p$ change? Sometimes, a particularly interesting phenomenon occurs: at coarse scales, the plane $\Pi_x$ looks like the tangent plane to a regulus containing the line coaxial with $\tube$. If this happens, we say that the tube $\tube$ has a ``regulus map.'' The presence or absence of these regulus maps will steer our proof towards either the Heisenberg or $SL_2$ examples. This statement will be made precise in Section \ref{dichotomySection}.

\begin{defn}\label{failsToAvoidRegulii}
Let $(\tubes,Y)$ be an $\eps$--extremal set of tubes. We say that $(\tubes,Y)$ \emph{avoids reguli} (with parameters $\gamma>0$ and $\beta>0$) if whenever $\{R(\tube)\}_{\tube\in\tubes}$ is a set of non-degenerate reguli for which $R(\tube)$ contains the line coaxial with $\tube$, we have
$$
\sum_{\tube\in\tubes}|\{\tube^\prime\in H(\tube)\colon \angle(v(\tube^\prime), R(\tube))\leq \delta^{\gamma}\}| \leq \delta^{\beta}\sum_{\tube\in\tubes}|H(\tube)|.
$$
\end{defn}

If the above property does not hold, then we will say that $(\tubes,Y)$ \emph{fails to avoid reguli} (with parameters $\gamma$ and $\beta$). In practice, we will set $\gamma=C_1\eps$ and $\beta = C_2\eps$, where $C_1$ and $C_2$ are constants with $C_1$ much larger than $C_2$.

\begin{lem}\label{regulusMapLem}
Let $(\tubes,Y)$ be an $\eps$--extremal set of tubes that is plany, and assume that
\begin{equation}\label{KakeyaSetIsHomogeneous}
\sum_{\tube\in\tubes}\chi_{Y(\tube)}\sim\mu\chi_A
\end{equation}
 for some number $\mu$ and some set $A$. Suppose that $(\tubes,Y)$ fails to avoid reguli with parameters $\gamma$ and $\beta$. Then there exists
\begin{itemize}
\item A number $\delta^{1/2}\leq\theta\leq \delta^{\gamma}$.
\item A set $\tubes^\prime\subset\tubes$, and for each $\tube\in\tubes^\prime$, a set $Y^\prime(\tube)\subset Y(\tube)$ such that
$$
\sum_{\tube^\prime\in\tubes^\prime}|Y^\prime(\tube)|\gtrapprox_{\eps}\delta^{\beta}\sum_{\tube\in\tubes}|Y(\tube)|.
$$
\item For each $\tube\in\tubes^\prime$, a non-degenerate regulus $R(\tube)$ containing the line coaxial with $\tube$.
\end{itemize}
So that if $\tube\in\tubes^\prime$ and $p\in Y^\prime(\tube)$, then $\angle(\Pi_p, T_pR(\tube))\approx_{\eps}\theta$.
\end{lem}
\begin{proof}
Since $(\tubes,Y)$ fails to avoid regulii, for each $\tube\in\tubes$ we can associate a regulus $R(\tube)$, as described in Definition \ref{failsToAvoidRegulii}. For each $\tube\in\tubes$ and each $p\in Y(\tube)$, define
$$
\theta_{p,\tube} =\angle(\Pi_p, T_pR(\tube) ).
$$
Note that $\theta_{p,\tube}$ is defined up to uncertainty $\delta^{1/2}$. In particular, we can assume that $\theta_{p,\tube}\geq \delta^{1/2}$.

By \eqref{KakeyaSetIsHomogeneous}, we have
$$
\sum_{\tube\in\tubes}|\{p\in Y(\tube)\colon \delta^{1/2}\leq\theta_{p,\tube}\leq\delta^{\gamma} \}|\geq\delta^\beta\sum_{\tube\in\tubes}|Y(\tube)|.
$$
Thus by dyadic pigeonholing, we can find a dyadic number $\delta^{1/2}\leq \theta\leq\delta^{\gamma}$ so that if we define
$$
Y^\prime(\tube)=p\in Y(\tube)\colon \theta\leq \theta_{p,\tube}\leq 2\theta,
$$
then
$$
\sum_{\tube\in\tubes}|Y^\prime(\tube)|\geq\delta^\beta|\log\delta|^{-1}\sum_{\tube\in\tubes}|Y(\tube)|,
$$
which completes the proof.
\end{proof}

\begin{defn}
We say that an $\eps$-extremal set of tubes obeys the regulus map at scale $\theta$ if it satisfies the conclusions of Lemma \ref{regulusMapLem}.
\end{defn}

\begin{rem}\label{threeParametersRegulusMap}
If $(\tubes,Y)$ obeys the regulus map at scale $\theta<\delta^{C_0\eps}$ and if $\tube\in\tubes$, then the function $p\mapsto T_p R(\tube)$ is called the \emph{regulus map} associated to the tube $\tube$. If we identify the unit line segment coaxial with $\tube$ with the interval $[0,1]$, and if $\operatorname{Gr}(2,\RR^3)$ is the (affine) Grassmannian of two-dimensional subspaces of $\RR^3$, then the regulus map of $\tube$ becomes a function $f_\tube\colon [0,1]\to\operatorname{Gr}(2,\RR^3)$ that is described by three parameters.

In particular, if $p_1,p_2,p_3\in\tube$ are three $\gtrapprox_{\eps}1$ separated points, and if $\Pi_1,\Pi_2,$ and $ \Pi_3$ are three planes containing the points $p_1,p_2,$ and $p_3$ respectively, then there is a unique (up to uncertainty $\lessapprox_{\eps}\theta$ ) regulus map $p\mapsto T_p R(\tube)$ satisfying $\angle(T_{p_i}R(\tube),\ \Pi_i)\lessapprox_{\eps}\theta,\ i=1,2,3$.
\end{rem}
\begin{rem}\label{virtualTube}\label{virtualTubeRemark}
A useful way of interpreting the regulus map is as follows. Let $(\tubes,Y)$ obey the regulus map at scale $\theta$. Let $\tube_0\in\tubes$ and let $L$ be a line in $R(\tube_0)$ in the same ruling as $\tube_0$ that is $\delta/\theta$ separated (with error $\approx_{\eps}1$) and  $\approx_{\eps}1$ skew to the line coaxial with $\tube_0$. By Lemma \ref{ThreeSkewLinesTwoTransversals}, such an $L$ must exist. Let $\tube_1$ be the $\delta$--neighborhood of $L$.

Let $\tube$ be a tube (not necessarily from $\tubes$) with the following properties: $\tube$ intersects $\tube_0$; $\tube$ satisfies $\angle(v(\tube),\ v(\tube_0))\gtrapprox_{\eps}1$; and $\tube$ makes an angle $\leq \theta$ with the tangent plane of $R(\tube)$ at the point $\tube\cap \tube_0$. Then $\tube$ also intersects the $\lessapprox_\eps \delta$--neighborhood of $\tube_1$. Conversely, if $\tube$ intersects $\tube_0$; $\tube$ satisfies $\angle(v(\tube),\ v(\tube_0))\gtrapprox_{\eps}1$; and $\tube$ also intersects $\tube_1$, then $\tube$ must make an angle $\lessapprox_{\eps}\theta$ with the tangent plane of $R(\tube)$ every point in $\tube\cap \tube_0$.

This means that for each $\tube_0\in\tubes,$ there is a tube $\tube_1$ (not necessarily in $\tubes$) that is $\delta/\theta$ separated from $\tube_0$ with error $\approx_{\eps}1$ so that for each constant $C_1$, there are constants $C_2$ and $C_3$ so that

\begin{equation}\label{equivalenceOfRegulusMapVirtualTube}
\begin{split}
&\{\tube\in H(\tube_0)\colon \angle(v(\tube),\ v(\tube_0))\geq\delta^{\eps},\ \angle(v(\tube),\ T_p R(\tube_0))\leq \delta^{C_1\eps}\theta\}\\
&\quad \subset \{\tube\in H(\tube_0)\colon \angle(v(\tube),\ v(\tube_0))\geq\delta^{\eps},\ \tube\cap N_{\delta^{1-C_2\eps}}(\tube_1)\neq\emptyset\}\\
&\quad\quad \subset \{\tube\in H(\tube_0)\colon \angle(v(\tube),\ v(\tube_0)) \geq\delta^{\eps},\ \angle(v(\tube),\ T_p R(\tube_0))\leq\delta^{C_3\eps}\theta\},
\end{split}
\end{equation}
where in the above equation, $T_p R(\tube_0)$ is the tangent plane of the regulus $R(\tube_0)$ at a point $p\in\tube_0\cap \tube$ (of course the intersection consists of more than one point, but since $\angle(v(\tube),\ v(\tube_0))\gtrapprox_{\eps}1$, the intersection is contained in a ball of radius $\lessapprox_{\eps}\delta$, so the choice of $p$ does not matter). The requirement that $\angle(v(\tube),\ v(\tube_0))\geq\delta^{\eps}$ in each of the three sets above can also be replaced with a requirement of the form $\angle(v(\tube),\ v(\tube_0))\geq\delta^{C\eps}$; doing so will also change the constants $C_2$ and $C_3$.

We can think of $\tube_1$ as a ``virtual tube'' (we call it virtual since it need not be in the set $\tubes$) that describes the regulus map of $\tube_0$. We will write $\tube_1 = V(\tube_0)$.
\end{rem}

\begin{figure}[h!]
 \centering
\begin{overpic}[width=0.4\textwidth]{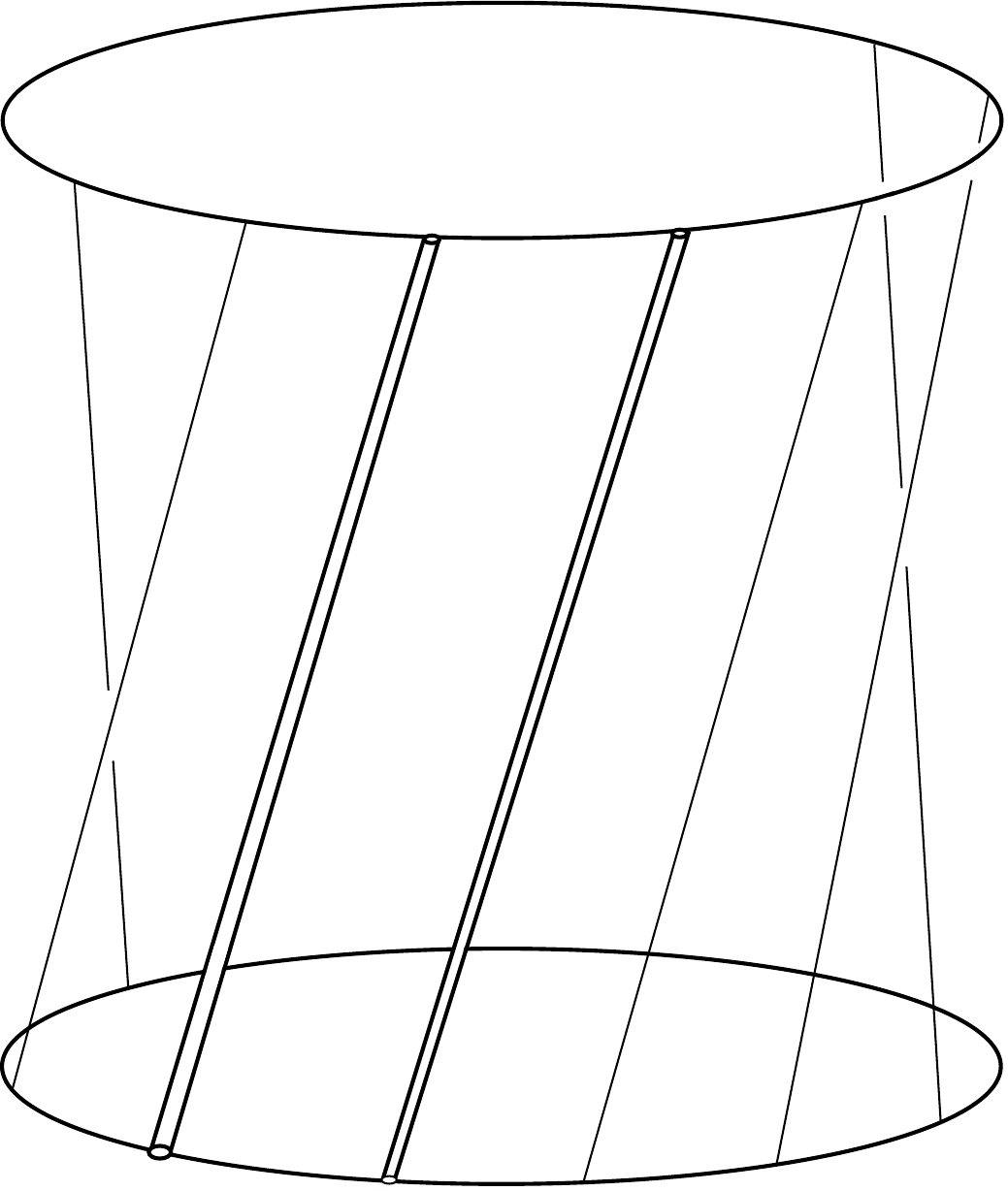}
\put (26,35) {$\tube_0$}
\put (46,35) {$V(\tube_0)$}
\end{overpic}
 \caption{The virtual tube of $\tube_0$ found along the ruling of the regulus}\label{virtualTubePic}
\end{figure}

\subsection{The hairbrush of two tubes is contained in a fat regulus}
If $L_1,L_2,$ and $L_3$ are three skew lines, then the union of all lines intersecting each of $L_1,L_2,$ and $L_3$ forms a regulus. On the other hand, the union of all lines that merely intersect $L_1$ and $L_2$ fill out all of $\RR^3$. In general, one would expect similar statements to hold if lines are replaced by tubes.

As we will see in the following section, however, if we restrict attention to tubes obeying the regulus map at scale $\theta$, then the union of all tubes that intersect \emph{two} skew tubes $\tube_1$ and $\tube_2$ is contained in the $\theta$-neighborhood of a regulus. This is because every tube intersecting $\tube_1$ and $\tube_2$ must also intersect the ``virtual tube'' associated to $\tube_1$. This fact will have several implications, which we will explore below.

\begin{lem}\label{thetaScaleHairbrush}
Let $(\tubes,Y)$ be an $\eps$-extremal set of tubes that obey the regulus map at scale $\theta$. Let $\tube_1,\tube_2\in\tubes$ be two tubes whose coaxial lines $L_1$ and $L_2$ are $\geq\delta^{C\eps}$ separated and skew. Let $L_1^\prime$ be the line coaxial with $V(\tube_1)$.

Then there is a number $\theta^\prime\approx_{\eps}\theta$ so that $H(\tube_1,\tube_2)$ is contained in a union of $\lessapprox_\eps\theta^{-1}$ distinct $\theta^\prime$-tubes (the $\theta^\prime$--neighborhoods of unit line segments) that are $\lessapprox_{\eps}1$ overlapping. Each of these $\theta^\prime$ tubes contains (at least one) line that intersects $L_1,L_2,$ and $L_1^\prime$. Each tube in $H(\tube_1,\tube_2)$ is contained in at least one of the $\theta^\prime$-tubes.
\end{lem}
\begin{proof}
Recall that $\tube_1$ is a tube of length one and thickness $\delta$. Thus we can cover $\tube_1$ by union of $\theta^{-1}$ disjoint tube segments of length $\theta$ (and thickness $\delta)$. If $p,q\in\tube_1$ lie in the same tube segment then $\angle(T_{p}R(\tube_1),\ T_qR(\tube_1))\lessapprox_{\eps}\theta$. Select $\theta_1\approx_{\eps}\theta$ sufficiently large so that $\angle(T_{p}R(\tube_1),\ T_qR(\tube_1))\leq\theta_1$ whenever $p$ and $q$ lie in the same segment.

 We will be interested in the tube segments that intersect at least one tube from $H(\tube_1,\tube_2)$. For each such tube segment, associate the set $N_{\theta_1}(T_pR(\tube_1))$, where $p$ is a point in the tube segment. We will call these sets ``slabs''; each slab is the $\theta_1$--neighborhood of a plane. Observe that if $\tube\in H(\tube_1,\tube_2)$, then $\tube$ is contained in the slab associated to the segment at which the intersection $Y(\tube)\cap Y(\tube_1)$ occurs (if the set $Y(\tube)\cap Y(\tube_1)$ intersects more than one tube segment, then $\tube$ is contained in all of the corresponding slabs).

Since $\tube_1$ and $\tube_2$ are $\delta^{C\eps}$ separated and skew, $v(\tube_2)$ makes an angle $\gtrapprox_{\eps}1$ with any plane containing the line coaxial with $\tube_1$ that intersects $\tube_2$ in the unit ball. Thus each of the slabs described above intersects $\tube_2$ in a tube segment of length $\approx_{\eps}\theta$, and these tube segments are $\lessapprox_{\eps}1$ overlapping. Thus if $W$ is a tube segment from $\tube_1$, then the set of tubes in $H(\tube_1,\tube_2)$ that intersect $\tube_1$ at a point of $W$ are contained in the $\approx_{\eps}\theta_1\approx_{\eps}\theta$ neighborhood of a line segment. Define $\theta^\prime\approx_{\eps}\theta$ so that the aforementioned tubes are contained in the $\theta^\prime$ neighborhood of a line segment.

Finally, observe that $N_{\theta_1}(T_pR(\tube_1))\cap B(0,2)$ is comparable\footnote{We say that two sets $A$ and $B$ are comparable if there is an absolute constant $C\geq 1$ (independent of $\delta$) so that $B$ is contained in the $C$-fold dilate of $A$ and vice-versa.} to the intersection of $B(0,2)$ with the $\theta_1$--neighborhood of the plane spanned by $p$ and $L_1^\prime$. If this plane intersects $L_2$ at the point $q$, then the line containing $p$ and $q$ intersects each of $L_1,L_2$, and $L_1^\prime$, and the line is contained in $N_{\theta_1}(T_pR(\tube_1))\cap\tube_2$. Thus each of the $\theta^\prime$-tubes described above contains a line that intersects each of $L_1,L_2,$ and $L_1^\prime$.
\end{proof}

\begin{cor}\label{hairbrushTwoTubesInFatRegulusOne}
Let $(\tubes,Y)$ be an $\eps$-extremal set of tubes that obey the regulus map at scale $\theta$. Let $\tube_1,\tube_2\in\tubes$ be two tubes that are $\geq\delta^{C\eps}$ separated and skew. Then $H(\tube_1,\tube_2)$ is contained in the $\lessapprox_{\eps}\theta$-neighborhood of a non-degenerate regulus; we will call this regulus $R(\tube_1,\tube_2)$.
\end{cor}
\begin{proof}
Apply Lemma \ref{thetaScaleHairbrush} to $\tube_1$ and $\tube_2$. The only thing to establish is that the regulus generated by the lines coaxial with $\tube_1,\ V(\tube_1)$, and $\tube_2$ are non-degenerate. Let $L_1$ and $L_2$  be the line coaxial with $\tube_1$ and $\tube_2$, respectively. Let $L_1^\prime$ be the line coaxial with $V(\tube_1)$. Recall from Remark \ref{virtualTubeRemark} that $L_1^\prime$ is $\delta/\theta\leq\delta^{1/2}$--separated (with error $\approx_{\eps}1$) and $\approx_{\eps}1$ skew to $L_1$. Since $L_2$ is $\approx_{\eps}1$--separated from $L_1$, we conclude that $L_2$ is $\approx_{\eps}1$--separated and $\approx_{\eps}1$ skew to $L_1^\prime$. Thus the regulus $R(\tube_1,\tube_2)=R_{L_1,L_1^\prime,L_2}$ is non-degenerate. 
\end{proof}

Note that Corollary \ref{hairbrushTwoTubesInFatRegulusOne} is only meaningful if $\theta$ is much smaller than one. In particular, if $\theta\sim 1$ then $R(\tube_1,\tube_2)$ could be any regulus containing the lines coaxial with $\tube_1$ and $\tube_2$ and the corollary would hold.
%
%
%
%
\subsection{The Hair on $H(\tube_1,\tube_2)$ is spiky}
In the previous section, we showed that $H(\tube_1,\tube_2)$ is contained in the $\theta$--neighborhood of a regulus. In this section, we will show that most tubes in the remainder of the Besicovitch set intersect this regulus transversely. A precise statement is given in Lemma \ref{PlaneMapTransverseToHairbrushOfTwoTubes}.

\begin{defn}
Let $\tube_1$ and $\tube_2$ be tubes. We say that the hairbrush of $\tube_1$ covers $\tube_2$ (with accuracy $s$) if
$$
\Big|Y(\tube_2)\cap \bigcup_{\tube\in H(\tube_1)} Y(\tube)\Big|\geq s|\tube_2|.
$$
\end{defn}
The following lemma is a simple consequence of Theorem \ref{WolffBd}.
\begin{lem}\label{lotsOfCoveringPairs}
Let $(\tubes,Y)$ be an $\eps$-extremal set of tubes. Then there is a constant $C$ so that there are $\gtrapprox_{\eps}|\tubes|^2$ pairs $(\tube_1,\tube_2)\in\tubes$ so that $\tube_1$ and $\tube_2$ are $\geq\delta^{C\eps}$ separated and $\tube_1$ covers tube $\tube_2$ with accuracy $\delta^{C\eps}$.
\end{lem}

Note that if $\tube_1$ and $\tube_2$ are $\geq\delta^{\eps}$ separated and skew, and if the hairbrush of $\tube_1$ covers $\tube_2$, then there is a subset $H^\prime(\tube_1,\tube_2)\subset H(\tube_1,\tube_2)$ of cardinality $\gtrapprox_{\eps}\delta^{-1}$ so that the sets $\{\tube\cap\tube_2\colon \tube\in H^\prime(\tube_1,\tube_2)\}$ intersect with multiplicity $\lesssim 1$. Note that $H^\prime(\tube_1,\tube_2)$ need not be the same as $H^\prime(\tube_2,\tube_1)$.
\begin{defn}\label{definedHairbrushDef}
We will call $H^\prime(\tube_1,\tube_2)$ the refined hairbrush of $\tube_1$ and $\tube_2$.
\end{defn}

The next lemma says that the refined hairbrush $H^\prime(\tube_1,\tube_2)$ is evenly spread out over the $\theta$--neighborhood of the regulus $R(\tube_1,\tube_2)$.

\begin{lem}\label{coveringRT1T3ByRhoBalls}
Let $(\tubes,Y)$ be an $\eps$-extremal set of tubes that obey the regulus map at scale $\theta$. Let $\tube_1,\tube_2\in\tubes$ be tubes whose coaxial lines are $\delta^{C\eps}$ separated and skew. Suppose that the hairbrush of $\tube_1$ covers $\tube_2$ with accuracy $\delta^{C\eps}$. Let $H^\prime(\tube_1,\tube_2)\subset H(\tube_1,\tube_2)$ be the refined hairbrush of $\tube_1$ and $\tube_2$.

Then for each $\theta\leq\rho\leq 1$, $R(\tube_1,\tube_2)$ can be covered by $\sim\rho^{-2}$  boundedly overlapping $\rho$--balls. For each of these balls $B$, we have
$$
\sum_{\tube\in H^\prime(\tube_1,\tube_2)}|Y(\tube)\cap B|\lessapprox_{\eps} \rho^{2}\sum_{\tube\in H^\prime(\tube_1,\tube_2)}|Y(\tube)|\approx_{\eps}\rho^{2}\delta.
$$
\end{lem}
\begin{proof}
By Lemma \ref{thetaScaleHairbrush}, $H^\prime(\tube_1,\tube_2)$ can be covered by $\approx_{\eps} \rho^{-1}$ $\rho$--tubes that are $\lessapprox_{\eps}1$ overlapping.

Since the sets $\{\tube\cap \tube_2\colon \tube\in H^\prime(\tube_1,\tube_2)\}$ are $\lesssim 1$ overlapping, at most $\lessapprox_{\eps}\rho\delta^{-1}$ of the tubes from $H^\prime(\tube_1,\tube_2)$ can be contained in each $\rho$--tube. Furthermore, if we cover each $\rho$--tube with $\rho^{-1}$ boundedly overlapping balls of radius $\rho$, then for each such ball $B$, we have
\begin{equation*}
\Big|B\cap \bigcup_{\tube\in H^\prime(\tube_1,\tube_2)}Y(\tube)\Big|\lessapprox_\eps\rho (\rho/\delta)(\delta^2) = \rho^2\delta.
\qedhere
\end{equation*}


\end{proof}
\begin{lem}[Geometric lemma]\label{geomLem}
Let $L$ be a line segment of length $\leq 1$ contained in the unit ball. Let $Q\in\RR[x,y,z]$ be a degree-two polynomial. Let $Z \subset Z(Q)\cap B(0,1)$. Suppose that $Z$ is smooth and that the Gauss curvature satisfies $\delta^{\eps}\leq|K_p|\lesssim 1$ for each $p\in Z$. Let $t>0$, and define
$$
Z_{\operatorname{tang}}(t)=\{p\in Z\colon \angle(L, T_pZ)\leq t\}.
$$
The for each $0<\rho<t$
$$
\mathcal{E}_\rho(Z_{\operatorname{tang}}(t))\lessapprox_{\eps}t/\rho^2,
$$
where $\mathcal{E}_\rho(\cdot)$ denotes the $\rho$-covering number
\end{lem}
\begin{proof}
This follows immediately from the requirement that $|K_p|\geq\delta^{\eps}$ for all $p\in Z$, and the observation that $Q$ has degree two (which gives us control over the higher derivatives of the Gauss map $p\mapsto T_pZ$).
\end{proof}

\begin{lem}\label{linesMakingSmallAngleWithRegulusTangPlane}
Let $(\tubes,Y)$ be an $\eps$-extremal set of tubes that obey the regulus map at scale $\theta<\delta^{C_0\eps}$. Let $\tube_1,\tube_2\in\tubes$; suppose that $\tube_1$ and $\tube_2$ are $\geq \delta^{C\eps}$ separated and skew, and that the hairbrush of $\tube_1$ covers $\tube_2$ with accuracy $\delta^{C\eps}.$ Let $Z = R(\tube_1,\tube_2)\cap B(0,1)$. Let $L$ be a line, let $t>\theta$, and let $p\in Z\backslash Z_{\operatorname{tang}}(t)$ with $\operatorname{dist}(p,L)\geq\delta^{C_1\eps}$.

Let $q=L\cap T_pR(\tube_1,\tube_2)$ (this is well-defined since $p\not\in Z_{\operatorname{tang}}(t)$, and thus $L$ is not parallel to $R(\tube_1,\tube_2)$), and let $v$ be the vector $q-p$. Let $s>\theta$, and let $\ell$ be a line that intersects $B(p,\theta)$; intersects $L$; and satisfies 
\begin{equation}\label{ellMustSatisfy}
\angle(\ell, T_pR(\tube_1,\tube_2))\leq s.
\end{equation}
If $C_0$ is chosen sufficiently large compared to the constants $C$ and $C_1$ above, then
$$
\angle(\ell,v)\lessapprox_{\eps} s/t,
$$
where the implicit constant depends on $C$ and $C_1$, but not on $C_0$
\end{lem}
\begin{proof}
First, we can assume that $C_0$ is sufficiently large so that $\operatorname{dist}(B(p,\theta), L)\geq \delta^{C_1\eps}/2$. Let $\ell$ be a line that intersects $B(p,\theta)$; intersects $L$; and satisfies \eqref{ellMustSatisfy}. Then $\ell\cap B(0,1)$ is contained in a rectangular prism of dimensions $1\times 1\times (s+\theta)$ (recall that $s>\theta$). This prism is comparable to the $s$--neighborhood of $T_pZ\cap B(0,1)$. Since $p\not\in Z_{\operatorname{tang}}(t)$, we have $\angle(L, T_pZ)\geq t$. Recall as well that $\operatorname{dist}(p, L)\geq\delta^{C_1\eps}$. Thus the aforementioned prism intersects $L$ in an interval $I$ of length $\lessapprox_{\eps}s/t$, where the implicit constant depends on $C_1$. Finally, since $\ell$ intersects $B(p,\theta)$ and intersects $I$, and since $\operatorname{dist}(B(p,\theta), I)\geq\delta^{C_1\eps}/2$, we have that $\angle(\ell, v)\lessapprox_{\eps}s/t$, as desired.
\end{proof}
\begin{lem}\label{PlaneMapTransverseToHairbrushOfTwoTubes}
Let $(\tubes,Y)$ be an $\eps$-extremal set of tubes that obey the regulus map at scale $\theta<\delta^{C_0\eps}$. Then if $C_0$ is sufficiently large, there is a pair of tubes $\tube_1,\tube_2\in\tubes$ so that the following holds
\begin{itemize}
\item $\tube_1$ and $\tube_2$ are $\gtrapprox_{\eps}1$ separated and skew.
\item The hairbrush of $\tube_1$ covers $\tube_2$ with accuracy $\gtrapprox_{\eps}1$.
\item There is a constant $C$ so that
\begin{equation}\label{manyTransverseTubes}
\sum_{\tube\in H^\prime(\tube_1,\tube_2)}\big|\{\tube^\prime\in H(\tube)\colon \angle(v(\tube^\prime), T_pR(\tube_1,\tube_2)\geq\delta^{C\eps}\ \textrm{for all}\ p\in \tube^\prime\cap R(\tube_1,\tube_2)\}\big|\gtrapprox_{\eps}\delta^{-5/2}.
\end{equation}
See Figure \ref{lemma312_fig} for the relationship between $\tube_1,\tube_2,\tube,$ and $\tube^\prime$.
\end{itemize}

\end{lem}
\begin{figure}[h!]
 \centering
\begin{overpic}[width=0.6\textwidth ]{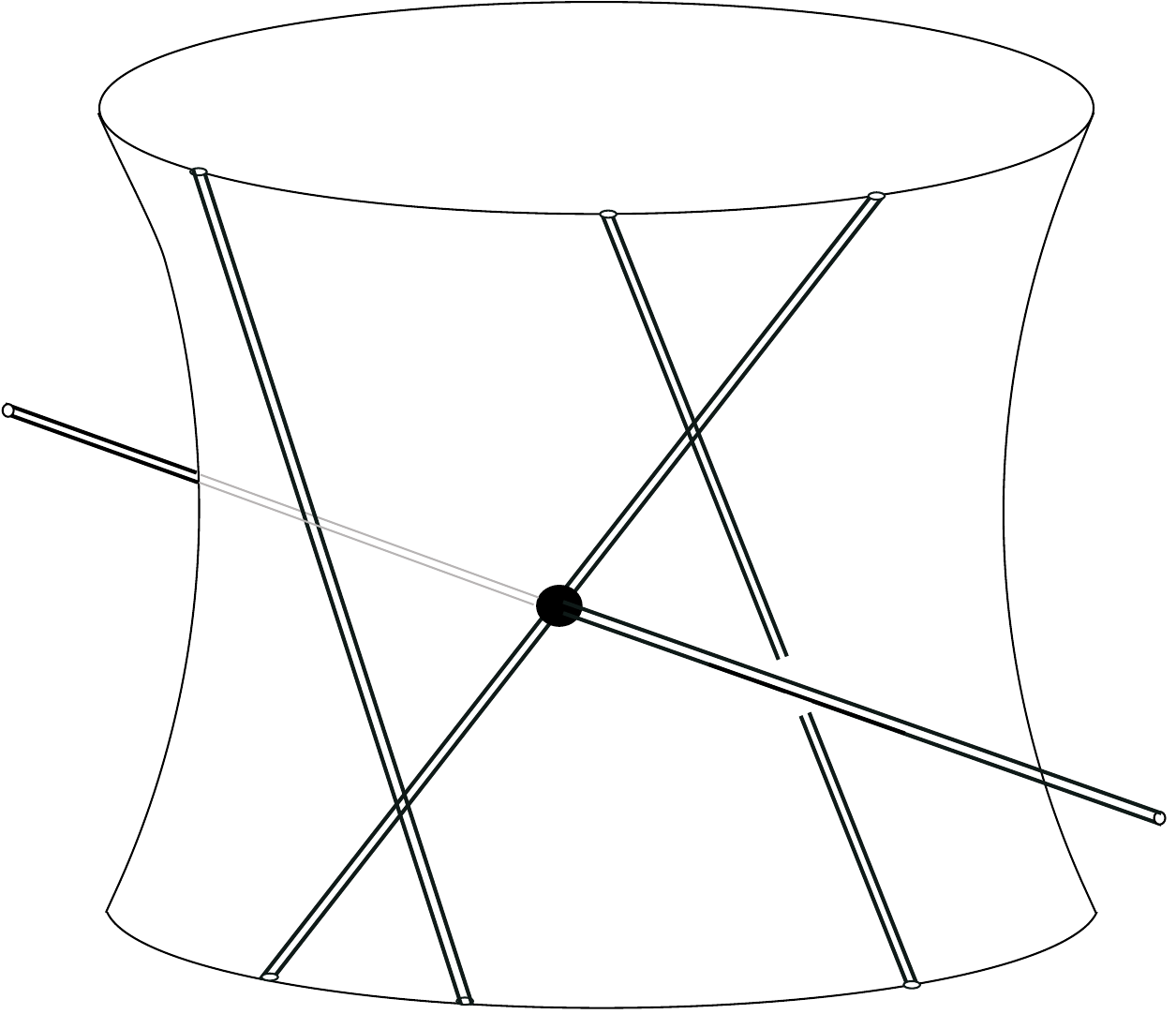}
\put (20,6) {$\tube$} 
\put (40,4) {$\tube_1$} 
\put (70,5) {$\tube_2$} 
\put (80,25) {$\tube^\prime$} 
\put (45,39) {$p$}
\put (45, 75) {$R(\tube_1,\tube_2)$}
 
\end{overpic}
 \caption{The relationship between $\tube_1,\tube_2,\tube$, and $\tube^\prime$.}\label{lemma312_fig} 
\end{figure}

\begin{proof}
Let $\tube_0\in\tubes$ with $\Big|\bigcup_{\tube\in H(\tube_0)}Y(\tube)\Big|\gtrapprox_{\eps}\delta^{1/2}$. For each $\tube\in\tubes,$ define $Y_1(\tube)=Y(\tube)\cap\bigcup_{\tube^\prime\in H(\tube_0)}Y(\tube^\prime)$. Let $Y_2$ be a refinement of $Y_1$ so that $\sum_{\tube\in\tubes}\chi_{Y_2(\tube)}\sim\mu\chi_A$ for some $\mu\approx_{\eps}\delta^{-1/2}$. Note that $(\tubes, Y_2)$ is still an $\eps$-extremal set of tubes that obey the regulus map at scale $\theta$.

Using Lemma \ref{lotsOfCoveringPairs} along with standard pigeonholing arguments, we can find tubes $\tube_1$ and $\tube_2$ so that $\tube_1$ covers $\tube_2$ with accuracy $\gtrapprox_{\eps}1$ and
\begin{equation}\label{largeJointHairbrushEst}
\sum_{\tube\in H^\prime_Y(\tube_1,\tube_2)}|Y_2(\tube)|\gtrapprox_{\eps}\delta.
\end{equation}
In particular, this means $\sum_{\tube\in H^\prime_Y(\tube_1,\tube_2)}|H_{Y_2}(\tube)|\gtrapprox_{\eps}\delta^{-5/2}$.

By Lemma \ref{coveringRT1T3ByRhoBalls}, we can cover $R(\tube_1,\tube_2)$ by $\approx_{\eps}\theta^{-2}$ $\theta$-balls, and each of these balls $B$ satisfies 
\begin{equation}\label{ballContribution}
\sum_{\tube\in H^\prime_Y(\tube_1,\tube_2)}|B\cap Y_2(\tube)|\lessapprox_{\eps}\theta^2\delta.
\end{equation} 

Since at most $\delta^{C_1\eps}\theta^{-2}$ of these balls can satisfy $\operatorname{dist}(B,\ \tube_0)<\delta^{C_1\eps}$, if we select $C_1$ sufficiently large then
$$
\sum_{\substack{B\\ \operatorname{dist}(B,\ \tube_0)\geq\delta^{C_1\eps}}}\sum_{\tube\in H^\prime_Y(\tube_1,\tube_2)}|B\cap Y_2(\tube)|\gtrapprox_{\eps}\delta.
$$
Fix this value of $C_1$. Let $Z=R(\tube_1,\tube_2)\backslash N_{\delta^{C_1\eps}(\tube_0)}\cap B(0,1);$ thus $N_{\theta}(Z)$ contains each of the balls in the above sum. By Lemma \ref{geomLem}, if we choose $C_2$ sufficiently large, then if we set $t = \delta^{C_2\eps}$, then
$$
\sum_{\substack{B\\ 
\operatorname{dist}(B,\ \tube_0)\geq\delta^{C_1\eps}\\
B\cap Z_{\operatorname{tang}(t)=\emptyset}
}}\sum_{\tube\in H^\prime_Y(\tube_1,\tube_2)}|B\cap Y_2(\tube)|\gtrapprox_{\eps}\delta.
$$

Finally, we will apply Lemma \ref{linesMakingSmallAngleWithRegulusTangPlane} and the fact that the tubes passing through each point of $A$ satisfy the robust transversality condition from Lemma \ref{robustTransLem}. Select $C$ to be a sufficiently large constant and let $s = \delta^{C\eps}$. If we choose $C_0$ sufficiently large, then for each of the balls $B$ in the above sum; for each $\tube\in H^\prime_Y(\tube_1,\tube_2)$ with $B\cap Y_2(\tube);$ and for each $x\in B\cap Y_2(\tube)$, at most half the tubes $\tube^\prime\in H_{Y_2}(\tube)$ with $x\in Y_2(\tube^\prime)$ satisfy
$$
\angle(\tube^\prime,\ T_p Z)\leq\delta^{C\eps}\quad \textrm{for some}\ p\in \tube^\prime\cap Z.
$$
With this choice of $C$ (which depends only on $C_1$ and $C_2$, which in turn are absolute constants (independent of $C_0$) ), we have that
$$
\sum_{\substack{B\\ 
\operatorname{dist}(B,\ \tube_0)\geq\delta^{C_1\eps}\\
B\cap Z_{\operatorname{tang}(t)=\emptyset}
}}\sum_{\tube\in H^\prime_Y(\tube_1,\tube_2)}|\{\tube^\prime\in Y_2(\tube)\cap B\colon \angle(\tube^\prime,\ T_p Z)\geq\delta^{C\eps}\quad \textrm{for all}\ p\in \tube^\prime\cap Z |\gtrapprox_{\eps}\delta^{-5/2},
$$
and thus
$$
\sum_{\tube\in H^\prime_Y(\tube_1,\tube_2)}|\{\tube^\prime\in Y_2(\tube)\colon \angle(\tube^\prime,\ T_p Z)\geq\delta^{C\eps}\quad \textrm{for all}\ p\in \tube^\prime\cap R(\tube_1,\tube_2) |\gtrapprox_{\eps}\delta^{-5/2}.
$$
Thus the tubes $\tube_1$ and $\tube_2$ satisfy the conclusions of the lemma. 
\end{proof}

\begin{rem}\label{lotsOfProperlyIntersectingQuadruples}
The above proof shows that not only does there exist one pair $\tube_1,\tube_2\in\tubes$ satisfying the conclusions of the lemma, but there actually exists $\gtrapprox_{\eps}\delta^{-4}$ such pairs. Thus, we have
\begin{equation}
\begin{split}
|\{&(\tube,\tube_1,\tube_2,\tube_3)\colon \tube_1,\tube_2,\tube_3\ \textrm{are}\ \gtrapprox_{\eps}1\ \textrm{separated and skew},\\ &\tube\in H(\tube_1,\tube_2,\tube_3),\ \angle(v(\tube_3),\ T_p R(\tube_1,\tube_2)\geq\delta^{C\eps}\ \textrm{for all}\ p\in \tube_3\cap R(\tube_1,\tube_2)\} |\gtrapprox_{\eps}\delta^{-13/2}.
\end{split}
\end{equation}
This observation will be used in Section \ref{incidenceGeomRegStripsSec}.
\end{rem}

\subsection{Regulus maps talk to each other}
Let $(\tubes,Y)$ be an $\eps$-extremal set of tubes that obey the regulus map at scale $\theta$. Let $\tube_1,\tube_2\in\tubes$ be two tubes that are $\geq\delta^{C\eps}$ separated and skew and suppose that the hairbrush $\tube_1$ covers $\tube_2$ with accuracy $\delta^{C\eps}$. One of the consequences of Lemma \ref{thetaScaleHairbrush} is that if $L$ is a line that intersects $\tube_1$ and $\tube_2$ at points $p$ and $q$, respectively, and if $\angle(L, T_pR(\tube_1))\lessapprox_{\eps}\theta$, then
\begin{equation}\label{regulusMapsSynchronized}
\angle(L,T_qR(\tube_2))\lessapprox_{\eps}\theta.
\end{equation}


The next lemma is an elaboration of the above observation.
\begin{lem}\label{reguliiTalkingLem}
Let $(\tubes,Y)$ be an $\eps$-extremal set of tubes that obey the regulus map at scale $\theta$. Let $\tube_1,\tube_2\in\tubes$ be two tubes whose coaxial lines are $\geq\delta^{C\eps}$ separated and skew and suppose that the hairbrush $\tube_1$ covers $\tube_2$ with accuracy $\delta^{C\eps}$. Let $p\in \tube_1$ and let $L$ be a line passing through $p$ and intersecting $\tube_2$. Let $q$ be a point of intersection of $L$ and $\tube_2$. Then
\begin{equation}\label{anglesTalkToEachOther}
\begin{split}
\angle(L,\ T_pR(\tube_1))&\lessapprox_\eps \theta+\angle(L,\ T_q R(\tube_2)),\\
\angle(L,\ T_qR(\tube_1))&\lessapprox_\eps \theta+\angle(L,\ T_p R(\tube_2)).
\end{split}
\end{equation}
Furthermore,
\begin{equation}\label{reguliiAlmostTalking}
L\cap B(0,1)\subset N_w(R(\tube_1,\tube_2)),
\end{equation}
where
$$
w\lessapprox_\eps \angle(L,\ T_pR(\tube_1))+\theta.
$$
\end{lem}
\begin{proof}
Let $L^\prime$ be a line passing through $q$ that intersects $\tube_1$ and satisfies $\angle(v(L^\prime),\ T_qR(\tube_2))\leq\theta$. Since $\tube_1$ and $\tube_2$ are $\gtrapprox_{\eps}1$ separated and skew, the angle between $L$ an $L^\prime$ is comparable to the angle between the planes spanned by $\tube_2,L$ and $\tube_2,L^\prime$. In particular, $\angle(v(L), v(L^\prime))\lessapprox_{\eps}\angle(v(L), T_qR(\tube_2))+\theta$. Since $\dist(p, L^\prime\cap T_1)\lesssim \angle(v(L), v(L^\prime))$, we have
\begin{equation}\label{distanceAngleRelation}
\dist(p, L^\prime\cap T_1)\lessapprox_{\eps}\angle(v(L),\ T_qR(\tube_2))+\theta.
\end{equation}

By \eqref{regulusMapsSynchronized}, we have $\angle(v(L^\prime),\ T_rR(\tube_1))\lessapprox_{\eps}\theta$ whenever $r$ is a point of intersection of $L^\prime$ and $\tube_1$. Thus

 \begin{equation*}
 \begin{split}
 \angle(L,\ T_p R(\tube_1)) &\leq \angle(v(L^\prime), T_rR(\tube_1))+\angle(T_pR(\tube_1), T_rR(\tube_1)) \\
 &\lessapprox_\eps\theta + \dist(q, r)\\
 &\lessapprox_\eps \theta + \angle(v(L), T_qR(\tube_2)),
 \end{split}
 \end{equation*}
 where the second inequality follows from the fact that if we identify the line coaxial with $\tube_1$ with the interval $[0,1]$ (we will use the variable $x$ to denote points in this interval), then 
 $$
\frac{d}{dx}\angle(T_{x_0}R(\tube_1), T_x(R(\tube_1))\approx_{\eps} 1.
$$

This gives us the first inequality of \eqref{anglesTalkToEachOther}. The second inequality follows from the fact that $\tube_1$ and $\tube_2$ play symmetric roles.

It remains to establish \eqref{reguliiAlmostTalking}. Recall that the lines $L$ and $L^\prime$ described above both pass through the point $q$, and their points of intersection with $\tube_1$ are $\lessapprox_\eps \angle(L,\ T_pR(\tube_1))+\theta$ separated. Since $\tube_1$ and $\tube_2$ are $\geq\delta^{C\eps}$ separated, the points of intersection $L\cap \tube_1$ and $L^\prime\cap \tube_1$ have separation $\gtrapprox_{\eps}1$ from $q$. This implies that $L^\prime\cap B(0,1)$ is contained in the $\lessapprox_\eps \angle(L,\ T_pR(\tube_1))+\theta$ neighborhood of $L$. Equation \eqref{reguliiAlmostTalking} follows from the fact that $L$ is contained in the $\lessapprox_{\eps}\theta$ neighborhood of $R(\tube_1,\tube_2)$.
\end{proof}

\subsection{The regulus map and the hairbrush of two tubes}
Let $(\tubes,Y)$ be an $\eps$-extremal set of tubes that obey the regulus map at scale $\theta$. Let $\tube_1,\tube_2\in\tubes$ be two tubes that are $\geq\delta^{\eps}$ separated and skew. Recall that the hairbrush $H(\tube_1,\tube_2)$ is contained in the $\lessapprox_{\eps}\theta$--neighborhood of the regulus $R(\tube_1,\tube_2)$. Let $\tube\in H(\tube_1,\tube_2)$, and let $\tube_3\in H(\tube)$. 
At each point $p\in \tube_3$, there are two interesting planes containing $p$: the tangent plane $T_p R(\tube_3)$ arising from the regulus map of $\tube_3$, and the tangent plane $T_p(R_{\tube_1,\tube_2,\tube_3})$ arising from the regulus determined by the lines coaxial with $\tube_1,\tube_2$, and $\tube_3$.

If $p$ is a point of intersection of $\tube_3$ and $R(\tube_1,\tube_2)$, then we expect $T_p R(\tube_3)$ and $T_p(R_{\tube_1,\tube_2,\tube_3})$ to be equal---this is a consequence of planiness. If $p$ is not a point of intersection of $\tube_3$ and $R(\tube_1,\tube_2)$, however, then there is no reason to expect that the two planes should be equal. The following lemma asserts that indeed, they are not.

\begin{lem}\label{decoherenceOfPlaneMap}
Let $(\tubes,Y)$ be an $\eps$-extremal set of tubes that obey the regulus map at scale $\theta$. Let $\tube_1,\tube_2\in\tubes$ be two tubes whose coaxial lines are $\geq\delta^{C\eps}$ separated and skew. Suppose that the hairbrush of $\tube_1$ covers $\tube_2$ with accuracy $\delta^{C\eps}$. Let $\tube\in H(\tube_1,\tube_2)$ and let $\tube_3\in H(\tube)$ be a tube whose hairbrush covers each of $\tube_1$ and $\tube_2$ with accuracy $\delta^{C\eps}$; whose coaxial line is $\geq\delta^{C\eps}$ separated and skew to the coaxial lines of $\tube_1$ and $\tube_2$; and which satisfies the estimate
\begin{equation}\label{largeIntersectionAngleT3RT1T2}
\angle(v(\tube_3),\ T_pR(\tube_1,\tube_2))\geq\delta^{C\eps}\ \textrm{for all}\ p\in \tube_3\cap R(\tube_1,\tube_2).
\end{equation}

Then there exist two points $p_1,p_2\in \tube_3$, so that if $p\in \tube_3$, then
\begin{equation}\label{compareAngleDistance}
\angle(T_pR(\tube_3),\ T_pR_{\tube_1,\tube_2,\tube_3})\gtrapprox_{\eps} \min\big(\dist(p,\ p_1),\ \dist(p,\ p_2)\big)-\theta.
\end{equation}
\end{lem}
\begin{rem}
Note that (as discussed in Remark \ref{threeParametersRegulusMap}) the tangency information from the regulus map $p\mapsto T_pR(\tube_3)$ is defined by three parameters (and is defined up to uncertainty $\theta$). If the point $p\in Y(\tube_3)$ is also contained in some tube from $H(\tube_1,\tube_2,\tube_3)$, then as mentioned above,
$$
\angle\big(T_pR(\tube_3),\ T_pR_{\tube_1,\tube_2,\tube_3}\big)\lessapprox_\eps\theta.
$$
The tubes from $H(\tube_1,\tube_2,\tube_3)$ intersect $\tube_3$ in at most two intervals of length $\lessapprox_\eps\theta$. This forces the regulus map $p\mapsto T_pR(\tube_3)$ to agree with the plane map $p\mapsto T_pR_{\tube_1,\tube_2,\tube_3}$ at two places. But since $R(\tube_3)$ is determined by three parameters, we do not expect the two maps to agree for all points $p\in \tube_3$.
\end{rem}
\begin{proof}[Proof of Lemma \ref{decoherenceOfPlaneMap}]
Let $p_1,p_2$ be the two points of intersection of $\tube_3\cap R(\tube_1,\tube_2)$ (these points are defined up to uncertainty $\approx_\eps\theta$). Let $L_3$ be the line coaxial with $\tube_3$ and let $p\in L_3$.

Let $L$ be the (unique) line passing through $p$ that intersects the lines coaxial with $\tube_1$ and $\tube_2$. Applying Lemma \ref{reguliiTalkingLem} to $\tube_3$ and $\tube_1$ and then to $\tube_3$ and $\tube_2$ (using the line $L$), we conclude that if $q_i$ is a point of intersection of $L$ and $\tube_i,\ i=1,2$, then
$$
\angle(L, T_{q_i} R(\tube_i))\lessapprox_{\eps}\theta+ \angle\big(L, T_pR(\tube_3)\big).
$$
On the other hand, since $L$ is a line in the regulus $R_{\tube_1,\tube_2,\tube_3}$, we have 
$$
\angle\big(L, T_pR(\tube_3)\big)\leq\angle\big(T_pR(\tube_3),\ T_pR_{\tube_1,\tube_2,\tube_3}\big),
$$ 
and thus 
$$
\angle(L, T_{q_i} R(\tube_i))\lessapprox_{\eps}\theta+ \angle\big(T_pR(\tube_3),\ T_pR_{\tube_1,\tube_2,\tube_3}\big).
$$
By \eqref{reguliiAlmostTalking} from Lemma \ref{reguliiTalkingLem}, we have that $L\cap B(0,1)$ is contained in the
$$
\lessapprox_\eps \theta + \angle\big(T_pR(\tube_3),\ T_pR_{\tube_1,\tube_2,\tube_3}\big)
$$
neighborhood of $R(\tube_1,\tube_2)$, and thus (since $p\in L\cap B(0,1)$), we have
\begin{equation}\label{distanceBd}
\dist(p,\ R(\tube_1,\tube_2))\lessapprox_\eps \theta + \angle\big(T_pR(\tube_3),\ T_pR_{\tube_1,\tube_2,\tube_3}\big).
\end{equation}
However by \eqref{largeIntersectionAngleT3RT1T2} we have
\begin{equation}\label{distComparablep1p2}
\dist(p,\ R(\tube_1,\tube_2))\gtrapprox_\eps \min\big(\dist(p, p_1),\ \dist(p,p_2)\big).
\end{equation}
Combining \eqref{distanceBd} and \eqref{distComparablep1p2} we obtain \eqref{compareAngleDistance}.
\end{proof}

\section{Dichotomy: Heisenberg or $SL_2$}\label{dichotomySection}
\begin{defn}[Heisenberg]
Let $(\tubes,Y)$ be a set of tubes. We say that $(\tubes,Y)$ is of \emph{Heisenberg type} (with parameter $\alpha$ and $\eps$) if at most $\delta^{-1/2+\alpha}$ tubes from $\tubes$ are contained in any $\delta^{\eps}$-regulus strip.
\end{defn}

\begin{defn}[$SL_2$]
Let $(\tubes,Y)$ be a set of tubes. We say that $(\tubes,Y)$ is of \emph{$SL_2$ type} (with parameter $\alpha$ and $\eps$) if we can write $\tubes=\bigsqcup_{S\in\mathcal{S}}\tubes(S)$; here each $S\in\mathcal{S}$ is a $\delta^{\eps}$-regulus strip that contains $\geq\delta^{-1/2+\alpha}$ tubes.
\end{defn}

\begin{rem}
If $(\tubes,Y)$ is an $\eps$--extremal set of tubes, we will abuse notation and say that $\tubes$ is of Heisenberg (resp. $SL_2$) type with parameter $\alpha$ if it is of Heisenberg (resp. $SL_2$) type with parameters $\alpha$ and $\epsilon$.
\end{rem}

\begin{lem}[Decomposing a Besicovitch set into Heisenberg + $SL_2$]\label{HeisenbergPlusSl2Lem}
Let $\epsilon>0$ and $\alpha>0$ and let $(\tubes,Y)$ be a set of $\delta$ tubes. Then we can write $\tubes=\tubes_1\sqcup\tubes_2$, where $\tubes_1$ is of Heisenberg type (with parameters $\alpha$ and $\eps$), and $\tubes_2$ is of $SL_2$ type (with parameters $\alpha$ and $\eps$).
\end{lem}
\begin{proof}
Define $\tubes_1^{(0)}=\tubes$ and define $\tubes_2^{(0)}=\emptyset$. Assuming that $\tubes_1^{(j-1)}$ and $\tubes_2^{(j-1)}$ have already been defined, we proceed as follows. If there exists a $\delta^{\eps}$ regulus strip that contains at least $\delta^{-1/2+\alpha}$ tubes from $\tubes_1^{(j-1)}$, then remove these tubes from $\tubes_1^{(j-1)}$ and place them in $\tubes_2^{(j-1)}$. Call the resulting sets $\tubes_1^{(j)}$ and $\tubes_2^{(j)}$. If no such set exists, then define $\tubes_1=\tubes_1^{(j-1)}$, define $\tubes_2=\tubes_2^{(j-1)}$, and halt the process. Since $|\tubes_1^{(j)}|\leq|\tubes|-j\delta^{-1/2+\alpha}$, this process must halt after finitely many steps.
\end{proof}
\begin{rem}\label{ifExtremalOneOfT1T2Extremal}
If we apply Lemma \ref{HeisenbergPlusSl2Lem} to an $\eps$--extremal set of tubes, then at least one of $\tubes_1$ or $\tubes_2$ must also be $\eps$--extremal.
\end{rem}
\section{Killing the Heisenberg example}
In this section we will show that an $\eps$--extremal set of tubes of Heisenberg type cannot exist. More precisely, we have the following.
\begin{prop}\label{killingHeisenbergProp}
There exist absolute constants $C$ (large) and $c>0$ (small) so that the following holds. For every $\alpha>0$, there is a $\delta_0>0$ so that if $0<\delta\leq\delta_0$, and if $(\tubes,Y)$ is an $\eps$--extremal collection of $\delta$ tubes that is of Heisenberg type with parameter $\alpha$ and $C\eps$, then $\eps>c\alpha$.
\end{prop}

Our proof will loosely follow the strategy used in \cite{BKT}. Namely, we will show that if an $\eps$--extremal Besicovitch set of Heisenberg type exists, then it is possible to construct an arrangement of points and lines in the plane that determine many point-line incidences. We will then use a variant of Bourgain's discretized sum-product theorem to show that such an arrangement is impossible.

\subsection{Finding the points and lines}
Recall that a line in $\RR^3$ is described by four parameters. If we fix two lines $L_1,L_2$, then the set of lines incident to both $L_1$ and $L_2$ can be described by two parameters; at least heuristically, a line incident to $L_1$ and $L_2$ can be identified with a corresponding point in the plane. In \cite{BKT}, Bourgain, Katz, and Tao make this heuristic precise---working over $\FP$ rather than $\RR$, they exhibited a transformation that sends lines incident to $L_1$ and $L_2$ to points in the plane, and they then constructed an incidence arrangement using these points.

In general, we would expect a similar phenomena to hold if the lines $L_1$ and $L_2$ are replaced by tubes $\tube_1$ and $\tube_2$. If the tubes come from an $\eps$-extremal Besicovitch set that obeys the regulus map at scale $\theta$, however, then Corollary \ref{hairbrushTwoTubesInFatRegulusOne} asserts that at scale $\theta$, the set of tubes incident to both $\tube_1$ and $\tube_2$ lie in the $\theta$ neighborhood of a regulus, and thus are described by \emph{one} parameter (i.e.~the location of the tube along the ruling of the regulus). Thus if one were to map tubes incident $\tube_1$ and $\tube_2$ to points in the plane, the resulting arrangement of points would lie close to a low degree algebraic curve. Naively, one would expect that this situation is easier to analyze, but it actually closely resembles the ``$SL_2$ example'' discussed in the introduction, which is an almost-counter-example to the Kakeya conjecture.

In this section, we will show that if the set of tubes is of Heisenberg type, then at scales smaller than $\theta^2$, the tubes incident to both $\tube_1$ and $\tube_2$ are ``truly described'' by two parameters, in the sense that they cannot be closely approximated by a simple one-dimensional parameterization\footnote{Of course this statement is only interesting if $\theta^2$ is much larger than $\delta$; otherwise we obtain an incidence arrangement with only one point and one line}. A precise version of this statement is needed in order to apply Bourgain's discretized sum-product theorem.

In the following lemma, we will begin the process of constructing a planar point-line arrangement with many incidences. 

\begin{lem}\label{findingPtsAndLinesLem}
Let $(\tubes,Y)$ be a $\eps$--extremal set of tubes of Heisenberg type with parameter $\alpha$ and $C\eps$. Then if $C$ is sufficiently large, there is a number $\tilde\delta\leq\delta^{\alpha}$, a set $\pts\subset B(0,1)\subset \RR^2$ of points, and a set $\lines$ of lines in $\RR^2$ with the following properties:
\begin{itemize}
\item $|\pts|\approx_{\eps/\alpha}\tilde\delta^{-1}$, $|\lines|\approx_{\eps/\alpha}\tilde\delta^{-1}$.
\item For each ball $B(x,r)\subset \RR^2$,
\begin{equation}\label{nonConcentrationPropertyPts}
|\pts\cap B(x,r)|\lessapprox_{\eps/\alpha} r\tilde\delta^{-1}.
\end{equation}
\item For each $p_1,p_2\in\pts$ and each $1\leq N\leq \tilde\delta^{-1}$,
\begin{equation}\label{nonConcentrationPropertyLines}
|\{L\in\lines\colon \dist(L,p_i)\leq N\tilde\delta,\ i=1,2\}| \lessapprox_{\eps/\alpha} N/\dist(p_1,p_2).
\end{equation}
\item There is a constant $C_0$ so that for each $L\in\lines,$
\begin{equation}\label{manyIncidences}
|\{ (p_1,p_2)\in\pts^2\colon \dist(p_1,p_2)\geq \delta^{C_0\eps},\ \dist(L,p_i)\leq\tilde\delta,\ i=1,2\}|\gtrapprox_{\eps/\alpha}\tilde\delta^{-1}.
\end{equation}
\end{itemize}
\end{lem}
\begin{rem}
When proving point-line incidence theorems, we can use the Cauchy-Schwarz inequality and the observation that at most one line passes through each pair of points to show that $N$ points and $N$ lines determine at most $O(N^{3/2})$ incidences.

If the points and lines satisfy \eqref{nonConcentrationPropertyPts} and \eqref{nonConcentrationPropertyLines}, then a similar argument can be used to bound the number of $\tilde\delta$--close pairs of points and lines. Thus the lower bound from \eqref{manyIncidences} matches the upper bound given by Cauchy-Schwarz. As we will see later, a stronger upper bound actually holds, which leads to a contradiction.
\end{rem}
\begin{proof}
The proof of Lemma \ref{findingPtsAndLinesLem} differs based on whether or not $(\tubes,Y)$ obeys a regulus map at some scale $\theta$ that is much smaller than one. We will deal with these two possibilities by considering two cases, which are discussed below. Here $C_1$ and $C_2$ are absolute constants, which will be determined later in the proof. The reader should think of $C_1$ as being much larger than $C_2$. 

\begin{itemize}
\item Case (A): $(\tubes,Y)$ fails to avoid reguli with parameters $\gamma=C_1\eps$ and $\beta = C_2\eps$.
\item Case (B): $(\tubes,Y)$ avoids reguli with parameters $\gamma=C_1\eps$ and $\beta = C_2\eps$.
\end{itemize} 

First, we will show that there is a pair of tubes $\tube_1,\tube_2\in\tubes$ so that the lines coaxial with $\tube_1$ and $\tube_2$ are $\gtrapprox_{\eps}1$ separated and skew; the hairbrush of $\tube_1$ covers $\tube_2$ with accuracy $\gtrapprox_{\eps}1$; there exists a set $\tubes^\prime\subset\tubes$; a shading $Y^\prime$ of the tubes in $\tubes^\prime$; and a number $\delta^{1/2-\alpha}\lessapprox_{\eps} s\lessapprox_{\eps}1$ so that for each $\tube_3\in\tubes^\prime,$ the following properties hold:

\begin{itemize}
\item[(P1):] $\tube_1,\tube_2,\tube_3$ are pairwise $\approx_\eps 1$ separated and skew, and $\tube_3$ makes an angle $\approx_\eps 1$ with the plane spanned by the directions of $\tube_1$ and $\tube_2$.
\item[(P2):]  \itemizeEqnVSpacing
\begin{equation}\label{modifiedHairbrush1}
|H^\prime(\tube_1,\tube_2)\cap H_{Y^\prime}(\tube_3)|\approx_\eps\delta^{-5/2}|\tubes^\prime|^{-1}.
\end{equation}
\item[(P3):] All of the tubes in the set from \eqref{modifiedHairbrush1} are contained in the $s$--neighborhood of a line (the line depends on the choice of $\tube_3$).
\item[(P4):] The tubes in the set from \eqref{modifiedHairbrush1} cannot be contained in a smaller neighborhood of a line; more precisely, for each $0<r<1$, at most a $r^{\eps}$ fraction of the tubes are contained in the $rs$--neighborhood of a line.
\end{itemize}

We begin the process of selecting the tubes $\tube_1$ and $\tube_2$; the set $\tubes^\prime$, and the shading $Y^\prime$. 

If we are in Case (A), then apply Lemma \ref{regulusMapLem} to $(\tubes,Y)$; we obtain a set of tubes $(\tubes_1,Y_1)$ which obey the regulus map at some scale $\theta \leq\delta^{C_1\eps}$ and which satisfy
$$
\sum_{\tube\in\tubes_1}|Y_1(\tube)|\gtrapprox_{C_2\eps}1.
$$

Use Lemma \ref{lotsOfCoveringPairs} to select two tubes $\tube_1,\tube_2\in\tubes$ whose coaxial lines are $\geq\delta^{C_{(A)}\eps}$ separated and skew, so that the hairbrush of $\tube_1$ covers $\tube_2$ with accuracy $\delta^{C_{(A)}\eps}$ and there are $\geq \delta^{C_{(A)}\eps-5/2}$ pairs
\begin{equation}\label{tubeTube3Pairs}
\begin{split}
\{(\tube,\tube_3)\colon &\tube\in H^\prime(\tube_1,\tube_2),\ \tube_3\in H(\tube)\}.
\end{split}
\end{equation}
Here $C_{(A)}$ is a constant of the form $C_{(A)} = C^\prime C_2$, where $C^\prime$ is an absolute constant. 

Refining the above set slightly and increasing $C^\prime$ (and thus $C_{(A)}$) if necessary, we can also ensure that $\tube_1,\tube_2$ and $\tube_3$ are $\geq \delta^{C_{(A)}\eps}$ separated and skew; that $v(\tube_3)$ makes an angle $\geq\delta^{C_{(A)}\eps}$ with the plane spanned by $v(\tube_1)$ and $v(\tube_2)$ (i.e. the determinant of $v(\tube_1),\ v(\tube_2),$ and $v(\tube_3)$ has magnitude $\geq\delta^{C_{(A)}\eps}$); and that the hairbrush of $\tube_3$ covers each of $\tube_1$ and $\tube_2$ with accuracy $\geq\delta^{C_{(A)}\eps}$. In particular, this means that $R_{\tube_1,\tube_2,\tube_3}$ is a $\delta^{C_{(A)}\eps}$ non-degenerate regulus.

After pigeonholing, we can select a set $\tubes^{\prime}\subset\tubes_1$ so that each tube $\tube_3\in\tubes^{\prime}$ is present in $\gtrapprox_{C_{(A)}\eps}\delta^{-5/2}|\tubes^\prime|^{-1}$ pairs $(\tube,\tube_3)$ from \eqref{tubeTube3Pairs}. For each $\tube_3\in\tubes^{\prime}$, define the shading
$$
Y^\prime(\tube_3)=Y_1(\tube_3)\cap \bigcup_{\tube\colon (\tube,\tube_3)\in\eqref{tubeTube3Pairs} }Y_1(\tube).
$$
At this point, we have selected the tubes $\tube_1$ and $\tube_2$, the set $\tubes^\prime$, and the shading $Y^\prime$ if we are in Case (A).

Now suppose that we are in Case (B). Let $(\tubes_1,Y_1)$ be a set of tubes (and their associated shading) so that
\begin{equation}\label{caseBMostOfTubesRemain}
\sum_{\tube\in\tubes_1}|Y_1(\tube)|\gtrapprox_{\eps}1,
\end{equation}
and for each $\tube\in\tubes_1$ and each regulus $R$ containing the line coaxial with $\tube$, we have
\begin{equation}\label{avoidsReguli}
|\{\tube^\prime \in H_{Y_1}(\tube)\colon \angle(\tube^\prime, T_pR)\leq\delta^{C_1\eps}\ \textrm{for some}\ p\in \tube^\prime\cap R\}|\lessapprox_{\eps} \delta^{C_2\eps}|H_{Y_1}(\tube)|.
\end{equation}
Observe that the implicit constant in \eqref{caseBMostOfTubesRemain} is independent of $C_2$. This fact will prove crucial in later arguments.

Such a set $(\tubes_1,Y_1)$ is easy to find. For example, let $\tubes_1=\{\tube\in\tubes\colon |Y(\tube)|\gtrapprox_{\eps}1 \}$ and let $Y_1(\tube)=Y(\tube)$ for all $\tube\in\tubes_1$.

We will use Lemma \ref{lotsOfCoveringPairs} to select two tubes $\tube_1,\tube_2$, and then select the set $\tubes^\prime$ as in Case (A). However, by \eqref{caseBMostOfTubesRemain}, we can replace the constant $C_{(A)}$ by $C_{(B)}$, which is an absolute constant, independent of $C_2$. In particular, for each $\tube_3\in\tubes^\prime$, we have
$$
Y^\prime(\tube_3)\gtrapprox_{\eps}1.
$$
This means that \eqref{avoidsReguli} still holds for the shading $Y^\prime$ (with a larger implicit constant, but (crucially), one that is independent of $C_2$). At this point, we have selected the tubes $\tube_1$ and $\tube_2$, the set $\tubes^\prime$, and the shading $Y^\prime$ if we are in Case (B).

The next step of the proof will be the same regardless of whether we are in Case (A) or Case (B). We will apply a two-ends reduction, as illustrated in Figure \ref{twoEndsFig}. The following lemma was first used by Wolff in \cite{wolff}. A proof of the lemma as it appears below can also be found in \cite[Lemma 6]{T2}.
\begin{lem}\label{twoEndsLem}
Let $T$ be a tube and let $Y(T)\subset T$ be a shading with $|Y(T)|\geq\delta$. Let $0<\rho<1$. Then there is a ball $B(p_0,s)$ with $\delta\leq s\leq 1$ and 
$$
|Y(T\cap B(p_0,s))|\geq\delta^{\rho}|Y(T)|
$$ so that for all $p\in \RR^3$ and all $\delta\leq r\leq 1$ we have
\begin{equation}\label{consequenceOfTwoEnds}
|Y(B(p,r)\cap B(p_0,s))|\leq (r/s)^{\rho}|Y(B(p_0,s))|.
\end{equation}
\end{lem}

\begin{figure}[h!]
 \centering
\begin{overpic}[width=0.7\textwidth]{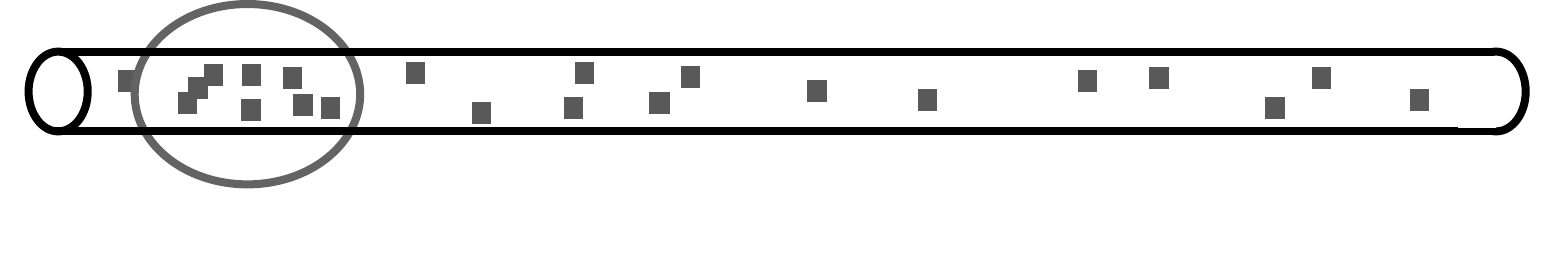}
 \put (10,1) {$B(p_0,s)$}
\end{overpic}
 \caption{The ball captures a maximal portion of the shading relative to its size}\label{twoEndsFig}
\end{figure}

Apply Lemma \ref{twoEndsLem} to each tube $\tube_3\in\tubes^{\prime}$ with the shading $Y^\prime(\tube_3)$ and exponent $\rho=\eps$, and let $B(p_{\tube_3},s_{\tube_3})$ be the resulting ball.
After pigeonholing (which induces a refinement of $\tubes^{\prime}$), we can assume that there is a number $s$ so that $s_{\tube_3} \sim s$ for all $\tube_3\in\tubes^\prime.$ 

Define the set of ``good pairs,''
\begin{equation}\label{defnOfGP}
GP=\{(\tube,\tube_3)\colon Y(\tube)\cap Y^\prime(\tube_3)\cap B(p_{\tube_3},s)\neq\emptyset\},
\end{equation}
where $B(p_{\tube_3},s)$ is the ball arising from the two-ends reduction. We have $|GP|\gtrapprox_{C_{(A|B)}\eps}\delta^{-5/2}$, where $C_{(A|B)}=C_{(A)}$ if we are in Case (A) and $C_{(A|B)}=C_{(B)}$ if we are in Case (B).

Since $(\tubes,Y)$ is of Heisenberg type with parameter $\alpha$ and $C\eps$, if we choose the constant $C$ from the statement of Proposition \ref{killingHeisenbergProp} to be larger than $\max(C_{(A)},C_{(B)})$, then at most $\delta^{-1/2+\alpha}$ tubes from $\tubes$ can be contained in any $\delta^{C(A|B)\eps}$ regulus strip. Note that $\gtrapprox_{C(A|B)\eps}\delta^{-1/2}$ tubes intersect each of $\tube_1,\tube_2$ and $\tube_3$ and are contained in the $s$ neighborhood of a line. We can cover this set of tubes by a set of $\delta^{C(A|B)\eps}$ non-degenerate regulus strips of cardinality $\lessapprox_{C(A|B)\eps}\delta^{-1/2}s$, so that each of these tubes is contained in at least one strip. Since at most $\delta^{-1/2+\alpha}$ tubes can be contained in any $\delta^{C(A|B)\eps}$ non-degenerate regulus strip, we conclude that $s\gtrapprox_{C(A|B)\eps}\delta^{1/2-\alpha}$. Once we have fixed the constant $C_2$, this means that $s\gtrapprox_{\eps}\delta^{1/2-\alpha}$. We have now selected the tubes $\tube_1,\tube_2$, the set $\tubes^\prime$, the shading $Y^{\prime}$, and the number $s$ that satisfies properties (P1) to (P4).


Since $|H^\prime(\tube_1,\tube_2)|\lessapprox_{\eps}\delta^{-1}$, we can use pigeonholing to select a tube $\tube_A\in H^\prime(\tube_1,\tube_2)$ with
$$
|\{\tube_3\in\tubes^\prime\colon \tube_A\in H_{Y^\prime}(\tube_3)\}|\gtrapprox_{\eps} \delta^{-3/2}.
$$

Observe that if $\tube_3\in\tubes^\prime,$ if $\tube\in H^\prime(\tube_1,\tube_2)$, and if $\tube_A,\tube\in H_{Y^\prime}(\tube_3)$, then both $\tube_A$ and $\tube$ are incident to the three tubes $\tube_1,\tube_2,\tube_3$, which are pairwise $\geq\delta^{C(A|B)\eps}$ separated and skew. By Lemma \ref{ThreeSkewLinesTwoTransversals}, $\tube_A$ and $\tube$ are uniformly separated (actually $s^\prime$ separated for some $s^\prime\leq s$) with error $\lessapprox_{C(A|B)\eps}1$ and $\approx_{C(A|B)\eps}1$ skew.

Let
\begin{equation}\label{defnTubes12}
\tubes_{12}=\{\tube\in H^\prime(\tube_1,\tube_2)\colon \tube\ \textrm{and}\ \tube_A\ \textrm{are}\ \geq\delta^{C_3\eps}\ \textrm{skew and}\ s\ \textrm{separated with error}\ \delta^{C_3\eps}\}.
\end{equation}

Since the tubes in $\tubes_{12}$ obey properties (P3) and (P4) from above, we have that that if $C_3$ is selected sufficiently large (depending on $C(A|B)$) then for each $\tube_3\in\tubes^\prime\cap H(\tube_A)$, all of the tubes in $H^\prime(\tube_1,\tube_2)\cap H_{Y^\prime}(\tube_3)$ are contained in $B_{\delta^{-C_3\eps}s}(\tube_A)$, and at most half of the tubes in $H^\prime(\tube_1,\tube_2)\cap H_{Y^\prime}(\tube_3)$ are contained in $B_{\delta^{C_3\eps}s}(\tube_A)$.

Thus if $C_3$ is selected sufficiently large in \eqref{defnTubes12}, then for each $\tube_3\in H(\tube_A)\cap \tubes^\prime$ we have
$$
|H_{Y^\prime}(\tube_3)\cap\tubes_{12}|\gtrapprox_{C(A|B)\eps} \delta^{-5/2}|\tubes^\prime|^{-1}\gtrapprox_{C(A|B)\eps}\delta^{-1/2}.
$$

We claim that for each $\tube_3\in  H(\tube_A)\cap \tubes^\prime$, there is a set of $\gtrapprox_{C_4\eps} \delta^{-3/2}$ tubes  $\tube^\prime\in H_{Y^\prime}(\tube_3)$ with $\angle(v(\tube^\prime),\ T_pR_{\tube_1,\tube_2,\tube_3})\geq \delta^{C_1\eps}$ at every point $p\in \tube^\prime\cap \tube_3$. Here the constant $C_4$ depends on $C_1,C_2,C(A|B),$ and $C_3$. 

If we are in Case (A), then this follows from Lemma \ref{decoherenceOfPlaneMap}, provided we select $C_1$ sufficiently large, depending on $C_2$ (here we use the fact that the hairbrush of $\tube_3$ covers each of $\tube_1$ and $\tube_2$ with accuracy $\geq \delta^{C_A\eps}$. On the other hand, if we are in Case (B) then this follows from \eqref{avoidsReguli}, provided we select the constant $C_2$ sufficiently large (independent of all other parameters), so that the cardinality of the set in the LHS of \eqref{avoidsReguli} (with the shading $Y_1$ replaced by $Y^\prime$) is at most half the cardinality of the set on the RHS. 

The proof from this point on no longer distinguishes between Cases (A) and (B). We will assume that the constant $C_1,C_2, C_{(A)}, C_{(B)}$, and $C_4$ have been determined, and when we write $A\lessapprox_{\eps}B$, the implicit constant in the $\lessapprox_{\eps}$ notation may depend on $C_1,C_2, C_{(A)}, C_{(B)}$, and $C_4$, since these are (fixed) absolute constants.

By pigeonholing, we can select a tube $\tube_B\in\tubes$ so that there are $\gtrapprox_\eps\delta^{-1}$ tubes $\tube_3\in H^\prime(\tube_A,\tube_B)\cap\tubes^\prime$ satisfying
\begin{equation}\label{angleRequirement}
\angle\big( v(\tube_B),\ T_pR_{\tube_1,\tube_2,\tube_3}\big)\gtrapprox_\eps 1\quad\textrm{for all}\ p\in \tube_3\cap \tube_B.
\end{equation}
Call this set of tubes $\tubes^{\prime\prime}.$

\begin{figure}[h!]
 \centering
\begin{overpic}[width=0.5\textwidth]{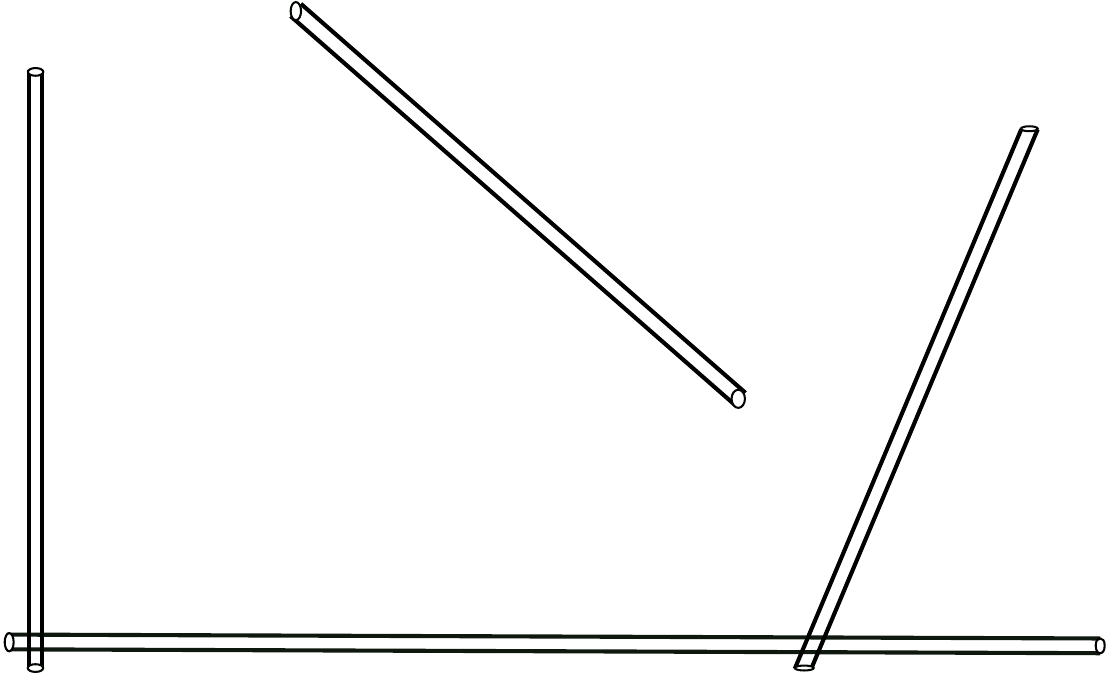}
 \put (5,26) {$\tube_1$}
 \put (84,22) {$\tube_2$}
 \put (42,6) {$\tube_A$}
 \put (37,41) {$\tube_B$}
\end{overpic}
 \caption{The relationship between $\tube_1,\tube_2,\tube_A$, and $\tube_B$}\label{tube12AB}
\end{figure}

Let $1\leq N\leq \delta^{-1}s$ and let $\tube,\tube^\prime\in\tubes_{12}$. Suppose there exists at least one tube $\tube_3\in\tubes^{\prime\prime}$ that intersects the $N\delta$-neighborhood of $\tube$ and $\tube^\prime$. Suppose that these points of intersection are $\geq\delta^{1-C\eps}N$ separated, where $C$ is a constant that depends only on the separation and skewness of $\tube_1,\tube_2$ and $\tube_3$. In particular, this implies that $\tube$ and $\tube^\prime$ are $\approx_{\eps}1$ skew and $\approx_{\eps}q$ separated with error $\lessapprox_{\eps}1$ for some $q$ satisfying $N\delta\leq q\leq s$.

Since $\tube_1,\tube_2,$ and $\tube_3$ are $\approx_{\eps}1$ separated and skew, and $v(\tube_3)$ makes an angle $\gtrapprox_{\eps}1$ with the plane spanned by $v(\tube_1)$ and $v(\tube_2)$, we can apply Lemma \ref{regulusThreeLinesFarFromCommonPlane}. Let $Q$ be the monic degree-two polynomial that vanishes on the three lines coxial with $\tube_1,\tube_2,$ and $\tube_3$. Since there are three $\approx_{\eps}1$ separated points on $\tube_A$ that are $\delta$--close to $Z(Q)$, we have that the restriction of $Q$ to the $\delta$ neighborhood of $\tube$ has magnitude $\lessapprox_{\eps}\delta$. Similarly, the restriction of $Q$ to the $N\delta$ neighborhood of $\tube$ and $\tube^\prime$ has magnitude $\lessapprox_{\eps}N\delta$. Now let $\tube_3^{\prime}$ be a tube that intersects $\tube_A$ and also intersects the $N\delta$ neighborhoods of $\tube$ and $\tube^\prime$. Then if we restrict $Q$ to the line coaxial with $\tube_3^{\prime}$, this univariate polynomial has magnitude $\lessapprox_\eps N\delta$ at three points, which have separation $q$ and $s$. Thus by Lagrange interpolation, the restriction of $Q$ to the line concentric with $\tube_3^{\prime}$ has magnitude $\lessapprox_\eps N\delta/sq$. Since $|\nabla Q|\gtrapprox_{\eps}1$ on $B(0,1)\cap Z(Q)$, we conclude that $\tube_3^{\prime}$ is contained in the $\lessapprox_\eps N\delta/sq$--neighborhood of $Z(Q)$.

Now consider the set
\begin{equation}\label{tubesMeetTubeTubePrime}
\{\tube_3 \in \tubes^{\prime\prime}\colon \tube_3\cap N_{N\delta}(\tube)\neq\emptyset,\ \tube_3\cap N_{N\delta}(\tube^\prime)\neq\emptyset \}.
\end{equation}
Since the tubes in \eqref{tubesMeetTubeTubePrime} are contained in $H^\prime(\tube_A,\tube_B)$, we have that the intersections of these tubes with $\tube_B$ must be $\lesssim 1$ overlapping. The tubes in \eqref{tubesMeetTubeTubePrime} are also contained in the $\approx_{\eps}N\delta/sq$ neighborhood of the regulus $R_{\tube_1,\tube_2,\tube_3^{\prime}}$, where $\tube_{3}^\prime$ is any tube in $\tubes^{\prime\prime}$ that intersects $N_{N\delta}(\tube_A),\ N_{N\delta}(\tube),$ and $N_{N\delta}(\tube^\prime)$ (if no such tube exists then \eqref{tubesMeetTubeTubePrime} is empty, which is good for us). Finally, the tubes in this set must intersect $\tube_B$, and $\tube_B$ makes an angle $\gtrapprox_{\eps}1$ with the tangent plane of $R_{\tube_1,\tube_2,\tube_3^\prime}$ at every point of intersection. At most $\lessapprox_\eps N/sq$ tubes from $\tubes^\prime$ can satisfy these three properties, so
$$
|\eqref{tubesMeetTubeTubePrime}| \lessapprox_\eps N/sq.
$$

Let $\tubes_{AB}$ be a random $\sim s^2$ refinement of the set of tubes from $\tubes^{\prime\prime}$; this refinement can be obtained by randomly selecting each tube with probability $s^2$. Then with high probability, $|\tubes_{AB}|\lessapprox_\eps\delta^{-1}s^2$ and for each $\tube, \tube^\prime\in\tubes_{12}$ that are $q$ separated with error $\lessapprox_{\eps}1$,  we have
$$
|\{\tube_3\in \tubes_{AB}\colon \dist(\tube, \tube_3)\leq N\delta,\ \dist(\tube^\prime,\tube_3)\leq N\delta \}|\lessapprox_\eps Ns/q.
$$

After applying a linear transformation that distorts angles by a factor of $\lessapprox_{\eps}1$, we can assume that the line coaxial with $\tube_1$ is the line $(0,0,0)+\RR(1,0,0),$ the line coaxial with $\tube_2$ is the line $(0,0,1)+\RR(0,1,0)$, and the line coaxial with $\tube_A$ is the line $(0,0,0)+\RR(0,0,1)$. After applying this transformation, the original tubes from $\tubes$ might be slightly distorted and thus might no longer be `tubes'' according to the definition given in Section \ref{introSection}, but nonetheless each $\tube\in\tubes$ will contain the $\approx_{\eps}\delta$ neighborhood of a line segment of length $\approx_{\eps}1$ and will be contained in the $\approx_{\eps}\delta$ neighborhood of a line segment of length $\approx_\eps 1$. We will abuse notation slightly and refer to these objects as tubes, since they will still obey all of the estimates discussed above.

We will now construct the set $\pts$. Recall  $\tubes_{12}\subset H^\prime(\tube_1,\tube_2)$, so in particular the sets $\{\tube\cap\tube_2\colon\tube\in\tubes_{12}\}$ are $\lesssim 1$ overlapping.  Let $\tubes_{12}^\prime$ be a subset of $\tubes_{12}$ of size $\approx_{\eps}s|\tubes_{12}|\approx_{\eps}\delta^{-1}s^2$ so that the sets $\{\tube\cap\tube_2\colon\tube\in\tubes_{12}^\prime\}$ are $\geq\delta/s$ separated. For each $\tube\in\tubes_{12}^\prime$, $\tube\cap \tube_1$ is contained in a ball of radius $\approx_{\eps}\delta$ centered at a point of the form $(x,0,0)$ and $\tube\cap \tube_2$ is contained in a ball of radius $\approx_{\eps}\delta$ centered at a point of the form $(0,y,1)$. By \eqref{defnTubes12}, $|x|,|y|\approx_{\eps}s$. Define
$$
\phi(\tube)=(s/x, s/y).
$$
Let $\tilde\delta = \delta/s^2$, so in particular $\tilde\delta\leq\delta^\alpha$.
\begin{rem}
If $A \leq \delta^{-C\eps} B$, then $A\leq \tilde\delta^{-C\eps/\alpha}B$. Thus we will frequently see the expression $A\lessapprox_{\eps/\alpha}B$.
\end{rem}

Define $\pts=\phi(\tubes_{12}^\prime)$. This is a set of $\approx_{\eps/\alpha}\tilde\delta^{-1}$ points that are $\tilde\delta$--separated and contained in a ball centered at the origin of radius $\lessapprox_{\eps}1$. Since the sets $\{\tube\cap\tube_2\colon\tube\in\tubes_{12}^\prime\}$ are $\geq\delta/s$ separated, we have
\begin{equation}\label{ptsInBallNonConcentration}
|B(x,r)\cap\pts|\lessapprox_{\eps/\alpha} r|\pts|.
\end{equation}

Next, we will construct the set $\lines$. If $\tube\in \tubes_{AB}$, then $\tube\cap \tube_A$ is contained in a ball of radius $\lessapprox_\eps\delta$ centered at the point $(0,0,z)$, and the line coaxial with $\tube$ has the form $(-az, -bz, 0)+\RR(a,b,1)$, with $|a|,\ |b|\lessapprox_{\eps} 1$. Let
$$
\psi(\tube)=\Big(\frac{bz-b}{s},\ \frac{az}{s}\Big),
$$
and let $\check\psi(\tube)$ be the line dual to $\psi(\tube)$. Define $\mathcal{L}=\check\psi(\tubes_{AB})$.

Now, suppose that $\tube\in\tubes_{12}$ and $\tube^\prime\in\tubes_{AB}$, i.e.
\begin{equation*}
\begin{split}
\tube&=\{ (1-t)x, ty, t),\ t\in\RR\}+O(\delta),\\
\tube^\prime &= \{(a(t-z),b(t-z), t)\colon t\in\RR\}+O(\delta).
\end{split}
\end{equation*}
And suppose furthermore that $\dist( \tube, \tube^\prime) \sim N\delta$ for some $1\leq N\leq\delta^{-1}$. Then we must have
$$
(1-t)x - a(t-z) =\Theta(N\delta),\quad \textrm{and}\quad ty-b(t-z) = \Theta(N\delta).
$$
Re-arranging, we conclude
$$
t=\frac{bz + \Theta(N\delta)}{b-y},
$$
and thus
$$
a\big( bz+\Theta(\delta)-z(b-y)\big)=\big( (b-y)-bz+\Theta(N\delta)\big)x,
$$
i.e.
\begin{equation*}
\begin{split}
&a\big( bz-z(b-y)\big)-\big( (b-y)-bz\big)x=\Theta(N\delta),\quad\textrm{and}\quad
xy+(bz-b)x +az y =\Theta(\delta).
\end{split}
\end{equation*}
Since $x,y\approx_{\eps}1$, we can divide by $zy$ and obtain
$$
\Big|1+\Big(\frac{bz-b}{s}\Big)\Big(\frac{s}{y}\Big) +\Big(\frac{az}{s}\Big)\Big( \frac{s}{x}\Big)\Big|\approx_{\eps}N\delta s^{-2}=N\tilde\delta.
$$
This is exactly the statement that $\dist(\psi(\tube), \check\phi(\tube^\prime))\approx_{\eps} N\tilde\delta.$ This implies \eqref{nonConcentrationPropertyLines} and \eqref{manyIncidences}. After dilating $\RR^2$ by a factor of $\lessapprox_{\eps/\alpha}1$, we can assume that $\pts\subset B(0,1)$.

We conclude: $\pts\subset B(0,1)$ is a set of $\approx_{\eps/\alpha}\tilde\delta^{-1}$ points that satisfies \eqref{nonConcentrationPropertyPts}, and $\lines$ is a set of $\approx_{\eps/\alpha}\tilde\delta^{-1}$ lines that satisfies \eqref{nonConcentrationPropertyLines}. Furthermore, if the constant $C$ is selected sufficiently large, then for each $L\in\lines$, there are $\gtrapprox_{\eps/\alpha}\tilde\delta^{-1}$ pairs $(p_1,p_2)\in\pts^2$ with $\dist(p_1,p_2)\geq \delta^{C\eps/\alpha}$, $\dist(p_1, L)\leq\tilde\delta$, and $\dist(p_2, L)\leq\tilde\delta.$
\end{proof}
\subsection{Establishing non-concentration of the points}
Recall that the set of points from the previous section satisfies the non-concentration estimate
\begin{equation}\label{nonConcentrationEstimateFluff}
|\pts\cap B(x,r)|\lessapprox_{\eps/\alpha}r|\pts|.
\end{equation}
Using standard arguments, will find a projective transformation $T$ to so that $T(\pts)$ has large intersection with a Cartesian product $S_1\times S_2$, where $S_1$ and $S_2$ each have cardinality roughly $\tilde\delta^{-1/2}$. This will be done in Lemma \ref{FingingProductStructure}.

Even though the set $S_1\times S_2$ still satisfies \eqref{nonConcentrationEstimateFluff}, this does not imply that the sets $S_1$ or $S_2$ individually satisfy any non-concentration estimates. For example, we might have $S_1 = [0,\tilde\delta^{1/2}]\cap\tilde\delta\ZZ$ and $S_2 = [0,1]\cap\tilde\delta^{1/2}\ZZ$.

In this section, we will show that because the lines also satisfy a non-concentration estimate, this forces the sets $S_1$ and $S_2$ to satisfy a non-concentration estimate as well.
\begin{lem}\label{FingingProductStructure}
Let $\pts,\lines,$ and $\tilde\delta$ satisfy the conclusion of Lemma \ref{findingPtsAndLinesLem}. Then there exist sets $A,B\subset[0,1]$, a set $\pts^\prime\subset A\times B$, and a set $G\subset B\times B$ so that
\begin{itemize}
\item $|A|,|B|\approx_{\eps/\alpha}\tilde\delta^{-1/2}$.
\item $|\pts^\prime|\approx_{\eps/\alpha}\tilde\delta^{-1}$.
\item For each ball $B(x,r)\subset \RR^2$,
\begin{equation}\label{nonConcentrationPropertyPtsRd2}
|\pts^\prime\cap B(x,r)|\lessapprox_{\eps/\alpha} r\tilde\delta^{-1}.
\end{equation}
\item $|G|\gtrapprox_{\eps/\alpha}\tilde\delta^{-1}$.
\item For every $b_1\in B$, $|(\{b_1\}\times B)\cap G|\gtrapprox_{\eps/\alpha} \tilde\delta^{-1/2}$.
\item For every $p\in\pts,$ there exist $(b_1,b_2)\in G$ so that the three points $p$, $(0,b_1)$, and $(1,b_2)$ are contained in the $\tilde\delta$--neighborhood of a line (this line need not be an element of $\lines$).
\item For every two intervals $I_1,I_2\subset[0,1]$,
\begin{equation}\label{nonConcentrationOnRectInLem}
|G\cap (I_1\times I_2)|\lessapprox_{\eps/\alpha} \max(|I_1|,|I_2|)\tilde\delta^{-1}.
\end{equation}
\item For each $(b_1,b_2)\in G$, if $L$ is the line connecting $(0,b_1)$ to $(1,b_2)$, then
$$
|\{(p_1,p_2)\in(\pts^\prime)^2\colon\dist(p_1,p_2)\geq\tilde\delta^{C\eps},\ \dist(p_i,L)\leq\tilde\delta,\ i=1,2, \}|\gtrapprox_{\eps}\tilde\delta^{-1}.
$$
\end{itemize}
\end{lem}
\begin{proof}

The basic idea is that we will trap most of the points from $\pts$ in the ``bush'' of two points $p_1$ and $p_2$. We will then apply a linear transformation sending lines through $p_1$ to vertical lines, and lines through $p_2$ to horizontal lines. Under this transformation, $\pts$ is mapped to a Cartesian product.

For each $p\in\pts$, define
$$
H(p)=\{p^\prime \in\pts\colon \dist(p,p^\prime)\geq\tilde\delta^{C\eps/\alpha},\ \textrm{there exists}\ L\in\lines\ \textrm{with}\ \dist(p,L)\leq\tilde\delta,\ \dist(p^\prime,L)\leq\tilde\delta\}.
$$
Discard all lines $L\in\lines$ that have fewer than  $\tilde \delta^{-1+C \eps/\alpha}$ pairs of incident points, where $C$ is a large constant. This guarantees
us that for each $p$, there are at most $\tilde \delta^{-{1 \over 2} - C \eps / \alpha}$ lines accounting for all but a small fraction of the points of any $H(p)$.
If the constant $C$ is chosen sufficiently large, then
\begin{equation*}
\begin{split}
\sum_{p\in \pts}|H(p)|&=\sum_{L\in\lines}|\{(p_1,p_2)\in\pts^2\colon \dist(p_1,p_2) \geq\tilde\delta^{C\eps/\alpha},\ \dist(p_1,L)\leq\tilde\delta,\ \dist(p_2,L)\leq\tilde\delta\}|\\
&\gtrapprox_{\eps/\alpha}\tilde\delta^{-2}\\
&\approx_{\eps/\alpha}|\pts|^{2}.
\end{split}
\end{equation*}

Thus by pigeonholing, there exist points $p_1,p_2$ so that $\dist(p_1,p_2)\geq\tilde\delta^{C\eps/\alpha}$ and $|\pts\cap H(p_1)\cap H(p_2)|\gtrapprox_{\eps/\alpha}\tilde\delta^{-1}$. If necessary, prune the set of lines incident to $p_1$ and $p_2$ so that there are $\leq\tilde\delta^{-1/2}$ lines incident to each of $p_1$ and $p_2$; we still have $|\pts\cap H(p_1)\cap H(p_2)|\gtrapprox_{\eps/\alpha}\tilde\delta^{-1}$.

Select a ball $B$ with $\dist(B,p_1)\geq\tilde\delta^{C_1\eps/\alpha},\ \dist(B,p_2)\geq\tilde\delta^{C_1\eps/\alpha}$, and $|\pts\cap H(p_1)\cap H(p_2)\cap B|\gtrapprox_{\eps/\alpha}\tilde\delta^{-1}.$ If $C_1$ is chosen sufficiently large then such a ball must exist. Define $\pts_1=\pts\cap H(p_1)\cap H(p_2)\cap B.$ Let $\varphi$ be a projective transformation sending points through $p_1$ to vertical lines and points through $p_2$ to horizontal lines. Define $\pts_2=\varphi(\pts_1)$; this set is contained in a Cartesian product $A\times B$, with $|A|\approx_{\eps/\alpha}\tilde\delta^{-1/2},\ |B|\approx_{\eps/\alpha}\tilde\delta^{-1/2}$.

For each ball $B(x,r)\subset B$ we have that $\varphi(B(x,r))$ is contained in a ball of radius $\lessapprox_{\eps/\alpha}r$ and contains a ball of radius $\gtrapprox_{\eps/\alpha}r$. Thus for each ball $B(x,r)$ we have
$$
|B(x,r)\cap \pts_2|\lessapprox_{\eps/\alpha} r|\pts_2|.
$$
Define
$$
\lines_1=\{\varphi(L)\colon L\in\lines\}.
$$
Inequality \eqref{nonConcentrationPropertyLines} still holds.

Without loss of generality, we can assume that
\begin{itemize}
\item $0,1\in A$.
\item $|\pts_2\cap(\{0\}\times B)|\approx_{\eps/\alpha}\tilde\delta^{-1/2}$.
\item $|\pts_2\cap(\{1\}\times B)|\approx_{\eps/\alpha}\tilde\delta^{-1/2}$.
\item For every $p_1=(0,b_1)$, there are $\gtrapprox_{\eps/\alpha}\tilde\delta^{-1/2}$ points of the form $p_2=(1,b_2)\in\pts_2$ so that there is a line $L\in\lines_1$ that is $\tilde\delta$--incident to $p_1$ and $p_2$.
\end{itemize}
Indeed, the second, third, and fourth properties above hold for $\gtrapprox_{\eps/\alpha}|A|^2$ pairs $(x_1,x_2)\in A^2$ that satisfy $\dist(x_1,x_2)\gtrapprox_{\eps/\alpha}1$. Select one of these pairs $(x_1,x_2)$ and apply an affine transformation sending the line $\{x_1\}\times\RR$ to $\{0\}\times\RR$ and the line $\{x_2\}\times\RR$ to $\{1\}\times\RR$.

We can refine $\lines_1$ so that every $L\in\lines_1$ is incident to a point of the form $(0,b_1)$ and a point $(1,b_2)$. Thus we can identify $\lines_1$ with a subset $G\subset B\times B$. We have $|G|\gtrapprox_{\eps/\alpha}|B\times B|\approx_{\eps/\alpha}\tilde\delta^{-1}$. Finally discard all points from $\pts_2$ that do not satisfy $\dist(p,L)\leq\tilde\delta$ for some $L$ in (the refined version of) $\lines_1.$

The final thing to verify is \eqref{nonConcentrationOnRectInLem}, but this is just a re-statement of \eqref{nonConcentrationPropertyLines}.
\end{proof}


The set $\pts^\prime\subset A\times B$ from Lemma \ref{FingingProductStructure} satisfies the non-concentration property $|\pts^\prime\cap B(x,r)|\lessapprox_{\eps/\alpha}r\tilde\delta^{-1}$. Next, we will show that the set $B$ satisfies a similar type of non-concentration property. In Lemma \ref{oneScaleRefinement} we will establish non-concentration at a single scale, and we will use this result in Lemma \ref{multiScaleRefinement} to establish non-concentration at every scale.

\begin{lem}\label{oneScaleRefinement}
Let $B\subset[0,1]$ be a set of $\tilde\delta$--separated points, with $\tilde\delta^{-1/2+\eps}\leq|B|\leq\tilde\delta^{-1/2}$. Let $G\subset B\times B$ with $|G|\geq\tilde\delta^{-1+\eps}$. Suppose that for all intervals $I_1,I_2\subset[0,1]$, we have
$$
|G\cap (I_1\times I_2)|\leq K \max(|I_1|,|I_2|)\tilde\delta^{-1}
$$
for some $K>0$. Suppose as well that for all intervals $I\subset[0,1]$ of length $\leq r_0$, we have
$$
|B\cap I|\leq |I|^{1/100}\tilde\delta^{-1/2}.
$$

Then there exists a set $B^\prime\subset B$ with $|B^\prime|\geq\tilde\delta^{3\eps}|B|$ so that for all intervals $J\subset[0,1]$ of length $|J|\leq \min(\tilde\delta^{9\eps}, r_0^{1/50})$,
\begin{equation}\label{JCapBNonConcentrated}
|J\cap B^\prime|\leq K|J|^{1/100}\tilde\delta^{-1/2}.
\end{equation}
\end{lem}
\begin{proof}
First, select $B_1\subset B$ so that $|B_1|\geq\tilde\delta^{\eps}|B|$ and
$$
|(\{b\}\times[0,1])\cap G|\geq\tilde\delta^{-1/2-2\eps}
$$
for all $b\in B_1$. Let $I\subset[0,1]$ be the shortest interval for which $|I\cap B_1|\geq|I|^{1/100}\tilde\delta^{-1/2}$. We know that $|I|\geq r_0$. If $|I|\geq \min(\tilde\delta^{9\eps}, r_0^{1/50})$ or if no such interval exists, then we let $B^\prime=B_1$ and we are done.

Henceforth we can assume that $|I|\leq r_0^{1/50}$. Let $B^\prime\subset B$ be the set of all $b\in B$ with
$$| ((B_1\cap I)\times\{b\})\cap G|\geq\tilde \delta^{3\eps}|B_1\cap I|.
$$
We have $|B^\prime|\geq\tilde\delta^{\eps}|B|\geq\tilde\delta^{-1/2+2\eps}.$

Let $J\subset[0,1]$ be any interval satisfying $|I|\leq |J|\leq \min(\tilde\delta^{9\eps}, r_0^{1/50})$. We have
\begin{equation}\label{upperBdOnGIJ}
|G\cap (I\times J)|\geq \big(\tilde\delta^{3\eps}|B_1\cap I|\big)\big(|B^\prime\cap J|\big)\geq r_0^{1/100}\tilde\delta^{-1/2+3\eps}|B^\prime\cap J|,
\end{equation}
and
\begin{equation}\label{lowerBdOnGIJ}
|G\cap (I\times J)|\leq K\max(|I|,|J|)\tilde\delta^{-1}=K|J|\tilde\delta^{-1}.
\end{equation}
Combing \eqref{upperBdOnGIJ} and \eqref{lowerBdOnGIJ},  we obtain
\begin{equation*}
\begin{split}
|J\cap B^\prime|&\leq K|J|r_0^{-1/100}\tilde\delta^{-1/2-3\eps}\\
&\leq K|J|^{1/100}\tilde\delta^{-1/2} (|J|^{99/100}\tilde\delta^{-3\eps}r_0^{-1/100}) \\
&\leq K|J|^{1/100}\tilde\delta^{-1/2}(|J|^{1/3}\tilde\delta^{-3\eps})(|J|^{197/300}r_0^{-1/100})\\
&\leq K|J|^{1/100}\tilde\delta^{-1/2}\big( (\tilde\delta^{9\eps})^{1/3}\tilde\delta^{-3\eps}\big)\big(r_0^{1/50})^{197/300}r_0^{-1/100}\big)\\
&\leq K|J|^{1/100}\tilde\delta^{-1/2}. \qedhere
\end{split}
\end{equation*}
\end{proof}
We will now use the preceding lemma to show that the set $B$ is non-concentrated at every scale.

\begin{lem}\label{multiScaleRefinement}
For every $\tau>0$, there exists a constant $C$ so that the following holds:

Let $B\subset[0,1]$ be a set of $\tilde\delta$--separated points, with $\tilde\delta^{-1/2+\eps}\leq|B|\leq\tilde\delta^{-1/2}$. Let $G\subset B\times B$ with $|G|\geq\tilde\delta^{-1+\eps}$. Suppose that for all intervals $I_1,I_2\subset[0,1]$, we have
$$
|G\cap (I_1\times I_2)|\leq K \max(|I_1|,|I_2|)\tilde\delta^{-1}.
$$

Then there is a set $B^\prime\subset B$ with $|B^\prime|\geq\tilde\delta^{-1/2+C\eps}$ so that for all intervals $I\subset[0,1]$ with $|I|\leq\min(\tilde\delta^{C\eps},\tilde\delta^{\tau})$, we have
\begin{equation}\label{nonConcentrationOnIntervalAllScales}
|B\cap I|\leq K|I|^{1/100}\tilde\delta^{-1/2}.
\end{equation}
\end{lem}
\begin{proof}
Let $N=\lceil|\log_{50}\tau|\rceil$, and let $C = 9\cdot 3^N$. Apply Lemma \ref{oneScaleRefinement} $N$ times, and let $B^\prime$ be the resulting set. Observe that after the $k$--th application of Lemma \ref{oneScaleRefinement}, we obtain a set $B_k$ with $\tilde\delta^{-1/2+3^k\eps}\leq|B_k|\leq \tilde\delta^{-1/2}$, and
$$
|B_k\cap J|\leq K|J|^{1/100}|B_k|
$$
for all intervals $J$ of length $|J|\leq\min(\tilde\delta^{\eps 9\cdot 3^k},\ \tilde\delta^{1/50^k})$. Thus after $N$ applications of Lemma \ref{oneScaleRefinement}, we obtain a set $B_N$ with $\tilde\delta^{-1/2+C\eps}\leq |B_N|\leq \tilde\delta^{-1/2}$, and
$$
|B_N\cap J|\leq K|J|^{1/100}|B_k|
$$
for all intervals $J$ of length $|J|\leq\min(\tilde\delta^{C\eps}, \tilde\delta^{\tau})$. Let $B^\prime=B_N$.
\end{proof}
Applying Lemma \ref{multiScaleRefinement} to the output from Lemma \ref{FingingProductStructure} with $K\lessapprox_{\eps/\alpha}1$, we obtain the following:
\begin{lem}\label{CartesianProductStructureNonConcentrated}
There is a constant $C$ so that the following holds. Let $(\tubes,Y)$ be an $\eps$--extremal set of tubes of Heisenberg type (with parameter $\alpha$ and $C\eps$). Then there is a number $\tilde\delta\geq\delta^{\alpha}$ so that: for all $\tau>0$, there exists a constant $C_\tau$, sets of points $B^\prime\subset[0,1]$, $\pts\subset B(0,1)\subset\RR^2$, and a set of lines $\lines$ with the following properties.
\begin{itemize}
\item $|\pts|\approx_{C_{\tau}\eps/\alpha}\tilde\delta^{-1}$.
\item For each ball $B(x,r)\subset \RR^2$,
\begin{equation}\label{nonConcentrationPropertyPtsRd2}
|\pts\cap B(x,r)|\lessapprox_{C_{\tau}\eps/\alpha} r\tilde\delta^{-1}.
\end{equation}
\item $|B^\prime|\approx_{C_{\tau}\eps/\alpha}\tilde\delta^{-1/2}$.
\item For every interval $I$ of length $\tilde\delta\leq|I|\leq \tilde\delta^{\tau}$,
\begin{equation}\label{nonConcentrationOfBCapI}
|B^\prime\cap I|\lessapprox_{C_{\tau}\eps/\alpha}|I|^{1/100}\tilde\delta^{-1/2}.
\end{equation}

\item Each $L\in\lines$ is incident to a point $(b_0,0)$ and $(b_1,1)$ with $b_0,b_1\in B^\prime$. For each line $L\in\lines$, we have
$$
|\{p\in\pts\colon \dist(L,p)\leq\tilde\delta\}|\gtrapprox_{C_\tau\eps/\alpha}\tilde\delta^{-1/2}.
$$
\end{itemize}
\end{lem}
\begin{proof}
Let $\pts^\prime\subset A\times B$ and $G\subset B\times B$ be the output obtained by applying Lemma \ref{FingingProductStructure} to the collection $(\tubes,Y)$. Apply Lemma \ref{multiScaleRefinement} to $B$ and $G$ with the value of $\tau$ specified in the statement of Lemma \ref{CartesianProductStructureNonConcentrated}, and let $B^\prime$ be the resulting subset of $B$. Let $\lines = G\cap (B^\prime\times B^\prime)$, where we identify the pair $(b_1,b_2)$ with the line passing though the points $(0,b_1)$ and $(1,b_2)$.

All of the required properties of the sets $B^\prime,\ \pts,$ and $\lines$ follow from the conclusions of Lemmas \ref{FingingProductStructure} and \ref{multiScaleRefinement}.
\end{proof}

\subsection{Reduction to Bourgain's discretized projection theorem}\label{reductionToBourgainSec}
We will now perform a reduction that transforms the output of Lemma \ref{CartesianProductStructureNonConcentrated} into the input for Bourgain's discretized projection theorem.

Apply a projective transformation $\phi$ that sends the line $\{0\}\times \RR$ to the line at infinity (we will assume that for at least half the incidences $\{(p,L)\in\pts\times\lines,\ d(p,L)\leq\tilde\delta\}$, the $x$--coordinate of $p$ is $\geq 1/2$. If not, we can apply a projective transformation that sends the line $\{1\}\times\RR$ to the line at infinity instead). Let $\Theta\subset S^1$ be the image of $B$ under this transformation. We have that $|\Theta|\gtrapprox_{C_\tau\eps/\alpha}\tilde\delta^{-1/2}$, so in particular $|\Theta|\geq\tilde\delta^{-1/3}$. Furthermore, for each interval $I$ of length $\tilde\delta\leq r\leq \delta^{\tau}$ we have
\begin{equation}\label{ThetaWellDistributed}
|\Theta\cap I|\lessapprox_{\eps_1}|I|^{1/100}|\Theta|,
\end{equation}
where
\begin{equation}\label{defnEps1}
\eps_1 = C_{\tau}\eps/\alpha.
\end{equation}

Let $\pts_1=\phi\big(\pts\cap ([1/2, 1]\times[0,1])\big)$, and let $\lines_1=\{\phi(L)\colon L\in\lines\}$. Observe that the direction of each line in $\lines_1$ is equal to one of the directions in the set $\Theta$. For each $\theta\in\Theta$, let
$$
\lines(\theta) = \{L\in\mathcal{L}_1\colon L\ \textrm{point in direction}\ \theta\}.
$$

For each $p\in\pts_1$, define
$$
\lines(p)=\{L\in \lines_1\colon \dist(p,L)\leq\tilde\delta\}.
$$
For each $L\in\lines_1$, define
$$
\pts(L)=\{p\in\pts_1\colon \dist(p,L)\leq\tilde\delta\}.
$$
For each $\theta\in\Theta$, define
$$
\pts(\theta)=\bigcup_{L\in\lines(\theta)}\pts(L).
$$

Define
$$
I_{\tilde\delta}(\pts_1,\lines_1)=\{(p,L)\in\pts_1\times\lines_1\colon\dist(p,L)\leq\tilde\delta\}.
$$

Select $\theta_1<\theta_2<\theta_3\in\Theta$ that are pairwise $\gtrapprox_{\eps_1}1$-separated so that
$$
\sum_{p\in \pts(\theta_1)\cap \pts(\theta_2)\cap\pts(\theta_3)}|\lines(p)|\gtrapprox_{\eps_1}\tilde\delta^{-3/2}.
$$
(We can use \eqref{ThetaWellDistributed} and a pigeonholing argument to guarantee that such a choice of $\theta_1,\theta_2,$ and $\theta_3$ exist). Let $\pts_2$ be the image of $\pts(\theta_1)\cap \pts(\theta_2)\cap\pts(\theta_3)$ under the linear transformation sending $\theta_1$ to the vertical direction and $\theta_2$ to the horizontal direction. Let $\lines_2$ be the image of $\lines_1$ under this transformation and let $\Theta_1$ be the image of $\Theta$ under this transformation (the transformation has determinant $\approx_{\eps}1$ and distorts angles by a factor of $\lessapprox_{\eps}1$, so \eqref{ThetaWellDistributed} still holds for $\Theta_1$).

Then $\pts_2\subset A\times B$, where $|A|\lessapprox_{\eps_1}\tilde\delta^{-1/2}$,  $|B|\lessapprox_{\eps_1}\tilde\delta^{-1/2}$, and the sets $A$ and $B$ are $\tilde\delta$-separated. Let $\lambda$ be the angle associated to (the image of) $\theta_2$, so $\tilde\delta^{C_2\eps_1}\leq\lambda\leq \tilde\delta^{-C_2\eps_1}$ for some absolute constant $C_2$.

We will now refine the sets $A$ and $B$ slightly to increase the separation between the points. Let $C_2$ be a constant to be determined later. We can find a refinement $A_1\subset A,\ B_1\subset B$ and a set $\pts_3\subset \pts_2\cap (A_1\times B_1)$ so that the points in $A_1$ and $A_2$ are $\tilde\delta^{1-C_2\eps_1}$ separated; $|\pts_3|\gtrapprox_{\eps_1}\tilde\delta^{-1}$; for each $p\in\pts_3,\ |\lines(p)|\gtrapprox_{\eps_1}\tilde\delta^{-1/2}$; and
\begin{equation}\label{lotsOfIncidencesWithTheta2}
|I_{\tilde\delta}(\pts_3, \lines(\theta_2))|\gtrapprox_{\eps_1} \tilde\delta^{-1}.
\end{equation}

We can identify $\tilde\delta^{-1}A_1$ and $\tilde\delta^{-1}\lambda B_1$ with subsets of $\ZZ$ by rounding each element of $\tilde\delta^{-1}A_1$ and $\tilde\delta^{-1}\lambda B_1$ to the closest integer---call these sets $\hat A_1$ and $\hat B_1$. Since $A_1$ is $\geq\tilde\delta$ separated, the map from $A_1$ to $\hat A_1$ is injective (or to be pedantic, it's at most two-to-one), and since $B_1$ is $\tilde\delta^{1-C_2\eps_1}$--separated and $\lambda\geq\tilde\delta^{C_2\eps_1}$, the map from $B_1$ to $\hat B_1$ is injective (again, at most two-to-one).

With this identification, we have that $|\hat A_1|\lessapprox_{\eps_1}\tilde\delta^{-1/2}$ and $|\hat B_1|\lessapprox_{\eps_1}\tilde\delta^{-1/2}$. By \eqref{lotsOfIncidencesWithTheta2}, Cauchy-Schwarz, and the fact that $|\lines(\theta_2)|\lessapprox_{\eps_1}\tilde\delta^{-1/2}$, there are $\gtrapprox_{\eps_1}\tilde\delta^{-3/2}$ solutions to the equation
\begin{equation}
\{a_1+b_1 = a_2+ b_2\colon a_1,a_2\in\hat A_1,\ b_1,b_2\in\hat B_1\}.
\end{equation}
Indeed, there are $\approx_{\eps_1}\tilde\delta^{-1/2}$ lines in $\lines(\theta_2)$ that are incident to $\approx_{\eps_1}\tilde\delta^{-1/2}$ points from $\pts_3$, and if $(x_1,y_1),\ (x_2,y_2)\in\pts_3$ are incident to the same line, then $|x_1+\lambda y_1 - (x_2+\lambda y_2)|\leq\tilde\delta$. Note that $\tilde\delta^{C_2\eps_1}\leq\lambda\leq\tilde\delta^{-C_3 \eps_1}$, and thus if $a_1,b_1,a_2,b_2\in \ZZ$ are the points associated to $x_1,y_1,x_2,$ and $y_2$ respectively, then $a_1+b_1 = a_2 + b_2$ (this is a slight lie, we actually have $|a_1+b_1-a_2-b_2|\leq 4$, but by pigeonholing this only decreases the number of quadruples by a constant factor). Thus there are $\gtrapprox_{\eps_1}\tilde\delta^{-3/2}$ such quadruples.

The next result allows us to find subsets of $\hat A_1$ and $\hat B_1$ that have small sum-set.

\begin{thm}[Balog-Szemer\'edi-Gowers theorem; \cite{TV}, Theorem 2.31]\label{BSGThm}
Let $A,B\subset\ZZ$. Suppose that
$$
|\{(a_1,a_2,b_1,b_2)\in A\times A\times B\times B\colon a_1+b_1=a_2+b_2\}|\geq K^{-1}|A|^{3/2}|B|^{3/2}.
$$
Then there exist sets $A^\prime\subset A,\ B^\prime\subset B$ with $|A^\prime|\geq K^{-C_1}|A|,\ |B^\prime|\geq K^{-C_1}|B|$, and
$$
d(A^\prime,B^\prime)\leq C_1\log K,
$$
where $C_1$ is an absolute constant and $d(\cdot,\cdot)$ denotes the Ruzsa distance.
\end{thm}

We will not define the  Ruzsa distance here. Instead, we will note that it is the input to the next theorem. For the interested reader, a definition of Ruzsa distance can be found in \cite{TV}, Chapter 2.

\begin{thm}[Pl\"unnecke-Ruzsa inequality; \cite{TV}, Proposition 2.27]\label{PRThm}
Let $A^\prime, B^\prime\subset\ZZ$, and suppose that $d(A^\prime, B^\prime)\leq C_1\log K$. Then there is a constant $C_2$ depending only on $C_1$ so that
\begin{equation}
\begin{split}
&|A^\prime+A^\prime|\leq K^{C_2}|A|,\\
&|B^\prime+B^\prime|\leq K^{C_2}|B|.
\end{split}
\end{equation}
\end{thm}

Applying Theorem \ref{BSGThm} and then \ref{PRThm} with $K\lessapprox_{\eps_1}1$, we obtain sets $\hat A_2\subset\hat A_1$ and $\hat B_2\subset\hat B_1$ with
\begin{itemize}
\item $|\hat A_2|\gtrapprox_{\eps_1}\tilde\delta^{-1/2}$,
\item $|\hat B_2|\gtrapprox_{\eps_1}\tilde\delta^{-1/2}$,
\item $|\hat A_2 + \hat A_2|\lessapprox_{\eps_1}\tilde\delta^{-1/2}$,
\item $|\hat B_2 + \hat B_2|\lessapprox_{\eps_1}\tilde\delta^{-1/2}$.
\end{itemize}
Let $A_2$ and $B_2$ be the subsets of $[0,1]$ associated to $\hat A_2$ and $\hat B_2$ under the maps $A_1\to \hat A_1$ and $B_1\to \hat B_1$ described above, and let $\pts_4 = \pts_3\cap (A_2\times B_2)$. Observe that for each $p\in\pts_4,\ |\lines(p)|\gtrapprox_{\eps_1}\tilde\delta^{-1/2}$, so in particular

\begin{equation}\label{mostIncidencesPreserved}
|I_{\tilde\delta}(\pts_4,\lines_1)|\gtrapprox_{\eps_1}\tilde\delta^{-3/2}.
\end{equation}

Let
$$
\Theta_2=\{\theta\in\Theta_1\colon I_{\tilde\delta}(\lines(\theta),\pts_4)\geq \tilde\delta^{-1+C_3\eps_1}\}.
$$
If we select $C_3$ sufficiently large then by \eqref{mostIncidencesPreserved}, $|\Theta_2|\gtrapprox_{\eps_1}|\Theta|$, so $\Theta_2$ still obeys \eqref{ThetaWellDistributed}.

Note that if $\theta\in\Theta_2$, then $|\pts(\theta)|\geq\tilde\delta^{C_3\eps_1}|\pts|$, and
$$
\mathcal{E}_{\tilde\delta}\big(\pi_\theta(\pts(\theta))\big)\leq |\lines(\theta)|\lessapprox_{\eps_1}\tilde\delta^{-1/2},
$$
where $\pi_\theta\colon\RR^2\to\RR$ is the orthogonal projection in the direction $\theta$, and $\mathcal{E}_{\tilde\delta}(\cdot)$ is the $\tilde\delta$-covering number.

\subsubsection{Bourgain's discretized projection theorem}
The following result was proved by Bourgain \cite{Bourgain}:
\begin{thm}\label{bourgainThm}
There exist absolute constants $s>0, \tau>0$, and $\rho_0>0$ so that the following holds for all $\rho\leq\rho_0$. Let $A,B\subset[0,1]$ be sets satisfying the following properties
\begin{itemize}
\item $|A|,|B|\leq\rho^{-1/2-s}$.
\item $\mathcal{E}_\rho(A+A)\leq\rho^{-1/2-s}$.
\item $\mathcal{E}_\rho(B+B)\leq\rho^{-1/2-s}$.
\end{itemize}
Let $\pts\subset A\times B$ with $|\pts|\geq\rho^{-1+s}$ and suppose that for all balls $B$ of radius $r\geq\rho$ we have
$$
|B\cap \pts|\leq r \rho^{-1 -s}.
$$

Let $\Theta\subset S^1$ be a set of points with $|\Theta|\geq\rho^{-1/4}.$ Suppose that for all arcs $I\subset S^1$ of length $\rho\leq r\leq \rho^{\tau}$, we have
$$
|\Theta\cap I|\leq \rho^{-s}|I|^{1/100}|\Theta|.
$$

Then there exists $\theta\in\Theta$ with the following property: if $\pts^\prime\subset\pts$ with $|\pts^\prime|\geq\rho^{s}|\pts|$, then
$$
\mathcal{E}_{\rho}(\pi_{\theta}(\pts^\prime))\geq\rho^{-1/2-s}.
$$
\end{thm}
Theorem \ref{bourgainThm} is a variant of Theorem 3 from \cite{Bourgain}. It is stated in the above form as Remarks (i) and (iii) in \cite[Section 7, p221]{Bourgain}. The following table will aid the reader when translating between the notation of the present paper and \cite{Bourgain}.
\begin{center}
\begin{tabular}{l|l}
\cite{Bourgain} & Theorem \ref{bourgainThm}\\
\hline
$\alpha$ & $1+O(s)$\\
$\kappa$ & 1/100\\
$\eta$ & $1/2+s$\\
$\tau$ & $\tau$\\
$\mathcal{B}$ & $\pts$\\
$\mathcal{A}^\prime$ & $\pts^\prime$
\end{tabular}
\end{center}


\subsection{Concluding the proof}
We are now ready to prove Proposition \ref{killingHeisenbergProp}. For the reader's convenience, we re-state it here
\begin{killingHeisenbergPropEnv}
There exist absolute constants $C$ (large) and $c>0$ (small) so that the following holds. For every $\alpha>0$, there is a $\delta_0>0$ so that if $0<\delta\leq\delta_0$, and if $(\tubes,Y)$ is an $\eps$--extremal collection of $\delta$ tubes that is of Heisenberg type with parameter $\alpha$ and $C\eps$, then $\eps>c\alpha$.
\end{killingHeisenbergPropEnv}
\begin{proof}
Let $s$, $\tau$, and $\rho_0$ be the constants from Theorem \ref{bourgainThm}. Suppose that Proposition \ref{killingHeisenbergProp} was false. Then by Lemma \ref{CartesianProductStructureNonConcentrated}, there is a number $\tilde\delta\leq\delta^{1-2\alpha}$, a constant $C_1$ and sets $A,B\subset[0,1]$, $\pts,\Theta,\lines$, that satisfy
\begin{itemize}
\item $|A|,|B|\leq\tilde\delta^{-1/2-\eps_2}$.
\item $\mathcal{E}_{\tilde\delta}(A+A)\leq\tilde\delta^{-1/2-\eps_2}$.
\item $\mathcal{E}_{\tilde\delta}(B+B)\leq\tilde\delta^{-1/2-\eps_2}$.
\item $\pts\subset A\times B,$ $|\pts|\geq\tilde\delta^{-1-\eps_2}$.
\item $|\Theta|\geq\tilde\delta^{-1/4}$, and if $I\subset S^1$ is an arc of length $\tilde\delta\leq r\leq\tilde\delta^\tau$, then
$$
|I\cap\Theta|\leq \tilde\delta^{-\eps_2}|I|^{1/100}|\Theta|.
$$
\item For each $\theta\in\Theta,$ there is a set $\pts(\theta)\subset\pts$ with $|\pts(\theta)|\geq \tilde\delta^{\eps_2}|\pts|$ such that
$$
\mathcal{E}_{\tilde\delta}(\pi_\theta(\pts(\theta)))\leq\tilde\delta^{-1/2-\eps_2}.
$$
\end{itemize}
In the items above, $\eps_2 = C_{\tau}\eps/\alpha,$ where $C_{\tau}$ is a constant that depends on $\tau$.

Let $c = s\alpha/C_\tau$. Then if $\eps< c\alpha$, the sets $\pts,\Theta$, and $\lines$ given above violate Theorem \ref{bourgainThm}, provided $\tilde\delta\leq\rho_0$, i.e. $\delta\leq\delta_0=\rho_0^{1/\alpha}$. This concludes the proof of Proposition \ref{killingHeisenbergProp}.
\end{proof}
\section{Killing the $SL_2$ example}\label{killingSl2ExampleSec}
In this section we will show that an $\eps$--extremal Besicovitch set of $SL_2$ type cannot exist. More precisely, we have the following.
\begin{prop}\label{killingSl2Prop}
For each constant $C$, there are positive constants $\alpha, \eps_0,\delta_0$ so that if $0<\eps\leq\eps_0$ and $0<\delta\leq\delta_0$, then there cannot exist an $\eps$--extremal  collection of $\delta$ tubes of $SL_2$ type with parameters $\alpha$ and $C\eps$.
\end{prop}

First, we will record a minor observation that will help simplify notation. If $(\tubes,Y)$ is an $\eps$-extremal set of tubes of  $SL_2$ type with parameters  $\alpha$ and $C\eps$, then $(\tubes,Y)$ is also a $\max(\eps,C\eps,\alpha)$-extremal set of tubes of  $SL_2$ type with parameters $\max(\eps,C\eps,\alpha)$ and $\max(\eps,C\eps,\alpha)$. Thus we can reduce Proposition \ref{killingSl2Prop} to the special case where $C=1$ and $\eps=\alpha$. We will call such sets ``$\eps$--extremal collections of tubes of $SL_2$ type.''

\subsection{Regulus strips are well-separated}\label{incidenceGeomRegStripsSec}
In this section we will show that after a refinement of $\mathcal{S}$, if $S_1,S_2\in\mathcal{S}$ are two regulus strips then the $\delta^{1/2}$ neighborhood of $S_1$ does not contain $S_2$, and vice-versa. This is a somewhat surprising statement, since the $\delta^{1/2}$ neighborhood of $S_1$ could contain as many as $\delta^{-1}$ tubes from $\tubes$ while still satisfying the Wolff axioms, so it could in theory contain as many as $\delta^{-1/2-\eps}$ regulus strips. Indeed, if the tubes from $\tubes$ were ``sticky'' in the sense of \cite{KLT}, then we would expect the $\delta^{1/2}$ neighborhood of $S_1$ to contain about $\delta^{-1}$ tubes, and thus to contain about $\delta^{-1/2}$ regulus strips. Thus a collection of tubes of $SL_2$ type is emphatically \emph{not} sticky.

\begin{figure}[h!]
 \centering
\begin{overpic}[width=0.4\textwidth]{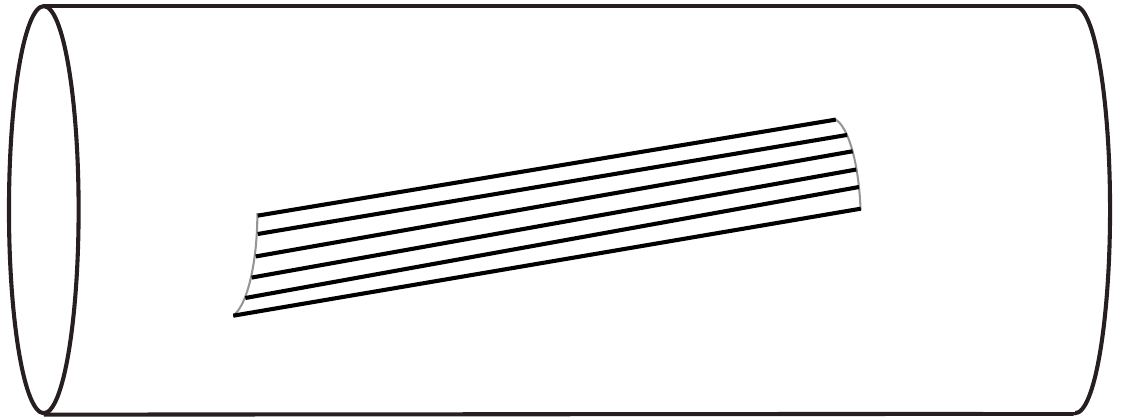}
\end{overpic}
 \caption{Each fat tube contains (at most) one regulus strip}\label{stripAlonePic}
\end{figure}


\begin{defn}
If $S$ is a regulus strip, we define the direction of $S$ to be the direction of the $\delta^{1/2}$-tube containing $S$. We denote this direction by $v(S)$. The vector $v(S)$ is defined up to uncertainty $O(\delta^{1/2})$. 
\end{defn}

Let $S_1$ and $S_2$ be regulus strips with $\angle(v(S_1),v(S_2))\geq\delta^{C\eps}$. If $S_1\cap S_2\neq\emptyset$, then the angle $\angle(T_pS_1,T_pS_2)$ is approximately constant, independent of the choice of $p\in S_1\cap S_2$. More precisely, the angle can vary by $\lessapprox_\eps\delta^{1/2}$.
\begin{defn}
Let $(\tubes=\bigsqcup_{S\in\mathcal{S}}\tubes(S),Y)$ be a set of tubes of $SL_2$ type. We say that the regulus strips in $(\tubes,Y)$ have tangential intersection (at angle $t$) if for all pairs $S_1,S_2\in\mathcal{S}$ satisfying
$$
\angle(v(S_1), v(S_2))\geq t,
$$
we have
\begin{equation}\label{largeAngle}
\angle\big(T_pS_1,\ T_pS_2\big)\lessapprox_{\eps}\delta^{1/2}\quad\textrm{for all}\ p\in \bigcup_{\tube\in\tubes(S_1)}Y(\tube)\cap\bigcup_{\tube\in\tubes(S_2)}Y(\tube).
\end{equation}
Note that this condition is only interesting if $\bigcup_{\tube\in\tubes(S_1)}Y(\tube)\cap\bigcup_{\tube\in\tubes(S_2)}Y(\tube)\neq\emptyset$. In practice, we will have $t=\delta^{C\eps}$ for some absolute constant $C$.
\end{defn}

\begin{lem}[Regulus strips are tangent whenever they intersect]\label{regStripsAreTangent}
Let $(\tubes=\bigsqcup_{S\in\mathcal{S}}\tubes(S),Y)$ be an $\eps$--extremal set of tubes of $SL_2$ type. Then there is a constant $C$ and a refinement $(\tubes,Y^\prime)$ so that the regulus strips in $(\tubes,Y^\prime)$ have tangential intersection at angle $\delta^{C\eps}$.
\end{lem}
\begin{proof}
Apply Lemma \ref{robustTransLem} to $(\tubes,Y)$ and let $(\tubes,Y_1)$ be the resulting refinement. For each $S\in\mathcal{S}$, define the shading
$$
Y_1(S)=\bigcup_{\tube\in\tubes(S)}Y_1(\tube).
$$
Note that $\sum_{S\in\mathcal{S}}|Y_1(S)|\gtrapprox_{\eps}1$, and that for all $p\in\RR^3$,
\begin{equation}\label{boundedOverlap}
\sum_{S\in\mathcal{S}}\chi_{Y_1(S)}(p)\lessapprox_\eps\delta^{-1/2}.
\end{equation}

Cover each regulus strip $S\in\mathcal{S}$ by $\approx_\eps\delta^{-1/2}$ rectangular prisms of dimensions $\delta^{1/2}\times\delta^{1/2}\times\delta$. If $W$ is one of these prisms, then $W$ is comparable to the intersection of a ball of radius $\delta^{1/2}$ with the $\delta$--neighborhood of a plane $\Pi(W)$. If $p\in S$ and if $W=W(p,S)$ is the prism containing $p$, then $\angle(T_pS, \Pi(W))\lesssim\delta^{1/2}$.

Define the shading
\begin{equation}\label{defnY2}
Y_2(S)=\{p\in Y(S)\colon |Y_1(S)\cap W(p,S)|\geq\delta^{C_1\eps+2}\}.
\end{equation}
If $C_1$ is selected sufficiently large then $\sum_{S\in\mathcal{S}}|Y_2(S)|\gtrapprox_{\eps}1.$

For each $\delta^{1/2}\leq t\leq 1$, define
$$
A_t=\{(S_1,S_2)\in\mathcal{S}^2\colon Y_2(S_1)\cap Y_2(S_2)\neq\emptyset,\ \angle(S_1,S_2)\geq\delta^{C\eps},\ t\leq \angle(T_pS_1,T_pS_2)< 2t\},
$$
where $C$ is a constant to be determined later. In the above definition, the point $p$ is chosen in $S_1\cap S_2$. Choosing a different $p$ could change $\angle(T_pS_1,T_pS_2)$ by $\lessapprox_{\eps}\delta^{1/2}$. For values of $t$ close to $\delta^{1/2}$ this might change which of the above sets the pair $(S_1,S_2)$ is assigned to, but this will not matter for our proof.

Since the tubes in $\tubes$ satisfy the conclusions of Lemma \ref{robustTransLem}, we have that at least half the pairs of intersecting tubes make an angle $\gtrapprox_{\eps}1$. Thus if the constant $C$ is chosen sufficiently large, we have
 $$
\sum_{\substack{t\ \textrm{dyadic}\\ \delta^{1/2}\leq t\leq 1}}\sum_{(S_1,S_2)\in A_t}|Y_2(S_1)\cap Y_2(S_2)|\approx_{\eps}\sum_{(S_1,S_2)\in\mathcal{S}^2}|Y_2(S_1)\cap Y_2(S_2)|\approx_{\eps}\delta^{-1/2}.
 $$
After pigeonholing, there is a number $t_0$ so that
\begin{equation}\label{valueOfBeta}
\sum_{(S_1,S_2)\in A_{t_0}}|Y_2(S_1)\cap Y_2(S_2)|\approx_{\eps}\delta^{-1/2}.
\end{equation}

We will show that $t_0\lessapprox_{\eps}\delta^{1/2}$. We will do this by proving that the $\delta^{1/2}t_0$--neighborhood of  most tubes in $\tubes$ has large intersection with $\bigcup_{\tube\in\tubes}Y(\tube)$. This will mean that
$$
\Big|\bigcup_{\tube\in\tubes}Y(\tube)\Big|\gtrapprox_{\eps}(\delta^{1/2}t_0)^{1/2},
$$
and this contradicts the fact that $(\tubes,Y)$ is $\eps$--extremal, unless $t_0\lessapprox_{\eps}\delta^{1/2}$.

Observe that if $\angle(S_1,\ S_2)\geq\delta^{C\eps}$ and $\angle(T_pS_1,\ T_pS_2)\approx_{\eps}t_0$ for some point $p\in S_1\cap S_2$, then $S_1\cap S_2$ is contained in a rectangular prism of dimensions $\approx_{\eps} \delta^{1/2}\times\delta t_0^{-1}\times\delta$, so $|S_1\cap S_2|\lessapprox_{\eps}\delta^{5/2}t_0^{-1}$.

Let $\mathcal{S}_1\subset\mathcal{S}$ be the set of strips $S_1\in\mathcal{S}$ satisfying
$$
\sum_{\substack{S_2\in\mathcal{S}\\ (S_1,S_2)\in A_{t_0}}}|Y_2(S_1)\cap Y_2(S_2)|\geq \delta^{C_2\eps-1/2}|\mathcal{S}|^{-1}.
$$
If the constant $C_2$ is selected sufficiently large, then $|\mathcal{S}_1|\gtrapprox_{\eps}\delta^{-3/2}$. Since each regulus strip $S_2\in\mathcal{S}$ can contribute $\lessapprox_{\eps}\delta^{5/2}t_0^{-1}$ to the above sum, we have that each regulus strip $S_1\in\mathcal{S}_1$ intersects at least $\gtrapprox_{\eps} \delta/(\delta^{5/2}t_0^{-1})=\delta^{-3/2}t_0$ other strips $S_2$ that satisfy $(S_1,S_2)\in A_{t_0}$. If $Y_2(S_1)\cap Y_2(S_2)\neq\emptyset$, and if $p_{S_1,S_2}\in S_1\cap S_2$ is a point of intersection, then by \eqref{defnY2},
$$
|B(p_{S_1,S_2},\delta^{1/2})\cap Y_2(S_2)|\gtrapprox_{\eps}\delta^2.
$$
Thus for each $S_1\in\mathcal{S}_1$, we have
$$
\sum_{\substack{S_2\in\mathcal{S}\\ (S_1,S_2)\in A_{t_0}}}|Y_2(S_2)\cap B(p_{S_1,S_2},\delta^{1/2})|\gtrapprox_{\eps} \delta^2 (\delta^{-3/2}t_0^{-1})=\delta^{1/2}t_0.
$$

By \eqref{boundedOverlap}, we have
$$
\Big|\bigcup_{\substack{S_2\in\mathcal{S}\\ (S_1,S_2)\in A_{t_0}}}Y_2(S_2)\cap B(p_{S_1,S_2},\delta^{1/2})\Big|\gtrapprox_{\eps} \delta t_0.
$$
But observe that the above set is contained in $N_{t_0\delta^{1/2}}(S_1)$, and $|N_{t_0\delta^{1/2}}(S_1)|\approx_{\eps}\delta t_0$.

We conclude that
$$
\Big|N_{t_0 \delta^{1/2}}(S_1)\cap\bigcup_{S\in\mathcal{S}}Y_2(S)\Big|\gtrapprox_{\eps} |N_{t_0 \delta^{1/2}}(S_1)|.
$$
In particular, this means we can find a refinement $\tubes_1\subset\tubes$ so that for each $\tube\in\tubes_1$,
$$
\Big|N_{t_0\delta^{1/2}}(\tube)\cap \bigcup_{\tube\in\tubes}Y(\tube)\Big|\gtrapprox_{\eps} |N_{t_0 \delta^{1/2}}(\tube)|.
$$
By Theorem \ref{WolffBd}, this in turn implies that
$$
\Big|\bigcup_{\tube\in\tubes}Y(\tube)\Big|\gtrapprox_{\eps} (t_0\delta^{1/2})^{1/2} = t_0^{1/2}\delta^{1/4}.
$$
Since $(\tubes,Y)$ was $\eps$--extremal, this is a contradiction unless $t_0\lessapprox_{\eps}\delta^{1/2}$.

Finally, we will refine the shadings $\{Y_2(S)\}$. By \eqref{valueOfBeta} and pigeonholing, we can select a set $A\subset\bigcup_{S\in\mathcal{S}}Y_2(S)$ with $|A|\gtrapprox_{\eps}\delta^{1/2}$ so that for each $p\in A$ there is an element $S_p\in\mathcal{S}$ so that
$$
|\{S\in\mathcal{S}\colon p\in Y_2(S),\ (S_p, S)\in A_{t_0}\}|\gtrapprox_{\eps}\delta^{-1/2}.
$$
For each $S\in\mathcal{S}$, define the shading
$$
Y(S)=\{p\in Y_2(S)\cap A \colon (S_p, S)\in A_{t_0}\}.
$$
By construction, for every pair $S_1,S_2\in\mathcal{S}$ and every $p\in Y(S_1)\cap Y(S_2)$, we have
$$
\angle( T_pS_1,\ T_pS_2)\lessapprox_{\eps} t_0\lessapprox_{\eps}\delta^{1/2}.
$$
Finally, for each $S\in\mathcal{S}$ and each $\tube\in\tubes(S),$ define $Y^\prime(\tube)=Y_1(\tube)\cap Y(S)$.
\end{proof}
In the following arguments, we will need to consider the Besicovitch set at scale $\delta$ and also at scale $\delta^{1/2}$. Let $\mathcal{S}$ be a set of regulus strips. For each $S\in\mathcal{S}$, define $\tube_{\delta^{1/2}}(S)$ to be a $\delta^{1/2}$ tube that contains $S$ (the exact choice of $\tube_{\delta^{1/2}}(S)$ is not important; for any two choices, one will be contained in the 100-fold dilate of the other). In what follows, we will call $\delta^{1/2}$--tubes ``fat tubes,'' and we will call $\delta$--tubes ``thin tubes.''

We can select the tubes $\{\tube_{\delta^{1/2}}(S)\colon S\in\mathcal{S}\}$ so that after refining $\mathcal{S}$ by a factor of $O(1)$, for every pair $S_1,S_2\in\mathcal{S}$ we either have that $\tube_{\delta^{1/2}}(S_1)=\tube_{\delta^{1/2}}(S_2)$ or the 100-fold dilate of $\tube_{\delta^{1/2}}(S_1)$ does not contain $\tube_{\delta^{1/2}}(S_2)$ and vice-versa. After performing this refinement, define
$$
\tubes_{\mathrm{fat}}=\{\tube_{\delta^{1/2}}(S)\colon S\in\mathcal{S}\}.
$$
This is the set of fat tubes associated to $\mathcal{S}$. In practice, the set $\mathcal{S}$ will be apparent from context.

For each fat tube $\tube_{\delta^{1/2}}\in \tubes_{\mathrm{fat}}$, define
$$
\mathcal{S}(\tube_{\delta^{1/2}})=\{S\in\mathcal{S}\colon \tube_{\delta^{1/2}}(S) = \tube_{\delta^{1/2}}\}.
$$
This is the set of regulus strips contained in $\tube_{\delta^{1/2}}$. Define
$$
\tubes(\tube_{\delta^{1/2}})=\bigcup_{S\in\mathcal{S}(\tube_{\delta^{1/2}})}\tubes(S).
$$
This is the set of thin tubes contained in $\tube_{\delta^{1/2}}$.

The following lemma says that if $\tube_{\delta^{1/2}}\in\tubes_{\mathrm{fat}}$, and if $S_1,S_2\in\mathcal{S}(\tube_{\delta^{1/2}})$, then the tangent planes of $S_1$ and $S_2$ at nearby points are nearly parallel.
\begin{lem}\label{tanPlanesParallelLem}
Let $(\tubes=\bigsqcup_{S\in\mathcal{S}}\tubes(S),Y)$ be an $\eps$--extremal set of tubes of $SL_2$ type. Suppose that the regulus strips in $(\tubes,Y)$ have tangential intersection at angle $\delta^{C\eps}$. Let $\tubes_{\mathrm{fat}}$ be the set of fat tubes associated to $\mathcal{S}$.

Then there is a refinement $(\tubes,Y^\prime)$ of $(\tubes,Y)$ so that the following holds. For each $\tube_{\delta^{1/2}}\in\tubes_{\mathrm{fat}}$ and for each $S_1,S_2\in\mathcal{S}(\tube_{\delta^{1/2}})$, if $p\in \bigcup_{\tube\in\tubes(S_1)}Y^\prime(\tube)$ and $q \in \bigcup_{\tube\in\tubes(S_2)}Y^\prime(\tube)$ with $\dist(p,q)\leq\delta^{1/2}$, then
\begin{equation}\label{stripsInsideFatTubeAreParallel}
\angle(T_pS_1, T_qS_2)\lessapprox_{\eps}\delta^{1/2}.
\end{equation}
\end{lem}
\begin{proof}
For each $S\in\mathcal{S}$, define $Y(S)=\bigcup_{\tube\in\tubes(S)}Y(\tube)$. Let $Y_1(S)\subset Y(S)$ be a refinement of the shadings with $\sum_{S\in\mathcal{S}}\chi_{Y_1(S)} \sim\mu\chi_A$ for some number $\mu\approx_{\eps}\delta^{-1/2}$ and some set $A$ with $|A|\approx_{\eps}\delta^{1/2}$.

For each $S_0\in\mathcal{S}$, define the hairbrush
$$
H_{Y_1}(S_0)=\{S\in\mathcal{S}\colon Y_1(S_0)\cap Y_1(S)\neq\emptyset\}.
$$

Select a regulus strip $S_0\in\mathcal{S}$ satisfying
$$
\sum_{S\in H(S_0)}|Y_1(S)|\geq \delta^{C_1\eps + 1/2}.
$$
If the constant $C_1$ is chosen sufficiently large then at least one such regulus strip must exist.

For each $S\in \mathcal{S}\backslash H_{Y_1}(S_0)$ that is $\geq\delta^{C_2\eps}$ separated and skew from $S_0$, define
$$
Y_2(S)=Y_1(S)\cap \bigcup_{S^\prime\in H(S_0)}Y_1(S^\prime).
$$
If $S$ is not $\geq\delta^{C_2\eps}$ separated and skew from $S_0$ then define $Y_2=\emptyset$. If the constant $C_2$ is chosen sufficiently large then
$$
\sum_{S\in\mathcal{S}}|Y_2(S)|\gtrapprox_{\eps}1.
$$

Next, let $\tube_{\delta^{1/2}}\in\tubes_{\mathrm{fat}}$. Let $S_1,S_2\in\mathcal{S}(\tube_{\delta^{1/2}})$. Let $p\in Y_2(S_1),$ and let $q \in Y_2(S_2)$ be points with $\dist(p,\ q)\leq\delta^{1/2}$. Then there at least one regulus strip from $H(S_0)$ that intersects $S_1$ at the point $p$. This regulus strip must be contained in some tube from $\tubes_{\mathrm{fat}}$; call this tube $\tube_{\delta^{1/2}}^\prime$. Though there may be many regulus strips from $H(S_0)$ that intersect $S_1$ at $p$, the corresponding fat tubes containing these regulus strips will all point in the same direction up to error $\lessapprox_{\eps}\delta^{1/2}$. Thus up to error $\lessapprox_{\eps}\delta^{1/2}$, we have that $T_p(S_1)$ is the span of the vectors $v(\tube_{\delta^{1/2}})$ and $v(\tube_{\delta^{1/2}}^\prime)$.

Since $\dist(p, q)\leq\delta^{1/2}$, and since $\tube_{\delta^{1/2}}$ is $\geq\delta^{C_2\eps}$ separated and skew to the tube containing $S_0$, the regulus strip in $H(S_0)$ that intersects $S_2$ at the point $q$ is contained in some fat tube $\tube_{\delta^{1/2}}^{\prime\prime}$ with $\angle(v(\tube_{\delta^{1/2}}^\prime),\ \tube_{\delta^{1/2}}^{\prime\prime})\lessapprox_{\eps}\delta^{1/2}$.  Thus up to error $\lessapprox_{\eps}\delta^{1/2}$, we have that $T_p(S_1)$ is the span of the vectors $v(\tube_{\delta^{1/2}})$ and $v(\tube_{\delta^{1/2}}^\prime)$. In particular, this means that $\angle(T_p(S_1),\ T_q(S_2))\lessapprox_{\eps}\delta^{1/2}$.

Finally, for each $S\in\mathcal{S}$ and each $\tube\in\tubes(S)$, define $Y^\prime(\tube)=Y(\tube)\cap Y_2(S)$.
\end{proof}
\begin{defn}
Let $(\tubes=\bigsqcup_{S\in\mathcal{S}}\tubes(S),Y)$ be an $\eps$--extremal set of tubes of $SL_2$ type. We say that $(\tubes,Y)$ obeys the regulus map at coarse resolution if $(\tubes,Y)$ satisfies the conclusions of Lemma \ref{tanPlanesParallelLem}.
\end{defn}
\begin{rem}
The property of obeying the regulus map at coarse resolution is similar to the ``consistency of planes and squares'' property discussed in \cite[Section 9]{KLT}. However, the conditions under which these two properties hold are rather different.
\end{rem}
The main consequence of Lemma \ref{tanPlanesParallelLem} is that for each fat tube $\tube_{\delta^{1/2}}\in \tubes_{\mathrm{fat}}$, there is a well-defined regulus map $p\mapsto T_pR(\tube_{\delta^{1/2}})$: to compute this regulus map, select any regulus strip $S\in\mathcal{S}(\tube_{\delta^{1/2}})$. Then the regulus map $T_pR(\tube_{\delta^{1/2}})$ is simply the tangent plane $T_qS$, where $q\in S$ is a point with $\dist(p,q)\lessapprox_{\eps}\delta^{1/2}$. This regulus map is defined up to error $\lessapprox_{\eps}\delta^{1/2}$. In particular, if $S\in\mathcal{S}(\tube_{\delta^{1/2}})$, then
\begin{equation}\label{regulusMapFatTubes}
\angle\big(T_p(S),\ T_p(R(T_{\delta^{1/2}}))\big)\lessapprox_{\eps}\delta^{1/2}\quad\textrm{for all}\ p\in \bigcup_{\tube\in\tubes(S)}Y(\tube).
\end{equation}
This gives us the following analogue of Corollary \ref{hairbrushTwoTubesInFatRegulusOne}, which has an identical proof; in the corollary below, condition \eqref{thinTubeIntersectsFatTubes} is the analogue of the statement that $\tube\in H(\tube_1,\tube_2)$.
\begin{lem}\label{fatRegulusOfFatTubes}
Let $(\tubes=\bigsqcup_{S\in\mathcal{S}}\tubes(S),Y)$ be an $\eps$--extremal set of tubes of $SL_2$ type. Suppose that the regulus strips in $(\tubes,Y)$ have tangential intersection at angle $\delta^{C_1\eps}$, and that $(\tubes,Y)$ obeys the regulus map at coarse resolution. Let $\tube_{\delta^{1/2}},\tube_{\delta^{1/2}}^\prime\in\tubes_{\mathrm{fat}}$ be $\delta^{1/2}$ tubes that are $\geq\delta^{C_2\eps}$ separated and skew.

Then there is a regulus $R=R(\tube_{\delta^{1/2}},\tube_{\delta^{1/2}}^\prime)$ so that every tube $\tube\in\tubes$ satisfying
\begin{equation}\label{thinTubeIntersectsFatTubes}
Y(\tube)\ \ \ \cap\bigcup_{\tube^\prime\in\tubes(\tube_{\delta^{1/2}})}Y(\tube^\prime)\neq\emptyset,\quad \ \ \textrm{and}\quad \ \ Y(\tube)\ \ \ \cap\bigcup_{\tube^\prime\in\tubes(\tube_{\delta^{1/2}}^\prime)}Y(\tube^\prime)\neq\emptyset
\end{equation}
is contained in the $\lessapprox_{\eps}\delta^{1/2}$--neighborhood of $R$.

Furthermore, the fat tubes from $\tubes_{\mathrm{fat}}$ containing the tubes satisfying \eqref{thinTubeIntersectsFatTubes} intersect with multiplicity $\lessapprox_{\eps}1$---they are essentially the $\delta^{1/2}$ neighborhood of lines in the ruling of $R(\tube_{\delta^{1/2}},\tube_{\delta^{1/2}}^\prime)$.
\end{lem}

We are now ready to prove the main result of this section.
\begin{lem}[Regulus strips are $\delta^{1/2}$--separated]\label{regStripsSeparatedLem}
Let $(\tubes=\bigsqcup_{S\in\mathcal{S}}\tubes(S),Y)$ be an $\eps$--extremal set of tubes of $SL_2$ type. Suppose that the regulus strips in $(\tubes,Y)$ have tangential intersection at angle $\delta^{C\eps}$, and that $(\tubes,Y)$ obeys the regulus map at coarse resolution. Let $\tubes_{\mathrm{fat}}$ be the associated set of fat tubes.

Then there is a set $\mathcal{S}^\prime\subset\mathcal{S}$ with
$$
\sum_{S\in\mathcal{S}^\prime}\sum_{\tube\in\tubes(S)}|Y(\tube)|\gtrapprox_{\eps} 1
$$
so that each fat tube $\tube_{\delta^{1/2}}\in \tubes_{\mathrm{fat}}$ contains at most one regulus strip from $\mathcal{S}^\prime$.
\end{lem}
\begin{proof}

For each fat tube $\tube_{\delta^{1/2}}\in \tubes_{\mathrm{fat}}$, let $Y^\prime$ be a refinement of the shadings of the tubes in $\tubes(\tube_{\delta^{1/2}})$ so that
$$
\sum_{\tube\in \tubes(\tube_{\delta^{1/2}})}\chi_{Y^\prime(\tube)}\sim \mu_{\tube_{\delta^{1/2}}}\chi_{A_{\tube_{\delta^{1/2}}}}
$$
for some number $\mu_{\tube_{\delta^{1/2}}}$ and some set $A_{\tube_{\delta^{1/2}}}$. After dyadic pigeonholing and refining the set $\tubes_{\mathrm{fat}}$ slightly, we can assume that there is number $\mu_{\operatorname{fine}}$ so that $\mu_{\tube_{\delta^{1/2}}}\sim \mu_{\operatorname{fine}}$ for each $\tube_{\delta^{1/2}}\in \tubes_{\mathrm{fat}}$.

Let $\tube_{\delta^{1/2}}\in \tubes_{\mathrm{fat}}$. Let $S\in\mathcal{S}(\tube_{\delta^{1/2}})$. Let $R$ be the regulus associated to $S$, and let $L$ be a line in the ruling of $R$ dual to the ruling containing the tubes of $S$. Thus $L$ makes an angle $\gtrapprox_{\eps}1$ with every tube in $\tubes(S)$, and $L$ intersects the $\approx_{\eps}\delta$ neighborhood of every tube in $\tubes(S)$). Furthermore, we can choose $L$ so that
\begin{equation}\label{LCapAIsBig}
|L\cap A_{\tube_{\delta^{1/2}}}|\approx_{\eps}\delta^{1/2}
\end{equation}
(recall that the line $L$ intersects $S$ in a line-segment of length roughly $\delta^{1/2}$, so \eqref{LCapAIsBig} asserts that most of this intersection is contained in $A_{\tube_{\delta^{1/2}}}$).

There are $\gtrapprox_{\eps}\mu_{\operatorname{fine}}\delta^{-1/2}$ tubes from $\tubes(\tube_{\delta^{1/2}})$ that intersect $L$. Since this set of tubes obeys the Wolff axioms, and each of these tubes makes an angle $\gtrapprox_{\eps}1$ with $L$, we can apply Lemma \ref{WolffHairbrushBd} to conclude that the union of these tubes has volume $\gtrapprox_{\eps}\mu_{\operatorname{fine}}\delta^{3/2}$.

For a $\gtrapprox_{\eps}1$ fraction of these tubes, we have $|Y^\prime(\tube)\cap A_{\tube_{\delta^{1/2}}}|\gtrapprox_{\eps}|\tube|\sim\delta^2$, so
$$
|A_{\tube_{\delta^{1/2}}}|\gtrapprox_{\eps}\delta^{3/2}\mu_{\operatorname{fine}}.
$$
We conclude that
\begin{equation}\label{sumOverTubesInFatTubeMass}
\sum_{\tube\in\tubes(\tube_{\delta^{1/2}})}|Y^\prime(\tube)|\approx_{\eps}\mu_{\operatorname{fine}}|A_{\tube_{\delta^{1/2}}}|\gtrapprox_{\eps}\mu_{\operatorname{fine}}^2\delta^{3/2}.
\end{equation}
 Since the tubes in each regulus strip $S\in\mathcal{S}(\tube_{\delta^{1/2}})$ contribute $\lessapprox_{\eps}\delta^{3/2}$ to the above sum, we conclude that
 \begin{equation}\label{lowerBoundNumberOfRegulusStrips}
 |\mathcal{S}(\tube_{\delta^{1/2}})|\gtrapprox_{\eps}\mu_{\operatorname{fine}}^2.
 \end{equation}
Inequality \eqref{lowerBoundNumberOfRegulusStrips} holds for each $S\in\mathcal{S}$. Since $|\mathcal{S}|\lessapprox_{\eps}\delta^{-3/2}$ and the sets $\{\mathcal{S}(\tube_{\delta^{1/2}})\colon \tube_{\delta^{1/2}}\in\tubes_{\operatorname{fat}}\}$ are disjoint, we have that
\begin{equation}\label{boundOnNumberFatTubes}
|\tubes_{\mathrm{fat}}|\lessapprox_{\eps}\delta^{-3/2}\mu_{\operatorname{fine}}^{-2}.
\end{equation}

For each $\tube_{\delta^{1/2}}\in\tubes_{\mathrm{fat}}$, define the hairbrush
$$
H(\tube_{\delta^{1/2}})=\Big\{\tube_{\delta^{1/2}}^\prime\in \tubes_{\mathrm{fat}} \colon \bigcup_{\tube\in\tubes(\tube_{\delta^{1/2}}^\prime)}Y^\prime(\tube)\ \ \cap  \bigcup_{\tube\in\tubes(\tube_{\delta^{1/2}})}Y^\prime(\tube) \neq\emptyset \Big\}.
$$

Next, we will count quadruples $(\tube_{\delta^{1/2}},\tube_{\delta^{1/2}}^\prime,\tube_{\delta^{1/2}}^{\prime\prime},\tube_{\delta^{1/2}}^{\prime\prime\prime})\in(\tubes_{\mathrm{fat}})^4$ where $\tube_{\delta^{1/2}}^\prime,\tube_{\delta^{1/2}}^{\prime\prime},\tube_{\delta^{1/2}}^{\prime\prime\prime}$ are in the hairbrush of $\tube_{\delta^{1/2}}$, and where the hairbrush of $\tube_{\delta^{1/2}}^\prime,\tube_{\delta^{1/2}}^{\prime\prime},\tube_{\delta^{1/2}}^{\prime\prime\prime}$ contains few fat tubes. Define

\begin{equation}\label{defnOfQ}
\begin{split}
\mathcal{Q} = \big\{(\tube_{\delta^{1/2}},&\tube_{\delta^{1/2}}^\prime,\tube_{\delta^{1/2}}^{\prime\prime},\tube_{\delta^{1/2}}^{\prime\prime\prime})\in (\tubes_{\mathrm{fat}})^4\colon \tube_{\delta^{1/2}}^\prime,\tube_{\delta^{1/2}}^{\prime\prime},\tube_{\delta^{1/2}}^{\prime\prime\prime}\in H(\tube_{\delta^{1/2}});\\
& \tube_{\delta^{1/2}}^\prime,\tube_{\delta^{1/2}}^{\prime\prime},\tube_{\delta^{1/2}}^{\prime\prime\prime}\ \textrm{are}\ \geq \delta^{C_1\eps}\ \textrm{separated and skew};\\& \angle\big(v(\tube_{\delta^{1/2}}^{\prime\prime\prime}),\ T_p(R(\tube_{\delta^{1/2}}^\prime,\tube_{\delta^{1/2}}^{\prime\prime})) \big)\geq \delta^{C_1\eps}\ \textrm{for all}\ p\in \tube_{\delta^{1/2}}^{\prime\prime\prime}\cap R(\tube_{\delta^{1/2}}^\prime,\tube_{\delta^{1/2}}^{\prime\prime}) \big\}.
\end{split}
\end{equation}
The constant $C_1$ will be chosen below.

First, if $(\tube_{\delta^{1/2}},\tube_{\delta^{1/2}}^\prime,\tube_{\delta^{1/2}}^{\prime\prime},\tube_{\delta^{1/2}}^{\prime\prime\prime})\in\mathcal{Q}$, then there are $\lessapprox_{\eps}1$ possible tubes $\tube_{\delta^{1/2}}^*\in\tubes_{\mathrm{fat}}$ for which $(\tube_{\delta^{1/2}}^*,\tube_{\delta^{1/2}}^\prime,\tube_{\delta^{1/2}}^{\prime\prime},\tube_{\delta^{1/2}}^{\prime\prime\prime})\in\mathcal{Q}$ (the implicit constant in the $\lessapprox_{\eps}$ depends on $C_1$). This is because the set of fat tubes that are incident to $\tube_{\delta^{1/2}}^\prime$ and $\tube_{\delta^{1/2}}^{\prime\prime}$ intersect with multiplicity $\lessapprox_{\eps}1$, and the tube $\tube_{\delta^{1/2}}^*$ must also intersect the set $\tube_{\delta^{1/2}}^{\prime\prime\prime}\cap N_{\delta^{1/2}}\big( R(\tube_{\delta^{1/2}}^\prime,\tube_{\delta^{1/2}}^{\prime\prime})\big)$. The latter set is contained in a ball of radius $\lessapprox_{\eps}\delta^{1/2}$.

Thus we have
\begin{equation}\label{QSmall}
|\mathcal{Q}| \leq \delta^{-\eps}|\tubes_{\mathrm{fat}}|^3.
\end{equation}

Next we will compute a lower bound for $|\mathcal{Q}|$. By Remark \ref{lotsOfProperlyIntersectingQuadruples} there are $\gtrapprox_{\eps}\delta^{-13/2}$ quadruples $(\tube,\tube_1,\tube_2,\tube_3)\in\tubes^4$ with $\tube\in H(\tube_1,\tube_2,\tube_3)$ and $\angle(v(\tube_3), T_pR(\tube_1,\tube_2))\geq\delta^{C_1\eps}$ for all points $p\in\tube_3\cap R(\tube_1,\tube_2)$. We will show that this large set of quadruples of $\delta$-tubes forces the existence of a large set of fat tubes with the same properties.

Observe that if $ (\tube_{\delta^{1/2}},\tube_{\delta^{1/2}}^\prime,\tube_{\delta^{1/2}}^{\prime\prime},\tube_{\delta^{1/2}}^{\prime\prime\prime})\in\mathcal{Q}$, then $\lessapprox_{\eps}\delta^{-7/2}\mu_{\operatorname{fine}}^3|\tubes_{\mathrm{fat}}|$ of the above quadruples of thin tubes $(\tube,\tube_1,\tube_2,\tube_3)$ can satisfy
\begin{equation*}
\tube\in\tubes(\tube_{\delta^{1/2}}),\quad \tube_1\in\tubes(\tube_{\delta^{1/2}}^\prime),\quad \tube_2\in \tubes(\tube_{\delta^{1/2}}^{\prime\prime}),\quad \textrm{and}\quad\tube_3\in \tubes(\tube_{\delta^{1/2}}^{\prime\prime\prime}).
\end{equation*}
To see this, note that $|\tubes(\tube_{\delta^{1/2}})|\lessapprox_{\eps}\delta^{-2}|\tubes_{\mathrm{fat}}|$. For each pair $(\tube,\tube_1)$ with $\tube_1\in H(\tube)$, let $\tube_{\delta^{1/2}}^\prime$ be the fat tube containing $\tube_1$. There are $\lessapprox_{\eps}\mu_{\operatorname{fine}}\delta^{-1/2}$ tubes $\tube_1^\prime\in\tubes(\tube_{\delta^{1/2}}^\prime)$ that satisfy $\tube_1^\prime\in H(\tube)$. Similarly for $\tube_2$ and $\tube_3$.

Thus
\begin{equation}\label{QBig}
|\mathcal{Q}|\gtrapprox_{\eps}\delta^{-13/2}\Big(\delta^{-7/2}\mu_{\operatorname{fine}}^3|\tubes_{\mathrm{fat}}|\Big)^{-1}=\delta^{-3}\mu_{\operatorname{fine}}^3|\tubes_{\mathrm{fat}}|.
\end{equation}

Combining \eqref{QSmall} and \eqref{QBig} we obtain
\begin{equation*}
\delta^{-3}\mu_{\operatorname{fine}}^{-3}|\tubes_{\mathrm{fat}}| \lessapprox_\eps |\tubes_{\mathrm{fat}}|^3,
\end{equation*}
so by \eqref{boundOnNumberFatTubes},
$$
\delta^{-3}\mu_{\operatorname{fine}}^{-3}\lessapprox_{\eps}|\tubes_{\mathrm{fat}}|^2\lessapprox_{\eps}\delta^{-3}\mu_{\operatorname{fine}}^{-4}
$$
and thus
\begin{equation}
\mu_{\operatorname{fine}}\lessapprox_{\eps}1,\quad |\tubes_{\mathrm{fat}}|\gtrapprox_{\eps}\delta^{-3/2}.
\end{equation}
This implies that each tube in $|\tubes_{\mathrm{fat}}|$ contains $\lessapprox_{\eps}1$ regulus strips from $\mathcal{S}$. Select one regulus strip from each fat tube and denote the resulting set of regulus strips by $\mathcal{S}^\prime$.
\end{proof}
\subsection{Discretization at scale $\delta^{1/2}$}\label{discetizationScaleDelta12Sec}
\begin{defn}
We define a ``grain'' to be a set $G$ of the form $Q\cap N_{\delta^{1-C\eps}}(\Pi)$, where $Q=Q(G)$ is a cube of side-length $\delta^{1/2}$ that is aligned with the grid $(\delta^{1/2}\ZZ)^3$, $\Pi=\Pi(G)$ is a plane, and $C$ is an absolute constant.
\end{defn}

We will choose the constant $C$ in the above definition sufficiently large so that the following holds. Let $(\tubes=\bigsqcup_{S\in\mathcal{S}}\tubes(S),Y)$ be an $\eps$--extremal set of tubes of $SL_2$ type, and suppose that the regulus strips in $(\tubes,Y)$ have tangential intersection at angle $\delta^{C_1\eps}$. If $S_1,S_2\in\mathcal{S}$ with $\angle(v(S_1), v(S_2))\geq\delta^{C_1\eps}$ and if $Q$ is a $\delta^{1/2}$ cube with
$$
Q\cap \bigcup_{\tube\in\tubes(S_1)}Y(\tube)\cap \bigcup_{\tube\in\tubes(S_2)}Y(\tube)\neq\emptyset,
$$
then $Q\cap (S_1\cup S_2)$ is contained in a grain.

\begin{defn}
If $\mathcal{S}$ is a set of regulus strips and if $G$ is a grain, define
$$
\mathcal{S}(G)=\{S\in\mathcal{S}\colon S\cap Q(G)\subset G\}.
$$
\end{defn}

\begin{lem}\label{refinementToSquares}
Let $(\tubes=\bigsqcup_{S\in\mathcal{S}}\tubes(S),Y)$ be an $\eps$--extremal set of tubes of $SL_2$ type, and let $\tubes_{\mathrm{fat}}$ be the associated set of fat tubes. Suppose that the regulus strips in $(\tubes,Y)$ have tangential intersection at angle $\delta^{C_1\eps}$, and that at most one regulus strip $S\in\mathcal{S}$ is contained in each fat tube.

Then there is a set $\mathcal{G}$ of grains with the following properties.

\begin{itemize}
\item $|\mathcal{G}|\gtrapprox_{\eps}\delta^{-3/2}$.
\item Each $\delta^{1/2}$ cube (aligned with the $\delta^{1/2}$ grid) contains at most one grain.
\item For each $G\in\mathcal{G}$, we have $|\mathcal{S}(G)|\gtrapprox_{\eps}\delta^{-1/2}$.
\end{itemize}
\end{lem}

\begin{proof}
For each $S\in\mathcal{S}$, define $Y_1(S) = \bigcup_{\tube\in\tubes(S)}Y(\tube)$. Let $Y_2(S)\subset Y_1(S)$ be a shading with $\sum_{S\in\mathcal{S}}\chi_{Y_2(S)}\sim \mu\chi_A$ for some $\mu\approx_{\eps}\delta^{-1/2}$ and some set $A\subset\RR^3$ with $|A|\approx_{\eps}\delta^{1/2}$ .

Select a strip $S_0\in\mathcal{S}$ with
$$
\sum_{S\in H_{Y_2}(S_0)}|Y_2(S)| \geq\delta^{C\eps-1}.
$$
If the constant $C$ is selected sufficiently large then such a regulus strip $S_0$ must exist. By Lemma \ref{WolffHairbrushBd},
\begin{equation}\label{bigUnionDelta12Nbhd}
\Big|\bigcup_{S\in H_{Y_2}(S_0)}N_{\delta^{1/2}}(Y_2(S))\Big|\gtrapprox_{\eps}1.
\end{equation}
Since each regulus strip in the above union lies in a distinct fat tube, and since $\angle(T_pS, T_p S_0)\lessapprox_{\eps}\delta^{1/2}$ at each point $p\in S\cap S_0$, the sets in the above union have overlap $\lessapprox_{\eps}1$ outside a cylinder of radius $\leq\delta^{C\eps}$ concentric with $S_0$; we will choose $C$ sufficiently large so that for each $S\in H(S_0)$ the intersection of $Y_2(S)$ with the compliment of this cylinder still has volume $\gtrapprox_{\eps}\delta^{3/2}$.

Cover the unit ball by cubes $Q$ of side-length $\delta^{1/2}$ aligned with the $\delta^{1/2}$ grid. For each cube $Q$ that intersects some set of the form $Y_2(S)\backslash N_{\delta^{C\eps}}(S_0)$ for some $S\in H_{Y_2}(S_0)$, let $G$ be a grain containing $S\cap Q$. Let $\mathcal{G}$ be the collection of all grains of this type. By \eqref{bigUnionDelta12Nbhd}, we have that $|\mathcal{G}|\gtrapprox_{\eps}\delta^{-3/2}$. If $G\in\mathcal{G}$ then $|\mathcal{S}(G)|\geq\mu\gtrapprox_{\eps}\delta^{-1/2}$.
%
\end{proof}
\subsection{Tubes and strips in parameter space}\label{tubesStripsSection}
In this section, we will identify the line $(a,b,0)+\RR(c,d,1)\subset\RR^3$ with the point $(a,b,c,d)\in\RR^4.$

\begin{defn}
Let $S$ be a regulus strip. Define $B_S\subset\RR^4$ to be the ball of radius $\delta^{1/2}$ centered at the point $(a,b,c,d)$ corresponding to the line coaxial with the $\delta^{1/2}$ tube containing $S$.
\end{defn}

Let $G$ be a grain. Define
$$
L_G=\{\ell\ \textrm{a line in}\ \RR^3\colon \ell\cap G\neq\emptyset,\ \angle(v(\ell),\Pi(G))\leq\delta^{1/2-C\eps}\}\cap B(0,1).
$$
We will think of $L_G$ as a subset of $\RR^4$ using the identification described above. The set $L_G$ is comparable to a rectangular prism of dimensions $\approx_{\eps}1\times\delta^{1/2}\times\delta^{1/2}\times\delta$. Note that if $G$ is a grain and if $S$ is a regulus strip with $S\cap Q(G)\subset G$, then $B_S\cap L_G\neq\emptyset$.

\begin{lem}
Let $S$ be a regulus strip. Then
$$
\bigcup_{G\colon S\cap Q(G)\subset G}\{\ell\in L_G\colon \angle(v(\ell), v(S))\geq\delta^{\eps}\}
$$
is contained in the $\approx_{\eps}\delta^{1/2}$ neighborhood of a (degree 2) cone $C_S\subset\RR^4$ whose vertex is contained in $B_S$. This cone is contained in the $\approx_{\eps}\delta^{1/2}$--neighborhood of a hyperplane $\Sigma_S\subset \RR^4$.
\end{lem}
\begin{proof}
Let $\tube$ and $\tube^\prime$ be two tubes (not necessarily from $\tubes$) that are contained in $\mathcal{S}$ and that are $\geq\delta^{\eps+1/2}$ separated. Then for each grain $G$ with $S\cap Q(G)\subset G$ and for each line $\ell\in L_G$ with $\angle(v(\ell), v(S))\geq\delta^{\eps}$, we have that $\ell$ intersects the $\lessapprox_{\eps}\delta$ neighborhoods of $\tube$ and $\tube^\prime$.

Let $L$ be the line coaxial with $\tube$ and let $L^\prime$ be the line coaxial with $\tube^\prime$. Applying a rigid transformation, we can assume that $L$ corresponds to $(1,0,0,1)$, and $L^\prime$ corresponds to $(1+a_0, b_0, c_0, 1+d_0)$, where $a_0,b_0,c_0,d_0$ have magnitude $\lessapprox_{\eps}\delta^{1/2}$, and at least one of $a_0,b_0,c_0,d_0$  has magnitude $\gtrapprox_{\eps}\delta^{1/2}$.

Then the set of lines incident to both $\tube$ and $\tube^\prime$ correspond to points $(a,b,c,d)$ satisfying
$$
\left\{
\begin{array}{ll}
|(1-a)(1-d)-bc|\lessapprox_{\eps}\delta,\\
|((1+a_0)-a)((1+d_0)-d)-(b_0-b)(c_0-c)|\lessapprox_{\eps}\delta,
\end{array}
\right.
$$
which can be re-arranged as
$$
\left\{
\begin{array}{ll}
|-d-a+ad-bc|\lessapprox_{\eps}\delta,\\
|-(1+a_0)d-(1+d_0)a +b_0c + c_0 b + ad-bc|\lessapprox_{\eps}\delta,
\end{array}
\right.
$$
and thus
$$
\left\{
\begin{array}{ll}
|(\delta^{-1/2}d_0,\ -\delta^{-1/2}c_0,\ -\delta^{-1/2}b_0, \delta^{-1/2}a_0)\cdot(a,b,c,d)|\lessapprox_{\eps}\delta^{1/2},\\
|(1-a)(1-d)-bc|\lessapprox_{\eps}\delta.
\end{array}
\right.
$$

This set is contained in the $\lessapprox_{\eps}\delta^{1/2}$ neighborhood of a (degree two) cone with vertex $(1,0,0,1)$ that is contained in the $\lessapprox_{\eps}\delta^{1/2}$ neighborhood of a plane. Undoing the linear transformation, we see that the vertex of the cone is contained in $B_S$.
\end{proof}

\subsection{Fourteen well-chosen points: fat tubes lie close to a quadric hypersurface}
\begin{lem}\label{closetoZQ}
Let $\mathcal{S}$ be a set of $\leq\delta^{-3/2-\eps}$ regulus strips, no two of which are contained in a common fat tube, and let $\mathcal{G}$ be a set of $\geq\delta^{-3/2+\eps}$ grains. Suppose that each grain is contained in a distinct $\delta^{1/2}$ cube, and that $|\mathcal{S}(G)|\geq\delta^{-1/2+\eps}$ for each grain $G\in\mathcal{G}$.

Then there is a degree-two polynomial $P$ in four variables and a subset $\mathcal{S}^\prime\subset S$ of size $|\mathcal{S}^\prime|\gtrapprox_{\eps}\delta^{-3/2}$ so that for each $S\in\mathcal{S}^\prime$, there is a point $q\in\CC^4$ with $P(q)=0$ and $\operatorname{dist}(q, B_S)\lessapprox_{\eps}\delta^{1/2}$.
\end{lem}
\begin{proof}
For each $S\in\mathcal{S},$ define the shading
$$
Y(S)=\bigcup_{G\in\mathcal{G}\colon S\cap Q(G)\subset G}G.
$$
Note that $Y(S)$ is not contained in $S$. However, $Y(S)$ is contained in the $\lessapprox_{\eps}\delta$ neighborhood of $S$. Furthermore, $\sum_{S\in\mathcal{S}}|Y(S)|\gtrapprox_{\eps}1$.

Let $S_1,S_2\in\mathcal{S}$ be regulus strips that are $\geq\delta^{C\eps}$ separated and skew, and let $G_0\in\mathcal{G}$ be a grain that is $\geq\delta^{C\eps}$ separated from $S_1$ and $S_2$, and with $\angle(\Pi(G),S_i)\geq\delta^{C\eps},\ i=1,2$, so that $S_1,S_2$, and $G_0$ satisfy
\begin{equation}\label{hairbrushS1S2}
|H_Y(S_1,S_2)|\geq\delta^{C\eps-1/2},
\end{equation}
\begin{equation}\label{hairbrushS1S2LargeVolume}
 \Big|\bigcup_{S\in H_Y(S_1)}Y(S)\ \cap\ \bigcup_{S\in H_Y(S_2)}Y(S)\Big|\geq\delta^{C\eps+1/2},
\end{equation}
and
\begin{equation}\label{manyPointsInJointHairbrushBush}
\Big|\bigcup_{S\in\mathcal{S}(G_0)}Y(S)\ \cap\ \bigcup_{S\in H_Y(S_1)}Y(S)\ \cap\ \bigcup_{S\in H_Y(S_2)}Y(S)\Big|\geq\delta^{C\eps+1}.
\end{equation}

We will show that if $C$ is selected sufficiently large then such a choice of $S_1,S_2,$ and $G_0$ must exist. First, let $Y_1$ be a refinement of the shadings $Y$ so that $\sum_{S\in\mathcal{S}}\chi_{Y_1}(S)\sim\mu\chi_A$ for some $\mu\approx_{\eps}\delta^{-1/2}$ and some set $A\subset\RR^3$ with $|A|\approx_{\eps}\delta^{1/2}$ (the set $A$ will be a union of grains, but we won't need this fact). Select $S_1\in\mathcal{S}$ with $\big|\bigcup_{S\in H_{Y_1}(S_1)}Y_1(S)\big|\gtrapprox_{\eps}\delta^{1/2}$ and for each $S\in\mathcal{S}$, define $Y_2(S)=Y_1(S)\cap\bigcup_{S^\prime\in H_{Y_1}(S)}Y_1(S^\prime)$.

Now select $S_2\in\mathcal{S}$ that is $\gtrapprox_{\eps}1$ skew to $S_1$, with $|Y_2(S_2)|\gtrapprox_{\eps}\delta^{3/2}$ and $|\bigcup_{S\in H_{Y_2}(S_2)}Y_2(S)|\gtrapprox_{\eps}\delta^{1/2}$. The pair $S_1,S_2$ satisfy \eqref{hairbrushS1S2} and \eqref{hairbrushS1S2LargeVolume}. Finally, select $G_0\in\mathcal{G}$ that is $\gtrapprox_{\eps}1$ separated from $S_1$ and $S_2$, with $\angle(\Pi(G),S_i)\gtrapprox_{\eps}1$ for $i=1,2$, and that satisfies \eqref{manyPointsInJointHairbrushBush}.

Next, we will consider the intersection of the (two-dimensional) cones $C_{S_1}$ and $C_{S_2}$. Since these cones are contained in $\RR^4$, in general one might suppose that they would intersect transversely. However, \eqref{hairbrushS1S2} asserts that if the constant $C_1$ is chosen sufficiently large then

$$
|N_{\delta^{1/2-C_1\eps}}(C_{S_1})\cap N_{\delta^{1/2-C_1\eps}}(C_{S_2})\cap B(0,1)|\gtrapprox_{\eps}\delta^{3/2}.
$$
Recall that the above set is a subset of $\RR^4$, so $|\cdot|$ denotes four-dimensional Lebesgue measure; thus $\delta^{3/2}$ is precisely the measure of the $\delta^{1/2}$--neighborhood of a curve in $\RR^4$. In particular this means that the three sets

\begin{align*}
&N_{\delta^{1/2-C_1\eps}}(C_{S_1})\cap N_{\delta^{1/2-C_1\eps}}(C_{S_2})\cap B(0,1),\\
&N_{\delta^{1/2-C_1\eps}}(C_{S_1})\cap N_{\delta^{1/2-C_1\eps}}(\Sigma_{S_2})\cap B(0,1),\\
&N_{\delta^{1/2-C_1\eps}}(C_{S_2})\cap N_{\delta^{1/2-C_1\eps}}(\Sigma_{S_1})\cap B(0,1)
\end{align*}
are all essentially the same, in the sense that the $\lessapprox_{\eps}\delta^{1/2}$ dilate of any of these three sets contains the other two.

We shall now abuse notation slightly and replace the cone $C_{S_2}$ by a slightly different cone $C^\prime_{S_2}$ satisfying $C^\prime_{S_2}\cap B(0,1)\subset N_{\delta^{1/2-C_1\eps}}(C_{S_2})$ and $C_{S_2}\cap B(0,1)\subset N_{\delta^{1/2-C_1\eps}}(C_{S_2}^\prime)$ so that $C_{S_1}\cap C^\prime_{S_2}\subset\RR^4$ is a one-dimensional degree-two curve.

The cone $C^\prime_{S_2}$ is determined as follows: Let $L_1$ and $L_1^\prime$ be lines that are $\delta^{1/2}$ separated (with error $\approx_{\eps}1$) that are coaxial with tubes contained in $S_1$. Let $L_2$  be a line that is coaxial with a tube contained in $S_2$. Then the regulus strips in the hairbrush of $S_1$ and $S_2$ are contained in the $\lessapprox_{\eps}\delta^{1/2}$ neighborhood of the regulus $R$ generated by $L_1,L^\prime,$ and $L_2$. Let $C^\prime_{S_2}$ be the set of lines that are incident to $L_2$ and are tangent to $R$.

Similarly, we will replace the line $L_{G_0}$ by a line $L_{G_0}^\prime$ satisfying $L^\prime_{G_0}\cap B(0,1)\subset N_{\delta^{1/2-C_1\eps}}(L_{G_0})$ and $L_{G_0}\cap B(0,1)\subset N_{\delta^{1/2-C_1\eps}}(L_{G_0}^\prime)$. To do this, select points $q_1\in C_{S_1}\cap N_{\delta^{1/2}}(L_{G_0})$ and $q_2\in C^\prime_{S_2}\cap N_{\delta^{1/2}}(L_{G_0})$. We can select points $q_1,q_2$ with $|q_i|\lessapprox_{\eps}1$, since $\angle(\Pi(G),S_i)\gtrapprox_{\eps}1$. Since $S_1$ and $S_2$ are $\gtrapprox_{\eps}1$ separated and skew, we have that $q_1$ and $q_2$ are $\gtrapprox_{\eps}1$ separated. Let $L_{G_0}^\prime$ be the line passing through $q_1$ and $q_2$.  

We are now ready to select our 14 points.
\begin{itemize}
\item Let $p_1,...,p_5$ be five points on the conic curve $C_{S_1} \cap C^\prime_{S_2}$.
\item Let $p_6$ and $p_7$ be the points of intersection of $L_{G_0}^\prime$ with $C_{S_1}, C^\prime_{S_2}$, respectively.
\item Let $p_8$ be another point on $L_{G_0}^\prime$.
\item Let $p_9$ and $p_{10}$ be the vertices of $C_{S_1}$ and $C^\prime_{S_2}$, respectively.
\item Let $p_{11}$ and $p_{12}$, be two points on $C_{S_1}$.
\item Let $p_{13}$ and $p_{14}$ be two points on $C^\prime_{S_2}.$
\end{itemize}

Let $P$ be a polynomial of degree $\leq 2$ in $\RR[a,b,c,d]$ that vanishes on the 14 points $p_1,...,p_{14}$. Such a polynomial must exist since the vector space of polynomials in four variables of degree $\leq 2$ has dimension 14. Note that there might be more than one such polynomial.

We will show that $C_{S_1}, C^\prime_{S_2},$ and $L_{G_0}^\prime$ are contained in $Z(P)$. Note that $|L_{G_0}^\prime \cap Z(P) | \geq 3$, so $L_{G_0}^\prime \subset Z(P)$. To see that $C_{S_1} \subset Z(P)$, it suffices to show that $C_{S_1} \cap Z(P)$ contains a (possibly reducible) curve of degree at least 5. First, $p_1,..,p_5$ guarantee that the curve $C_{S_1} \cap C^\prime_{S_2} \subset Z(P)$, since $C_{S_1} \cap C^\prime_{S_2}$ is a curve of degree two. Now, since $p_9 \subset Z(P)$, observe that if $p \in C_{S_1}\cap Z(P)$ is a point distinct from $p_9$ and not contained in the curve $C_{S_1} \cap C^\prime_{S_2}$, then the entire line $L(p)$ connecting $p_9$ to $p$ must lie in $Z(P)$. This is because if $\ell_p$ and $\ell_{p_9}$ are the lines in $\RR^3$ corresponding to the points $p$ and $p_9$, respectively, then $L(p)$ corresponds to the set of lines in $\RR^3$ that pass through the point $\ell_p\cap\ell_{p_9}$ and that lie in the plane spanned by these two lines. There exists a line in this set that intersects $\ell_{p_{10}}$ (the line corresponding to the vertex of $C^\prime_{S_2})$. This means that $L(p)$ must intersect $C_{S_1} \cap C^\prime_{S_2}$, and thus $P$ vanishes at the three distinct points $p_9, p,$ and $L(p) \cap C_{S_1} \cap C^\prime_{S_2}$.

Therefore, since $P$ vanishes at the points $C_{S_1}\cap L_{G_0}^\prime,$ $p_{11}$, and $p_{12}$, we have that the (reducible) degree-five curve
$$
\big( C_{S_1} \cap C^\prime_{S_2} \big)\ \cup\ \big(L(C_{S_1} \cap L_{G_0}^\prime) \big)\ \cup\ L(p_{11})\ \cup\ L(p_{12})
$$
is contained in $Z(P)$. This in turn implies that $C_{S_1} \subset Z(P)$. An identical argument shows that $C^\prime_{S_2} \subset Z(P).$

Recall that our goal is to show that most regulus strips from $\mathcal S$ are contained in the $\lessapprox_{\eps}\delta^{1/2}$--neighborhood of $Z(P)$. Since $C_{S_1}, C^\prime_{S_2},$ and $L_{G_0}^\prime$ are contained in $Z(P)$, we know that the strips in $H_Y(S_1), H_Y(S_2)$, and $\mathcal{S}(G_0)$ are contained in the $\lessapprox_{\eps}\delta^{1/2}$--neighborhood of $Z(P)$.

Before continuing, we will need a lemma that controls the behavior of degree-two polynomials in the plane
\begin{lem}[conic sections approximating lines]\label{mostlyCloseToVariety}
Let $P\in\RR[x,y]$ be an irreducible degree-two polynomial and let $w,s,t>0$. Suppose that there are three $s$--separated points $q_1,q_2,q_3\in Z(P)\cap B(0,1)$ so that $q_2$ has distance $\leq w$ from the line $L$ joining $q_1$ and $q_2$. Then there is an interval $I\subset L$ of length at most $t$ so that for all $p\in B(0,1)\cap (L\backslash I)$, there is a point $q\in \CC^2$ with $P(q)=0$ and 
$$
\operatorname{dist}(p,q) \lesssim w/(ts)^{O(1)}.
$$
\end{lem}

\begin{proof}
After applying a rigid transformation and re-labeling the points if necessary, we can assume that $q_1$ and $q_3$ lie on the $x$ axis, $q_2$ lies above (or on) the $x$ axis, and the projection of $q_2$ to the $x$ axis lies between $q_1$ and $q_3$. Apply the linear transformation $(x,y)\mapsto (x/|q_3|, y)$. Let $q_2^\prime$ be the image of $q_2$ under this transformation. Let $f$ be polynomial obtained by pre-composing $P$ with this transformation. Without loss of generality, we can assume that each coefficient of $f$ has magnitude at most one, and that one coefficient of $f$ has magnitude one. It suffices to prove Lemma \ref{mostlyCloseToVariety} for $f$ in place of $P$. Since $f$ vanishes at $(0,0)$ and $(1,0)$, we can write $f(x,y) = ax^2 + bxy + cy^2 - ax + ey$.

Note that $Z(f)$ is smooth, and contains at most two connected components. By reflecting around the line $x=1/2$ if necessary, we can assume that there is a closed curve $\gamma$ (homeomorphic to the closed interval $[0,1])$ with endpoints $(0,0)$ and $q_2^\prime=(x,y)$, with $s\leq x \leq 1-s$ and $0\leq y \leq w$. Since $Z(f)$ can intersect the line $y=w$ at most twice, we conclude that the interior of the curve $\gamma$ lies between the lines $y=0$ and $y=w$.

Implicitly differentiating the equation $f(x,y)=0$, we conclude that 
$$
\frac{dy}{dx} = \frac{2a(x-1)+by}{bx + 2cy + e}.
$$
Thus there is an interval $J\subset\RR$ of length $\gtrsim s$ so that $|\frac{dy}{dx}|\lesssim w/s$ for all $(x,y)\in \gamma$ with $x\in J$. For each such point $(x,y)$, we have
$$
\Big| \frac{2a(x-1)+by}{bx + 2cy + e} \Big|\lesssim w/s,
$$
and since $0<y<w$, we have
\begin{equation}\label{partialBoundOnA}
|a|\ |x-1| \lesssim (w/s)|bx + 2cy + e| + |by| \lesssim (w/s)\big(|b| + |e|\big) + (w^2/s)|c|.
\end{equation}
Since \eqref{partialBoundOnA} holds for all $x\in J$, there exists at least one value of $x$ with $|x-1|\gtrsim s$. We conclude that
$$
|a| \lesssim (w/s^2)\big(|b| + |e|\big) + (w^2/s^2)|c|.
$$
Re-arranging the equation $0=f(x,y) = c y^2 + (bx + e)y + (ax^2 - ax)$ and solving for $y$, we have that one of the solutions is given by
$$
y = \frac{-(bx+e) \pm \sqrt{(bx+e)^2 - 4acx(x-1)}}{2c},
$$
where the sign is chosen so that $-(bx+e)\pm |bx+e|=0$. Note that the discriminant might be negative, so $y$ need not be real. Nonetheless, we will show that $y$ usually has small magnitude. We will consider two cases

Case 1: $|b|+|e|< 100w|c|/t$. Then for all $x\in [-1,1]$, we have
$$
|y| = \Big|\frac{-(bx+e) \pm \sqrt{(bx+e)^2 - 4acx(x-1)}}{2c}\Big|\lesssim \frac{|b|+|e|}{|c|} + \frac{\sqrt{|ac|}}{|c|}\lesssim w/t + \sqrt{|a/c|}\lesssim w/(ts^2).
$$

Case 2:  $|b|+|e|\geq 100w|c|/t$. Then there is an interval $I$ of length $\lesssim t$, so that for all $x\in [-1,1]\backslash I$, we have $|bx + e| > t(|b|+|e|)\geq 100w|c|$. For all such $x$, we have
 $$
|y| = \Big|\frac{-(bx+e) \pm \sqrt{(bx+e)^2 - 4acx(x-1)}}{2c}\Big| \lesssim \Big|\frac{-(bx+e)\pm|bx+e|+ |ac|}{|c|}\Big|\lesssim |a|\lesssim w/s.
$$
\end{proof}

We are now ready to show that most of the remaining regulus strips in $\mathcal{S}$ lie close to $Z(P)$. Let $G\in\mathcal{G}$ be a grain contained in
\begin{equation}\label{setG}
B(0,1)\ \cap\ \bigcup_{S\in H(S_1)}Y(S)\ \cap\ \bigcup_{S\in H(S_2)}Y(S)\ \cap\ \bigcup_{S\in \mathcal{S}(G_0)}Y(S).
\end{equation}
The above set is a union of $\gtrapprox_{\eps}\delta^{-1}$ grains and is contained in the $\lessapprox_{\eps}\delta^{1/2}$ neighborhood of a plane.

Let $L$ be a unit line segment contained in $L_G$. Then $L$ passes $\lessapprox_{\eps}\delta^{1/2}$--close to $Z(P)$ at three $\gtrapprox_{\eps}1$--separated points; call these points $q_1,q_2,q_3$. Let $q_1^\prime,q_2^\prime, q_3^\prime\in Z(P)$ with $\operatorname{dist}(q_i,q_i^\prime)\lessapprox_{\eps}\delta^{1/2}$ for $i=1,2,3.$. Let $L^\prime$ be the line passing through $q_1^\prime$ and $q_3^\prime$, and let $\Pi$ be the plane spanned by $q_1^\prime,q_2^\prime,$ and $q_3^\prime$. Then the restriction of $P$ to $\Pi$ satisfies the hypotheses of Lemma \ref{mostlyCloseToVariety} (with the points $q_1^\prime,q_2^\prime,q_3^\prime$; $w\lessapprox_{\eps}\delta^{1/2}$, $s\gtrapprox_\eps1$, and $t\gtrapprox_{\eps}1$ to be chosen appropriately below). Thus there exists a line segment $L_{\operatorname{bad}}\subset L$ of length $\delta^{C_1\eps}$ so that $(L\backslash L_{\operatorname{bad}})\subset N_{\delta^{1/2-C_2\eps}}(Z(P))$. Here $C_1$ is an absolute constant that we will determine later, while $C_2$ depends on $C_1$. If $C_1$ is chosen sufficiently large, then at most half the regulus strips $S\in\mathcal{S}(G)$ can correspond to balls $B_S\subset \RR^4$ that intersect $L_{\operatorname{bad}}$. Call a strip $S\in\mathcal{S}(G)$ ``good'' if $B_S\subset N_{\delta^{1/2-C_2\eps}}(Z(P)).$ In particular, if $B_S$ intersects $L\backslash L_{\operatorname{bad}}$ then $S$ is good.

Let
$$
\mathcal{S}^\prime=\bigcup_{G\subset \eqref{setG}}\{S\in \mathcal{S}(G)\colon S\ \textrm{is good}\}.
$$

We have $|\mathcal{S}^\prime|\gtrapprox_{\eps}\delta^{-3/2},$ and $B_S\subset N_{\delta^{1/2-C_2\eps}}(Z(P))$ for each $S\in\mathcal{S}^\prime.$

\end{proof}

\subsection{$Z(P)$ is $SL_2(\RR)$}
In this section we will show that after applying a suitable linear transformation to $\RR^3$, the degree-two surface $Z(P)$ from Lemma \ref{closetoZQ} is actually $SL_2 = \{(a,b,c,d)\in\RR^4\colon ad-bc=1\}.$

\begin{lem}\label{formOfP}
Let $(\tubes=\bigsqcup_{S\in\mathcal{S}}\tubes(S),Y)$ be an $\eps$--extremal set of tubes of $SL_2$ type. Let $\mathcal{G}$ be a set of $\geq\delta^{C\eps-3/2}$ grains, each in a distinct $\delta^{1/2}$ cube, and suppose that for each grain $G\in\mathcal{G}$, we have $|\mathcal{S}(G)|\geq \delta^{C\eps-1/2}$.

Suppose furthermore that there exists a monic degree-two polynomial $P$ in four variables so that for each $S\in\mathcal{S}$, there is a point $q\in\CC^4$ with $P(q)=0$ and $\operatorname{dist}(q, B_S)\lessapprox_{\eps}\delta^{1/2}$. Then the polynomial $P$ is of the form
\begin{equation}\label{FormOfP}
P(a,b,c,d)=Aa + Bb + Cc+Dd+ E+ F(ad-bc)+H(a,b,c,d),
\end{equation}
where all of the coefficients of $H$ have magnitude $\lessapprox_{\eps}\delta^{1/2}$.
\end{lem}
\begin{proof}
Let $Q$ be a $\delta^{1/2}$ cube and let $(x,y,z)\in Q$. Then the set of all lines $\{\ell\in\RR^3\colon Q\cap\ell\neq\emptyset\}$ is comparable to the set
$$
\{(a,b,c,d)\in\RR^4\colon a = x+zc+O(\delta^{1/2}),\ b=y+zd+O(\delta^{1/2})\}.
$$

The above set is the $\delta^{1/2}$--neighborhood of a plane $\Pi_{Q}\subset\RR^4$. For each grain $G\in\mathcal{G}$, the line segment $L_G$ is contained in the $\delta^{1/2}$--neighborhood of $\Pi_{Q(G)}$, and is also contained in the $\delta^{C\eps+1/2}$--neighborhood of $Z(P)$.

In particular, the function
$$
f_{G}(c,d)=P(x+zc, y+zd, c, d)
$$
has magnitude $\lessapprox_{\eps}\delta^{1/2}$ on the $\delta^{1/2}$--neighborhood of some line in the $(c,d)$--plane. Since this holds for every grain $G\in\mathcal{G}$, $|\mathcal{G}|\gtrapprox_{\eps}\delta^{-3/2}$, and the grains in $\mathcal{G}$ lie in distinct $\delta^{1/2}$ cubes, we conclude that the function
$$
(c,d)\mapsto P(x+zc, y+zd, c, d)
$$
has magnitude $\lessapprox_{\eps}\delta^{1/2}$ on the $\delta^{1/2}$--neighborhood of some line in the $(c,d)$--plane for $\gtrapprox_{\eps}\delta^{-3/2}$ triples $(x,y,z)\in\RR^3$ that are $\geq\delta^{1/2}$ separated. This implies \eqref{FormOfP}.
\end{proof}
Since the coefficients of $H(a,b,c,d)$ have magnitude $\lessapprox_{\eps}\delta^{1/2}$ and at least one coefficient of $P$ has magnitude one, we have that $|\nabla P|\gtrsim 1$ on $B(0,1)\cap Z(P)$. Thus we can replace the polynomial $Aa + Bb + Cc+Dd+ E+ F(ad-bc)+H(a,b,c,d)$ by the polynomial $P(a,b,c,d)=Aa + Bb + Cc+Dd+ E+ F(ad-bc)$; the $\approx_{\eps}\delta^{1/2}$ neighborhoods of the zero-sets of these two polynomials are comparable. Furthermore, for each $S\in\mathcal{S}$, $B_S\cap N_{\lessapprox_{\eps}\delta^{1/2}}(Z(P))\neq\emptyset$, i.e. for each $S\in\mathcal{S}$, there is a real point $q\in Z(P)$ that lies close to $B_S$.

Each linear transformation $T\colon\RR^3\to\RR^3$ sends lines to lines. Thus each such transformation induces a (not necessarily linear) map on $\RR^4$, which we have identified with the parameter space of lines. The next lemmas show that after applying a suitable transformation to $\RR^3$, the polynomial $P$ from Lemma \ref{formOfP} has a particularly simple form.

\begin{lem}
Let $(\tubes=\bigsqcup_{S\in\mathcal{S}}\tubes(S),Y)$ be an $\eps$--extremal set of tubes of $SL_2$ type. Suppose that there exists a degree-two hypersurface $Z(P)\subset\RR^4$ of the form $P(a,b,c,d)=Aa + Bb + Cc+Dd+ E+ F(ad-bc)$ so that $B_S\subset N_{\delta^{1/2-C\eps}}Z(P)$ for each $S\in\mathcal{S}$.

Then there exists a linear transformation $T\colon\RR^3\to\RR^3$ that distorts angles by $O(1)$ so that after applying this transformation, the image of $Z(P)$ is transformed to $Z(P^\prime)$, where $P^\prime(a,b,c,d)=A^\prime a + B^\prime b + C^\prime c+D^\prime d+ E^\prime+ F^\prime(ad-bc)$, and  $F^\prime\gtrapprox_{\eps}1$.
\end{lem}
\begin{proof}
First, observe that at least one of $A, B, C, D, F$ must have magnitude $\geq \delta^{C\eps}$ for some absolute constant $C$. We will consider several cases.
\medskip

\noindent {\bf Case 1}: $|F|\leq\delta^{C\eps}$; $|A|, |B|\leq2\delta^{C\eps}$. If this occurs, then let $\Pi$ be the span of the set of vectors
$$
\{v = (c,d,1)\colon Cc + Dd + E = 0\}.
$$
If $\tube\in\tubes$, then $v(\tube)$ lies in the $\delta^{C\eps}$ neighborhood of $\Pi$. However, since $(\tubes,Y)$ is $\eps$--extremal, if the constant $C$ is chosen sufficiently large then this contradicts Lemma \ref{allTubesCoplanar}. Thus Case 1 cannot occur.

\medskip

\noindent {\bf Case 2}: $F\leq \delta^{C\eps},\ A\geq 2\delta^{C\eps}$. Apply the linear transformation
$$
(x,y,z)\mapsto (x,y,z+y).
$$
This sends the line $(a,b,0)+\RR(c,d,1)$ to the line
$$
\Big(a-\frac{bc}{1+d},\ b-\frac{bd}{1+d},\ 0\Big) +\RR \Big(\frac{c}{1+d},\ \frac{d}{1+d},\ 1\Big).
$$
Thus the variety $Z\big(Aa + Bb + Cc + Dd + E + F(ad-bc)\big)$ gets mapped to
\begin{equation*}
\begin{split}
&Z\Big(A(a-\frac{bc}{1+d}) + B(b-\frac{bd}{1+d}) + C(\frac{c}{1+d}) + D(\frac{d}{1+d}) + E \\
&\quad+ F\big(\big[a-\frac{bc}{1+d}\big]\big[\frac{d}{1+d}\big]-\big[b-\frac{bd}{1+d}\big]\big[\frac{c}{1+d}\big] \Big)\\
&=Z\Big(\frac{1}{1+d}\big( Aa + Bb + Cc + (D+E)d + E + (A+F)(ad-bc)\big)\Big)\\
&=Z\Big(Aa + Bb + Cc + (D+E)d + E + (A+F)(ad-bc)\Big).
\end{split}
\end{equation*}
We are now in the situation where the coefficient of $ad-bc$ (namely $A+F$) has magnitude $\gtrapprox_{\eps} 1$.

\medskip

\noindent {\bf Case 3}: $F\leq \delta^{C\eps},\ B\geq 2\delta^{C\eps}$. The argument is similar to that in Case 2, except we apply the linear transformation
$$
(x,y,z)\mapsto (x,y,z+x).\qedhere
$$

\end{proof}




\begin{lem}\label{tubesNearStandardSL2}
Let $(\tubes=\bigsqcup_{S\in\mathcal{S}}\tubes(S),Y)$ be an $\eps$--extremal set of tubes of $SL_2$ type. Suppose that there exists a degree-two hypersurface $Z(P)\subset\RR^4$ of the form $P(a,b,c,d)=Aa + Bb + Cc+Dd+ E+ F(ad-bc)$, and $F\gtrapprox_{\eps}1$.

Then there exists a linear transformation $T\colon\RR^3\to\RR^3$ that distorts angles by $\lessapprox_{\eps}1$ so that after applying this transformation, the image of $Z(P)$ is transformed to $Z(ad-bc-1)$.
\end{lem}
\begin{proof}
Define the linear transformation
$$
T_1:(x,y,z)\mapsto \Big(x-(\frac{B}{F})z-\frac{D}{F},\ \  y-(\frac{A}{F})z-\frac{C}{F},\ \  z\Big).
$$

$T_1$ sends $Z(P)$ to $Z(P^\prime)$, with $P^\prime(a, b, c, d)= F(ad-bc)+ E^\prime$, with $E^\prime=2\frac{AD-BC}{F}+E.$ Since $|F|\gtrapprox_{\eps}1$, this transformation distorts angles by $\lessapprox_{\eps}1$. Let $\tubes_1$ be the image of the tubes in $\tubes$ under this transformation.

Every tube in $\tubes_1$ intersects the $\lessapprox_{\eps}|E^\prime/F|$--neighborhood of the $z$ axis. By Corollary \ref{fewTubesHitALine}, we have $|E^\prime/F|\approx_{\eps}1$. Next, consider the transformation $(x,y,z)\mapsto (x \sqrt{E/F},\ y\sqrt{E/F}, z)$. This sends the zero-set of $Z(P^\prime)$ to $Z(ad-bc-1)$. Since $E^\prime$ and $F$ have magnitude $\approx_{\eps}1$, this transformation distorts angles by $\lessapprox_{\eps}1$.
 \end{proof}

\begin{lem}\label{tubesTransverseToSL2}
Let $S$ be a $\delta^{\eps}$ regulus strip with $B_S\subset N_{\delta^{1/2-C\eps}}(SL_2),\ SL_2 = Z(ad-bc-1)$. Then the $\delta$-tubes contained in $S$ correspond to $\delta$-balls in $\RR^4$ that are contained in a rectangular prism of dimensions $\delta^{1/2}\times\delta\times\delta\times\delta$ whose major axis is $\gtrapprox_{\eps}1$--transverse to $SL_2$.
\end{lem}
\begin{proof}
Let $(a,b,c,d)\in SL_2$ be a point so that the line $(a,b,0) + \RR (c,d,1)$ is coaxial with the $\delta^{1/2}$ neighborhood of $S$. Recall that the set of tubes contained in $S$ is comparable to a rectangular prism in $\RR^4$ of dimensions $\approx_\eps \delta^{1/2}\times\delta\times\delta\times\delta$.

Let $\tube$ be a $\delta$--tube contained in $S$. Then $\tube$ is coaxial with a line of the form
$(a+u\delta^{1/2},b+v\delta^{1/2},0)  + \RR (c+w\delta^{1/2},d+x\delta^{1/2},1)$, where $u,v,w,x$ have magnitude $\lessapprox_\eps 1$.

Note that
\begin{equation}\label{determinantCalc}
(a+u\delta^{1/2})(d+x\delta^{1/2})-(b+v\delta^{1/2})(c+w\delta^{1/2}) = 1 + \delta^{1/2}(ax + du - bw - cv) + O(\delta).
\end{equation}

The tubes inside the regulus strip $S$ form a one parameter family of points $(u,v,w,x)\in\RR^4$. Our goal is to show that as the choice of tubes varies, the RHS of \eqref{determinantCalc} varies as well.

The one parameter family of points mentioned above has the property that the intersection of the corresponding tubes with balls of radius $\delta^{1/2}$ centered at the three points $(a,b,0), (a+c/2, b+d/2, 1/2)$, and $(a+c, b+d, 1)$ must each be contained inside a rectangular prism of dimensions $\approx_{\eps}\delta^{1/2}\times\delta^{1/2}\times\delta$ (these rectangular prisms are the ``grains'' from Lemma \ref{refinementToSquares}). Furthermore, the normal directions of these rectangular prisms must agree with the direction dictated by the fact that all of the regulus strips lie near $SL_2$.

These three conditions are linear, and give us the constraints
$$
-bu + av  =  e
$$

$$
-(b+\frac{1}{2}d)u + (a + \frac{1}{2}c)v   -   \frac{1}{2}(b+\frac{1}{2}d)w  + \frac{1}{2}(a+ \frac{1}{2}c)x  =  f,
$$
 and
$$
-(b+d)  u  + (a+c)  v -(b+d) w +(a+c) x  = g,
$$
where  $e,f,$ and  $g$ have magnitude $\lessapprox_{\eps}1$. Next we will verify that the  $4\times 4$ matrix
 $$
\left[\begin{array}{rrrr}
-b&    a&    0&   0\\
-(b+ 1/2 d) &    (a +  1/2 c)&  -(b+1/2 d)    &   (a +1/2 c)\\
-(b+d)&  a+c&    -(b+d)&   a+c\\
d&   -c&    -b&   a
\end{array}\right]
$$
is invertible. Indeed, the determinant is $-{1 \over 2} (ad-bc)^2=-{1 \over 2}.$
Thus the function
$$
(u,v,x,w)\mapsto (a + \delta^{1/2}x)(d+\delta^{1/2}w)- (b+\delta^{1/2}v)(c+\delta^{1/2}x)
$$
has derivative $\approx_\eps\delta^{1/2}$ in the direction of the ruling of $S$. This is exactly the statement that the major axis of the $\delta^{1/2}\times\delta\times\delta\times\delta$ prism corresponding to $S$ makes an angle $\gtrapprox_{\eps}1$ with the tangent plane of $SL_2$ at the point of intersection.
\end{proof}

\begin{lem}\label{tubesNearSL2}
Let $(\tubes=\bigsqcup_{S\in\mathcal{S}}\tubes(S),Y)$ be an $\eps$--extremal set of tubes of $SL_2$ type. Suppose that for each $S\in\mathcal{S}$, $B_S\subset N_{\delta^{1/2-C\eps}}(SL_2)$.

Then to each $S\in\mathcal{S}$ we can associate a tube $\tube_S$ (not necessarily from $\tubes(S)$) with the following properties.
\begin{itemize}
\item $\tube_S\subset S$; in particular the tubes $\{T_S\colon S\in\mathcal{S}\}$ are $\delta^{1/2}$--separated in parameter space; i.e.~no tube is contained in the $\delta^{1/2}$ neighborhood of any other tube.
\item Each $\delta$--tube $\tube_S$ corresponds to a $\delta$--ball in $\RR^4$ that intersects $SL_2$.
\item If $S,S^\prime\in\mathcal{S}(G)$ for some grain $G$ and if $\angle(v(S),v(S^\prime))\geq\delta^{C\eps}$, then the $\lessapprox_{\eps}\delta$ neighborhoods of $\tube_S$ and $\tube_{S^\prime}$ intersect. In particular, if we replace each tube $\tube_S$ by its $\lessapprox_{\eps}\delta$ neighborhood, then there are $\gtrapprox_{\eps}\delta^{-5/2}$ tube-tube intersections that make an angle $\gtrapprox_{\eps}1$.
\end{itemize}
\end{lem}
\begin{proof}
By Lemma \ref{tubesTransverseToSL2}, for each $S\in\mathcal{S}$ there is a tube contained in $S$ (not necessarily from $\tubes(S)$, however) whose corresponding $\delta$--ball in $\RR^4$ intersects $SL_2$. Call this tube $\tube_S$.

If $S$ and $S^\prime$ make an angle $\geq\delta^{C\eps}$ and share a common grain, Then every tube contained in $S$ must intersect the $\lessapprox_{\eps}\delta$ neighborhood of every tube contained in $S^\prime$ (regardless of whether these tubes are actually present in $\tubes(S)$ and $\tubes(S^\prime)$, respectively).
\end{proof}

\subsection{Using $SL_2$ structure to find an impossible incidence arrangement}
In this section we will use the set of tubes $\{\tube_S\}_{S\in\mathcal{S}}$ from Lemma \ref{tubesNearSL2} to construct an arrangement of points and lines in the plane that determine many incidences.

\begin{lem}\label{tangentPlaneOrth}
Let $p_1$ and $p_2$ be points in $SL_2\cap B(0,1)$. Suppose that $p_1$ and $p_2$ correspond to lines in $\RR^3$ that are $\geq\delta^{\eps}$ separated and skew. Then the line connecting $p_1$ and $p_2$ makes an angle $\gtrapprox_\eps 1$ with the tangent plane of $SL_2$ at the point $p_1$.
\end{lem}
\begin{proof}
Let $p_i=(a_i,b_i,c_i,d_i),\ i=1,2.$ Observe that $T_{p_1}(SL_2)$ is orthogonal to the vector $(d_1, -c_1, -b_1, a_1)$. Thus we need to show that
$$
|(d_1, -c_1, -b_1, a_1)\cdot(a_1-a_2, b_1-b_2, c_1-c_2, d_1-d_2)|\gtrapprox_\eps 1.
$$
But this is equal to
\begin{equation}\label{vectorDot}
\begin{split}
&|2(a_1d_1 - b_1c_1) - (d_1a_2 -c_1b_2 -b_1c_2 + a_1d_2)|\\
&=2- (d_1a_2 -c_1b_2 -b_1c_2 + a_1d_2)\\
&=\left|\begin{array}{cc}a_1-a_2 & b_1-b_2 \\ c_1-c_2 & d_1-d_2\end{array} \right|.
\end{split}
\end{equation}
As observed in \eqref{XijIsOne}, since $p_1$ and $p_2$ correspond to lines in $\RR^3$ that are $\gtrapprox_{\eps}1$ separated and skew, the above determinant is $\gtrapprox_{\eps}1$.
\end{proof}

\begin{lem}\label{notCollinear}
Let $p_1,p_2,p_3\in SL_2\cap B(0,2)$. Suppose that $p_1$ and $p_2$ correspond to points that are $\geq\delta^\eps$ separated and skew. Let $L_{p_1,p_2}$ be the line connecting $p_1$ and $p_2$. Then
\begin{equation}\label{distLP2}
\dist(L, p_3)\approx_{\eps} \dist(p_1,p_3).
\end{equation}
\end{lem}
\begin{proof}
Since $SL_2$ has degree 2 and $L$ is a line, we have that $L$ intersects $SL_2$ at precisely the points $p_1$ and $p_2$. By Lemma \ref{tangentPlaneOrth}, $\angle(v(L), T_{p_i}(SL_2))\gtrapprox_\eps 1,\ i=1,2.$ \eqref{distLP2} now follows from compactness.
\end{proof}
\begin{lem}\label{tripleHairbrush}
Let $p_i = (a_i,b_i,c_i,d_i),\ i=1,2,3$ be three points in $SL_2\cap B(0,1)\subset \RR^4$. Suppose that the lines corresponding to $p_3$ and each of $p_1$ and $p_2$ are $\geq\delta^\eps$ separated and skew, and that
\begin{align}
\dist\big( p_1,\ p_2\big)&\geq \delta^{1/2}\label{distanceBetweenPoints12}.
\end{align}
Then the set of points $(a,b,c,d)\in\RR^4$ satisfying
\begin{equation}\label{hyperplaneAndSL2Intersection}
\left\{
\begin{array}{l}|(a_i,b_i,c_i,d_i)\cdot (a,b,c,d)- 1|\leq  \delta^{1-C\eps},\ i=1,2,3,\\
|ad-bc - 1|\leq \delta^{1-C\eps}
\end{array}
\right.
\end{equation}
is contained in a union of two balls of radius $\lessapprox_{\eps}\delta^{1/2}$.
\end{lem}
\begin{proof}
First, note that if $H\subset\RR^4$ is a hyperplane passing through the origin, then at each point $p\in H\cap SL_2\cap B(0,1),$ the tangent planes $T_pH$ and $T_pSL_2$ make an angle $\gtrsim 1$ (i.e.~$H$ and $SL_2$ intersect $O(1)$--transversely).  Next, note that \eqref{hyperplaneAndSL2Intersection} can be written
\begin{align}
& |(a_1,b_1,c_1,d_1)\cdot(a,b,c,d)-1|\lessapprox_{\eps}\delta,\label{1stHyperplane}\\
& |(a_3,b_3,c_3,d_3)\cdot(a,b,c,d) -1|\lessapprox_{\eps}\delta,\label{3rdHyperplane}\\
& |\delta^{-1/2}(a_1-a_2, b_1-b_2,c_1-c_2,d_1-d_2)\cdot(a,b,c,d)|\lessapprox_{\eps}\delta^{1/2},\label{inHyperplaneThroughOrigin}\\
&|ad-bc - 1|\lessapprox_{\eps}\delta.\label{liesInSl2}
\end{align}

If $p=(a,b,c,d)$ satisfies \eqref{1stHyperplane} and \eqref{3rdHyperplane}, then $p$ lies in the $\approx_{\eps}\delta$ neighborhood of a plane in $\RR^4$. If $p$ satisfies \eqref{inHyperplaneThroughOrigin} and \eqref{liesInSl2}, then $p$ lies in the ($\gtrapprox_{\eps} 1$) transverse intersection of the $\approx_{\eps}\delta$--neighborhood of a non-degenerate quadric hypersurface and the $\approx_{\eps}\delta^{1/2}$--neighborhood of a hyperplane. This hyperplane is $\gtrapprox_{\eps} 1$ transverse to the plane given by \eqref{1stHyperplane} and \eqref{3rdHyperplane}. This is because by Lemma \ref{notCollinear}, the points $p_1,\ p_3$, and $\delta^{-1/2}(a_1-a_2,b_1-b_2,c_1-c_2, d_1-d_2)$ are $\gtrapprox_\eps 1$--far from being collinear, so the corresponding hyperplanes intersect $\gtrapprox_{\eps} 1$--transversely.
\end{proof}

\begin{cor}\label{tubesBehaveLineLines}
Let $\tubes$ be a set of $\delta$ tubes whose corresponding $\delta$-balls in parameter space are $\delta^{1/2}$ separated and intersect $SL_2$. Let $\tube_A,\tube_B\in\tubes$ be $\gtrapprox_{\eps}1$ separated and skew. If $\tube_1,\tube_2\in H(\tube_A)$, then $|H(\tube_1,\tube_2,\tube_B)|\lessapprox_{\eps}1$.
\end{cor}
\begin{proof}
Let $L_1$ and $L_2$ be the lines coaxial with $\tube_1$ and $\tube_2$, respectively. Since the tubes in $\tubes$ are $\gtrsim\delta^{1/2}$--separated, the lines $L_1$ and $L_2$ satisfy \eqref{distanceBetweenPoints12}. Since $\tube_A$ and $\tube_B$ are $\gtrapprox 1$ skew, we can apply Lemma \ref{tripleHairbrush} and conclude that the tubes in $H(\tube_1,\tube_2,\tube_B)$ are contained in the union of two balls of radius $\lessapprox\delta^{1/2}$. But since the tubes in $\tubes$ are $\delta^{1/2}$ separated in parameter space, this implies that  $|H(\tube_1,\tube_2,\tube_B)|\lessapprox_{\eps}1$.
\end{proof}

\begin{lem}\label{reductionToPTCurve}
Let $\tubes$ be the set of  $\delta$--tubes whose corresponding $\delta$-balls in parameter space are $\delta^{1/2}$ separated and intersect $SL_2\cap B(0,1)$ (in particular, this means $|\tubes|\lesssim\delta^{-3/2}$). Suppose that
$$
|\{(\tube_1,\tube_2)\in\tubes^2\colon \angle\big(v(\tube_1), v(\tube_2)\big)\geq\delta^{C\eps},\ N_{\delta^{1-C\eps}}(\tube_1)\cap N_{\delta^{1-C\eps}}(\tube_2)\neq\emptyset\}|\geq\delta^{-5/2+C\eps}.
$$

Then there exists a set of $\delta^{1/2}$ separated points $\pts\subset B(0,1)\subset\RR^2$ with $|\pts|\leq\delta^{-1}$ and a set $\mathcal{C}$ of algebraic curves of degree $\leq 2$ with $|\mathcal{C}|\leq\delta^{-1}$. The pair $(\pts,\mathcal{C})$ has the following properties:
\begin{itemize}
\item The points and curves have two degrees of freedom: If $p,q\in\pts$ are distinct points, then
$$
|\{\gamma\in\mathcal{C}\colon \dist(p,\gamma)\leq\delta^{1-C\eps},\ \dist(q,\gamma)\leq\delta^{1-C\eps}\}|\lessapprox_{\eps} 1.
$$
Note that the above bound does not depend on $\dist(p,q)$!
\item The points and curves determine many incidences:
$$
|\{(p,\gamma)\in\pts\times\mathcal{C}\colon \dist(p,\gamma)\leq\delta^{1-C\eps}\}|\gtrapprox_\eps\delta^{-3/2}.
$$
\end{itemize}
\end{lem}
\begin{proof}
First, we will abuse notation slightly and replace each tube in $\tubes$ by its $\delta^{1-C\eps}$ neighborhood. Select $\tube_A,\tube_B\in\tubes$ that are $\gtrapprox_{\eps}1$ separated and skew, and with
$$
|\{(\tube,\tube^\prime)\in H(\tube_A)\times H(\tube_B)\colon \tube\cap \tube^\prime\neq\emptyset\}|\gtrapprox_\eps\delta^{-3/2}.
$$
For each $\tube\in H(\tube_A)$, let $\beta(\tube)$ be the set of points $(a,b,c,d)\in\RR^4$ satisfying
\begin{equation*}
\begin{split}
&(a_1,-b_1,-c_1,d_1)\cdot(a,b,c,d)=2,\\
&(a_2,-b_2,-c_2,d_2)\cdot(a,b,c,d)=2,\\
&ad-bc = 1,
\end{split}
\end{equation*}
where $(a_1,b_1,c_1,d_1)$ is the point corresponding to $\tube_B$ and $(a_2,b_2,c_2,d_2)$ is the point corresponding to $\tube$. Observe that if $(a,b,c,d)\in\beta(\tube)$, then
$$
(a,b,c,d)\in  SL_2,\quad \Big|\begin{array}{cc} a-a_1&b-b_1\\ c-c_1&d-d_1 \end{array}\Big|=0,\quad \Big|\begin{array}{cc} a-a_2&b-b_2\\ c-c_2&d-d_2 \end{array}\Big|=0.
$$

This means that $\beta(\tube)\subset\RR^4$ corresponds to the set of lines in $SL_2$ that intersect the lines coaxial with $\tube$ and $\tube_B$, so $N_{\delta^{1-C\eps}}(\beta(\tube))$ corresponds to the set of lines in the $\delta^{1-C\eps}$--neighborhood of $SL_2$ that intersect the $\delta^{1-C\eps}$ neighborhoods of $\tube$ and $\tube_B$. By Corollary \ref{tubesBehaveLineLines}, if $\tube_1,\tube_2\in H(\tube_A)$ are distinct, then $\beta(\tube_1)\cap\beta(\tube_2)$ contains the points corresponding to $\lessapprox_\eps 1$ tubes from $H(\tube_B)$.

We can assume that at least a $\gtrsim 1$ fraction of the tubes from $H(\tube_B)$ are contained in a set of the form $N_{\delta}(W)$, where
$$
W\subset \{(a,b,c,d)\in SL_2\colon (a_1,-b_1,-c_1,d_1)\cdot(a,b,c,d)=2\}
$$
is a surface patch with diameter $\leq 1/100$. Let $\pi\colon\RR^4\to\RR^2$ be a projection that sends $W\to\pi(W)\subset\RR^2$ with bi-Lipschitz constant $O(1)$. For each $\tube\in\tubes$, define $\gamma(\tube)$ to be the Zariski closure of $\pi(\beta(\tube)\cap W)$. This is an algebraic curve of degree $\leq 2$. Define
$$
\mathcal{C}=\{\gamma(\tube)\colon \tube\in H(\tube_A)\},
$$
and define
$$
\pts = \{\pi(p(\tube))\colon \tube\in H(\tube_B),\ p(\tube)\in W \},
$$
where $p(\tube)\in\RR^4$ is the point in $SL_2\subset \RR^4$ corresponding to the line coaxial with $\tube$.
\end{proof}
\subsection{A planar incidence bound}
In this section, we will show that the point-curve configuration constructed in the previous section cannot exist. This will conclude the proof of Proposition \ref{killingSl2Prop}. We will use the discrete polynomial partitioning theorem from \cite{GK}. Though the theorem works in any dimension, we will only state the planar version here
\begin{thm}[Discrete polynomial partitioning; \cite{GK}, Theorem 4.1]\label{partitioningThm}
Let $\pts\subset\RR^2$ be a set of $n$ points. Then for each $D\geq 1$, there is a polynomial $P(x,y)$ of degree at most $D$ so that $\RR^2\backslash Z(P)$ is a union of $O(D^2)$ open connected sets (called cells), and each of these cells contains $O(n/D^2)$ points from $\pts.$
\end{thm}

In addition to the points contained in each cell, some of the points may lie on set $Z(P)$, which is called the ``boundary'' of the partition. The following theorem will help us deal with these points.

\begin{thm}[Wongkew, \cite{Won}]
Let $V\subset\RR^d$ be a real algebraic variety of codimension $m$ whose defining polynomials have degree $\leq D$. Then there exist constants $\{C_j\}$ depending only on $d$ so that for each $\rho>0$,
$$
|N_{\rho}(V)\cap B(0,1)|\leq\sum_{j=m}^d c_j D^j\rho^j,
$$
where $|\cdot|$ denotes $d$-dimensional Lebesgue measure.
\end{thm}
\begin{thm}[Harnack]
Let $P(x,y)$ be a bivariate polynomial of degree $\leq D$. Then $Z(P)$ has $O(D^2)$ connected components.
\end{thm}

\begin{cor}\label{WonkewCor}
Let $P(x,y)$ be a bivariate polynomial of degree $\leq D$. Then for each $\rho>0$,
$$
\mathcal{E}_\rho\big(Z(P)\cap B(0,1)\big)\lesssim D\rho^{-1}+D^2.
$$
\end{cor}

\begin{lem}\label{discreteST}
Let $\pts\subset B(0,1)\subset \RR^2$ be a set of $\leq\delta^{-1}$ points that are $\delta^{1/2}$ separated. Let $\mathcal{C}$ be a set of $\leq\delta^{-1}$ algebraic curves of degree $\leq C_0$. Suppose that for any two points $p,q\in\pts$, there are $\leq A$ curves $\gamma\in\mathcal{C}$ with $\dist(p,\gamma)\leq\delta^{1-C\eps},\ d(q,\gamma)\leq\delta^{1-C\eps}$. Then
$$
|\{(p,\gamma)\in\pts\times\gamma\colon \dist(p,\gamma)\leq\delta^{1-C\eps}\}|\lessapprox_{\eps} A^{1/2}\delta^{-4/3}.
$$
\end{lem}
\begin{proof}
Apply Theorem \ref{partitioningThm} with $D=\delta^{-1/6}$ to the set $\pts$, and let $P$ be the resulting partitioning polynomial. Then $\lesssim |\pts|D^{-2}\leq \delta^{-2/3}$ points from $\pts$ are contained in each connected component of $\RR^2\backslash Z(P)$. Define $\pts_0=\pts\cap N_{\delta^{1-C\eps}}(Z(P))$, and for each connected component $\Omega$ of $\RR^2\backslash Z(P)$, define $\pts_{\Omega}=\pts\cap\Omega\backslash\pts_0$. By Corollary \ref{WonkewCor} with $\rho=\delta^{1-C\eps}$, we have $|\pts_0|\lessapprox_{\eps} D^2\delta^{-1/2}\lesssim \delta^{-5/6},$ and $|\pts_i|\lesssim\delta^{-2/3}$ for each index $i\geq 1$.

Let
$$
\mathcal{C}_\Omega=\{\gamma\in\mathcal{C}\colon\gamma\cap\Omega\neq\emptyset\}.
$$
Then
$$
\sum_{\Omega}|\mathcal{C}_{\Omega}|\lesssim D\delta^{-1}.
$$
If $\mathcal{P}^\prime$ is a set of points and $\mathcal{C}^\prime$ is a set of curves, define
$$
I(\mathcal{P}^\prime,\ \mathcal{C}^\prime)=|\{(p,\gamma)\in\mathcal{P}^\prime\times\mathcal{C}^\prime\colon d(p,\gamma)\leq\delta^{1-C\eps}\}|.
$$
We have
\begin{equation}
\begin{split}
I(\pts,\mathcal{C})&=I(\pts_0,\mathcal{C})+\sum_{\Omega}I(\pts_\Omega,\mathcal{C}_\Omega)\\
&\lesssim A^{1/2}|\pts_0|\ |\mathcal{C}|^{1/2} + \sum_{\Omega}A^{1/2}|\pts_\Omega|\ |\mathcal{C}_\Omega|^{1/2}\\
&\lessapprox_{\eps} A^{1/2}\delta^{-4/3} + A^{1/2}\Big(\sum_{\Omega}|\pts_\Omega|^2\Big)^{1/2}\Big(|\mathcal{C}_\Omega|\Big)^{1/2}\\
&\lesssim A^{1/2}\delta^{-4/3} + A^{1/2}(\delta^{-1}D^{-2}) (\delta^{-1/2})\\
&\lesssim A^{1/2}\delta^{-4/3}.\qedhere
\end{split}
\end{equation}
\end{proof}
We are now ready to prove Proposition \ref{killingSl2Prop}. For the reader's convenience, we re-state it here.
\begin{killingSl2PropEnv}
There are positive constants $\eps_0,\delta_0$ so that if $0<\eps\leq\eps_0$ and $0<\delta\leq\delta_0$, then there cannot exist an $\eps$--extremal collection of $\delta$ tubes of $SL_2$ type.
\end{killingSl2PropEnv}
\begin{proof}
Suppose there existed an $\eps$-extremal collection of tubes $(\tubes,Y)$ of $SL_2$ type. Apply Lemma \ref{reductionToPTCurve} to $(\tubes,Y)$. We obtain a set $(\pts,\mathcal{C})$ that contradicts Lemma \ref{discreteST}.
\end{proof}

\section{Wrap-up}
\begin{proof}[Proof of Theorem \ref{mainThm}]
Let $C$ be the constant from the statement of Proposition \ref{killingHeisenbergProp}. Let $\alpha_0,\eps_0,\delta_0$ be the output of  Proposition \ref{killingSl2Prop} for this value of $C$. Apply Proposition \ref{killingHeisenbergProp} with this value of $\alpha$, and let $c_1,\delta_1$ be the corresponding output. Let $0<\delta\leq\min(\delta_0,\delta_1)$ and $0<\eps\leq \min(\eps_0, c_1\alpha_0)$.

Suppose that there exists an $\eps$--extremal set of $\delta$ tubes; call this set $(\tubes,Y)$. Apply Lemma \ref{HeisenbergPlusSl2Lem} to $(\tubes,Y),$ and let $\tubes_1,\tubes_2$ be the resulting sets of tubes.

By Proposition \ref{killingHeisenbergProp}, $(\tubes_1,Y)$ cannot be $\eps$--extremal. By Proposition \ref{killingSl2Prop}, $(\tubes_2,Y)$ cannot be $\eps$--extremal. But by Remark \ref{ifExtremalOneOfT1T2Extremal}, at least one of $\tubes_1$ or $\tubes_2$ must be $\eps$--extremal. This is a contradiction. We conclude that if $0<\delta\leq\min(\delta_0,\delta_1)$ and $0<\eps\leq\min(\eps_0,c_1\alpha_0)$, then an $\eps$--extremal collection of $\delta$--tubes cannot exist.

To phrase this result in the language of Theorem \ref{mainThm}, let $c=\delta_0^3$ and let $C=3/\eps$. If $\delta\geq\delta_0$ or if $\lambda\leq\delta^{\eps}$, then   \eqref{volumeBdEqn} is just the trivial assertion that
$$
\Big|\bigcup_{\tube\in\tubes}Y(\tube)\Big|\geq \delta^3.
$$
On the other hand, if $\delta\leq\delta_0$, $\lambda\geq\delta^{\eps}$, and if inequality \eqref{volumeBdEqn} failed, then this would mean there exists a $\eps$--extremal collection of $\delta$--tubes. As we have shown, such a collection cannot exist.
\end{proof}

\appendix

\section{The $SL_2$ example}\label{SL2Appendix}
In this appendix we will describe some of the properties of the $SL_2$ example. In $\RR^3$, the example would have the following (hypothetical) form. Let $\mathcal{L}_{\operatorname{coarse}}\subset B(0,1)\subset\RR^4$ be a set of $\delta^{-3/2}$ points $(a,b,c,d)\in\RR^4$ that are $\delta^{1/2}$-separated and satisfy $ad-bc=1$. Let
$$
\tubes_{\operatorname{fat}}=\{B(0,1)\cap N_{\delta^{1/2}}(L)\colon L\in\mathcal{L}_{\operatorname{coarse}}\}.
$$
This is a set of $\delta^{-3/2}$ ``fat tubes'' of length $\sim 1$ and thickness $\sim\delta^{1/2}$, each of which is contained in the unit ball in $\RR^3$. For each point $p\in B(0,1)\subset\RR^3$, there are roughly $\delta^{-1/2}$ fat tubes from $\tubes_{\delta^{1/2}}$ passing through $p$, and there is a plane $\Pi_p$ containing $p$ so that each of the fat tubes containing $p$ makes an angle $\lesssim \delta^{1/2}$ with $\Pi_p$.

We will place $\delta^{-1/2}$  $\delta$-tubes inside each of the fat tubes. The $\delta$-tubes will have the property that whenever two fat tubes intersect, all of the $\delta$-tubes contained inside them also intersect. The $\delta$-tubes inside each fat tube are contained in a regulus strip $S$. This strip has the property that for each $p\in S$, $\angle(T_pS,\ \Pi_p)\lesssim\delta^{1/2}$. Thus the tangent space of the regulus strip always ``agrees'' with the direction of the plane $\Pi_p$ that is specified by the $SL_2$ structure of the fat tubes. In particular, this means that whenever two regulus strips (each contained in their respective fat tube) intersect, they must be tangent, and thus the intersection has large volume.

At scale $\delta$, the union of tubes looks rather empty---the union of $\delta$ tubes has volume roughly $\delta^{1/2}$. At scale $\delta^{1/2}$. however, the union of tubes looks large: there are about $\delta^{-1/2}$ fat tubes passing through each point in the unit ball.

If we cover the unit ball by $\sim\delta^{-3/2}$ cubes of side-length $\delta^{1/2}$, then the union of the $\delta$--tubes will intersect each of these cubes in $O(1)$ rectangular prisms of dimensions $\delta^{1/2}\times\delta^{1/2}\times\delta$ (such rectangular prisms are called grains).

\subsection{Constructing the $SL_2$ example}
While Theorem \ref{mainThm} asserts that the $SL_2$ example cannot exist in $\RR^3$, it is possible to construct an analogue of the $SL_2$ example in a slightly different setting.

Let $R=\FP[t]/(t^2)$. Note that $R$ is not a field; not only do certain elements of $R$ fail to have multiplicative inverses, $R$ also also contains nilpotent elements. Thus the geometry of $R^3$ differs dramatically from that of $\RR^3$. Nonetheless, it is possible to define an analogue of points, lines, and planes in $R^3$ that share some of the properties of their counterparts in $\RR^3$.

We will write points $(x,y,z)\in R^3$ as $(x_1+tx_2, y_1+ty_2, z_1+tz_2)$, with $x_1,x_2,y_1,y_2,z_1,z_2\in\FP$. We define a plane in $R^3$ to be a set of the form
\begin{equation}\label{defnOfPlane}
\{(x,y,z)\in R^3\colon (u,v,w)\cdot(x,y,z) = s\}
\end{equation}
that contains $|R|^2 = p^4$ elements. Note that not every set of the form \eqref{defnOfPlane} contains $p^4$ elements. For example, if $s,u,v,w \in t\FP$, then the set \eqref{defnOfPlane} will contain $p^5$ elements.

We define a line in $R^3$ to be a set of the form
\begin{equation}\label{defnOfLine}
\{(a,b,0)+s(c,d,1)\colon s\in R\}
\end{equation}
(to be slightly pedantic, we should also include lines whose directions lie in the $xy$ plane but this will not affect our arguments). We will identify the line \eqref{defnOfLine} with the point $(a,b,c,d)\in R^4$.

Define
$$
\mathcal{L}=\{(a+\alpha at, b+\alpha bt, c+\alpha ct, d + \alpha dt)\in R^4\colon a,b,c,d,\alpha\in \FP, ad-bc=1\}.
$$
This set of lines satisfies the following analogue of the Wolff axioms:
\begin{itemize}
\item $|\mathcal{L}|=|R|^2$.
\item At most $|R|$ lines from $\mathcal{L}$ lie in any plane.
\item Each pair of lines from $\mathcal{L}$ intersect in at most one point.
\item If two lines intersect, then they are contained in (at least one) common plane.
\item The intersection of two planes contains at most one line from $\mathcal{L}$.
\item every triples of pairwise intersecting, non-concurrent lines from $\mathcal{L}$ are contained in at least (and thus exactly) one common plane.
\end{itemize}

Wolff's hairbrush argument from \cite{wolff} shows that the union of any set of lines satisfying the above axioms must the above size $\gtrsim|R|^{5/2}$. On the other hand, every line in $\mathcal{L}$ is contained in
$$
X = \{(x_1+x_2t, y_1+y_2t, z_1+z_2t)\in R^3\colon z_2 = x_1y_2 - x_2y_1\},
$$
which has size $|R|^{5/2}$.

\section{Acknowledgments}
The authors would like to thank Terry Tao for helpful comments and suggestions to an earlier version of this manuscript. The authors would also like to thank the anonymous referee for comments and corrections to an earlier version of this manuscript.

\end{document}